\documentclass[12pt,twoside,a4paper]{amsart}
\usepackage[centertags]{amsmath}
\usepackage{amssymb,amscd,amsfonts,latexsym,a4wide,amsthm}

\advance\headheight by 3pt

\renewcommand{\subjclass}[1]{\thanks{\emph{2000 Mathematics Subject Classification:}~#1}}
\renewcommand{\keywords}[1]{\thanks{\emph{Keywords and Phrases:}~#1}}

\newtheorem{thm}{Theorem}[section]
\newtheorem{cor}[thm]{Corollary}
\newtheorem{lemma}[thm]{Lemma}
\newtheorem{prop}[thm]{Proposition}

\theoremstyle{definition}

\newtheorem{prg}{}[section]

\theoremstyle{remark}

\numberwithin{equation}{section}

\newcommand{\Cc}{{\mathbb C}}

\newcommand{\Zz}{{\mathbb Z}}
\newcommand{\Qq}{{\mathbb Q}}

\newcommand{\Rr}{{\mathbb R}}

\newcommand{\Xx}{{\mathbb X}}

\newcommand{\OQq}{\overline{{\mathbb Q}}}

\newcommand{\LL}{\mbox{$\mathcal{L}$}}

\newcommand{\x}{{\bf x}}

\newcommand{\cc}{{\bf c}}
\newcommand{\dd}{{\bf d}}
\newcommand{\hh}{{\bf h}}
\newcommand{\rr}{{\bf r}}
\newcommand{\ii}{{\bf i}}
\newcommand{\jj}{{\bf j}}
\newcommand{\kk}{{\bf k}}
\newcommand{\XX}{{\bf X}}

\newcommand{\Ur}{\mathcal{U}({\bf r})}
\newcommand{\Uri}{\mathcal{U}({\bf r},{\bf i})}
\newcommand{\Uw}{\mathcal{U}_w}

\newcommand{\del}{\partial}

\newcommand{\eps}{\varepsilon}

\newcommand{\half}{\textstyle{\frac{1}{2}}}

\renewcommand{\leq}{\leqslant}
\renewcommand{\geq}{\geqslant}

\newcommand{\kdots}{,\ldots ,}

\renewcommand{\span}{{\rm span}\,}

\newcommand{\HLcQ}{H_{\mathcal{L},{\bf c},Q}}

\newcommand{\DL}{\Delta_{\mathcal{L}}}
\newcommand{\HL}{H_{\mathcal{L}}}

\newcommand{\lin}{{\rm lin}}

\newcommand{\propersubset}
{\,\mbox{{\raisebox{-1ex}
{$\stackrel{\textstyle{\subset}}{\scriptscriptstyle{\not=}}$}}}\,}

\title[Improvement of the Quantitative Subspace Theorem]
{A further improvement of the Quantitative Subspace Theorem}

\subjclass{11J68, 11J25}
\keywords{Diophantine approximation, Subspace Theorem}

\author[J.-H.~EVERTSE]{Jan-Hendrik~EVERTSE}
\author[R.~G.~FERRETTI]{Roberto~G.~FERRETTI}
\address{J.-H. Evertse,
Universiteit Leiden, Mathematisch Instituut,
Postbus 9512, 2300 RA Leiden,
The Netherlands}
\email{evertse@math.leidenuniv.nl}
\address{R.G. Ferretti,
Universit\`a della Svizzera Italiana,
Via Buffi 23, CH-6900 Lugano, Switzerland}
\email{roberto.ferretti@lu.unisi.ch}
\date{\today}

\begin{document}

\maketitle

\begin{abstract}
In 2002, Evertse and Schlickewei \cite{Evertse-Schlickewei2002}
obtained a quantitative
version of the so-called Absolute Parametric Subspace Theorem.
This result deals with a parametrized class of twisted heights.
One of the consequences of this result
is a quantitative version
of the Absolute Subspace Theorem, giving an explicit upper bound
for the number of subspaces containing the solutions of the
Diophantine inequality under consideration.

In the present paper, we further improve Evertse's and Schlickewei's
quantitative version of the Absolute Parametric Subspace Theorem,
and deduce an improved quantitative version of the Absolute Subspace Theorem.
We combine ideas from the proof of Evertse and Schlickewei 
(which is basically a substantial
refinement of Schmidt's proof of his Subspace Theorem from 1972
\cite{Schmidt1972}), with ideas from Faltings' and W\"{u}stholz' proof
of the Subspace Theorem \cite{Faltings-Wuestholz1994}.
\end{abstract}

\section{Introduction}\label{1}

\begin{prg}
Let $K$ be an algebraic number field.
Denote by $M_K$ its set of places and by $\|\cdot\|_v$ ($v\in M_K$)
its normalized absolute values, i.e., if $v$
lies above $p\in M_{\Qq}:=\{\infty\}\cup\{ {\rm prime\ numbers}\}$, then
the restriction of $\|\cdot \|_v$ to $\Qq$ is
$|\cdot |_p^{[K_v:\Qq_p]/[K:\Qq ]}$.
Define the norms and absolute height of $\x =(x_1\kdots x_n)\in K^n$
by $\| \x\|_v:=\max_{1\leq i\leq n}\|x_i\|_v$ for $v\in M_K$ and
$H(\x ):=\prod_{v\in M_K}\| \x\|_v$.

Next, let
$S$ be a finite subset of $M_K$, $n$ an integer $\geq 2$,
and $\{ L_1^{(v)}\kdots L_n^{(v)}\}$ ($v\in S$) linearly independent systems
of linear forms from $K[X_1\kdots X_n]$.
The Subspace Theorem asserts
that
for every $\eps >0$, the set of solutions of
\begin{equation}\label{1.1}
\prod_{v\in S}\prod_{i=1}^n \frac{\|L_i^{(v)}(\x )\|_v}{\|\x\|_v}
\leq H(\x )^{-n-\varepsilon}\ \ \mbox{in }\x\in K^n
\end{equation}
lies in a finite union $T_1\cup\cdots\cup T_{t_1}$ of proper linear subspaces
of $K^n$. Schmidt \cite{Schmidt1975} proved the Subspace Theorem in the case
that $S$ consists of the archi\-med\-ean places of $K$ and Schlickewei 
\cite{Schlickewei1977} extended this to the general case. Much work on the
$p$-adization of the Subspace Theorem was done independently by Dubois and Rhin
\cite{Dubois-Rhin1975}.
 
By an elementary combinatorial argument originating from Mahler
(see \cite[\S21]{Evertse-Schlickewei2002}), inequality \eqref{1.1}
can be reduced to a finite number of systems of inequalities
\begin{equation}\label{1.2}
\frac{\|L_i^{(v)}(\x )\|_v}{\|\x\|_v}\leq H(\x )^{d_{iv}} \
(v\in S,\, i=1\kdots n)\ \ \mbox{in } \x\in K^n,
\end{equation}
where
\[
\sum_{v\in S}\sum_{i=1}^n d_{iv}<-n.
\]
Thus, an equivalent formulation of the Subspace Theorem is, that the
set of solutions of \eqref{1.2} is contained in a finite union
$T_1\cup\cdots \cup T_{t_2}$ of proper linear subspaces of $K^n$.
Making more precise earlier work of Vojta \cite{Vojta1989} and
Schmidt \cite{Schmidt1993}, 
Faltings and W\"{u}stholz \cite[Theorem 9.1]{Faltings-Wuestholz1994} 
obtained the
following refinement: \emph{There exists a single, effectively computable
proper linear subspace $T$ of $K^n$ such that \eqref{1.2} has only finitely
many solutions outside $T$.}

\eqref{1.2} can be translated into a single twisted height inequality.
Put
\[
\delta := -1-\frac{1}{n}\big(\sum_{v\in S}\sum_{i=1}^n d_{iv}\big),\ \ \ \ 
c_{iv}:= d_{iv}-\frac{1}{n}\sum_{j=1}^n d_{jv}\ (v\in S,\, i=1\kdots n).
\]
Thus, 
\[
\sum_{i=1}^n c_{iv} =0\ \mbox{for } v\in S,\ \ \ \delta >0.
\]
For $Q\geq 1$, $\x\in K^n$ define the twisted height
\begin{equation}\label{1.3a}
H_Q(\x ):=\prod_{v\in S} 
\Big(\max_{1\leq i\leq n} \|L_i^{(v)}(\x )\|_vQ^{-c_{iv}}\Big)
\cdot\prod_{v\not\in S}\|\x\|_v.
\end{equation}
Let $\x\in K^n$ be a solution to \eqref{1.2} and take $Q:= H(\x )$.
Then
\begin{equation}\label{1.4}
H_Q(\x )\leq Q^{-\delta}.
\end{equation}

It is very useful to consider \eqref{1.4} with arbitrary reals $c_{iv}$,
not just those arising from system \eqref{1.2}, and with arbitrary reals $Q$
not necessarily equal to $H(\x )$. As will be explained in Section \ref{2},
the definition of $H_Q$ 
can be extended to $\OQq^n$ (where it is assumed that $\OQq\supset K$)
hence \eqref{1.4} can be considered for points $\x\in\OQq^n$.
This leads to the following \emph{Absolute Parametric Subspace Theorem:}
\\[0.1cm]
\emph{Let $c_{iv}$ ($v\in S,\, i=1\kdots n$) be any reals with
$\sum_{i=1}^n c_{iv}=0$ for $v\in S$, and let $\delta >0$.
Then there are a real $Q_0>1$ and a finite number of
proper linear subspaces $T_1\kdots T_{t_3}$ of $\OQq^n$, defined over $K$,
such that for
every $Q\geq Q_0$ there is $T_i\in\{ T_1\kdots T_{t_3}\}$ with}
\[
\{ \x\in \OQq^n:\, H_Q(\x)\leq Q^{-\delta}\}\subset T_i.
\]
Recall that a subspace of $\OQq^n$ is defined over $K$ if it has a basis from $K^n$.
In this general form, this result was first stated and proved in \cite{Evertse-Schlickewei2002}.   
The non-absolute version of the Parametric Subspace Theorem, with solutions
$\x\in K^n$ instead of $\x\in\OQq^n$, was proved implicitly along 
with the Subspace Theorem. To our knowledge, this notion of twisted height
was used for the first time, but in a function field setting, by Dubois
\cite{Dubois1977}.
\end{prg}

\begin{prg}
In 1989, Schmidt was the first to obtain a quantitative version of the
Subspace Theorem.
In \cite{Schmidt1989} he obtained, in the case $K=\Qq$, 
$S=\{ \infty\}$,
an explicit upper bound for the number $t_1$
of subspaces containing the solutions of \eqref{1.1}.
This was generalized to arbitrary $K,S$ by Schlickewei \cite{Schlickewei1992}
and improved by Evertse \cite{Evertse1996}.
First in 1996 Schlickewei \cite{Schlickewei1996} in a special case,
and then in 2002
Evertse and Schlickewei \cite{Evertse-Schlickewei2002} in full generality,
obtained a quantitative version 
of the Absolute Parametric Subspace Theorem, i.e., with explicit upper bounds for $Q_0$ and $t_3$.
As it turned out, this version is in general more useful for applications
than the existing
quantitative versions of the basic Subspace Theorem concerning \eqref{1.1}.
For instance,
the work of Evertse and Schlickewei led to uniform upper bounds for the 
number of solutions of linear equations in unknowns
from a multiplicative group of finite rank \cite{Evertse-Schlickewei-Schmidt2002} and
for the zero multiplicity of linear recurrence sequences \cite{Schmidt1999},
and more recently to results on the
complexity of $b$-ary expansions of algebraic
numbers \cite{Bugeaud-Evertse2008}, \cite{Bugeaud2008}, 
to improvements and generalizations of the Cugiani-Mahler theorem \cite{Bugeaud2007},
and approximation to algebraic numbers by algebraic numbers \cite{Bugeaud2011}.
For an overview of recent applications of the Quantitative Subspace Theorem 
we refer to Bugeaud's survey paper \cite{Bugeaud2010}. 
\end{prg}

\begin{prg}
In the present paper, we obtain an improvement of the quantitative version of 
Evertse and Schlickewei on the Absolute Parametric Subspace Theorem, 
with a substantially sharper bound for $t_3$. Our general result is stated in Section \ref{2}. 
In Section \ref{3} we give some applications to \eqref{1.2} and \eqref{1.1}.

To give a flavour, in this introduction we state special cases of our results.
Let $K,S$ be as above, and let $c_{iv}$ $(v\in S,\, i=1\kdots n)$ be reals with
\begin{equation}\label{1.5}
\sum_{i=1}^n c_{iv} =0\ \mbox{for } v\in S,\ \ \ \ \sum_{v\in S}\max (c_{1v}\kdots c_{nv})\leq 1;
\end{equation}
the last condition is a convenient normalization. Further,
let $L_i^{(v)}$ ($v\in S,\, i=1\kdots n)$ be linear forms such that for $v\in S$,
\begin{equation}\label{1.6}
\left\{
\begin{array}{l}
\{ L_1^{(v)}\kdots L_n^{(v)}\}\subset\{ X_1\kdots X_n,\, X_1+\cdots +X_n\},
\\
\{ L_1^{(v)}\kdots L_n^{(v)}\}\ \mbox{is linearly independent,}
\end{array}\right. 
\end{equation}
and let $H_Q$ be the twisted height defined by \eqref{1.3a} and then extended to $\OQq$.
Finally, let $0<\delta\leq 1$.
Evertse and Schlickewei proved in \cite{Evertse-Schlickewei2002} that in this case,
the above stated Absolute Parametric Subspace Theorem holds with 
\[
Q_0:= n^{2/\delta},\ \ \ t_3\leq  4^{(n+9)^2}\delta^{-n-4}. 
\]
This special case is the basic tool in the work of
\cite{Evertse-Schlickewei-Schmidt2002}, \cite{Schmidt1999} quoted above.
We obtain the following improvement.

\begin{thm}\label{th:1.1}
Assume \eqref{1.5}, \eqref{1.6} and let $0<\delta\leq 1$. 
Then there are proper linear subspaces $T_1\kdots T_{t_3}$ of $\OQq^n$, all defined
over $K$, with
\[
t_3\leq 10^6 2^{2n}n^{10}\delta^{-3}\big(\log (6n\delta^{-1})\big)^2,
\]
such that for every $Q$ with $Q\geq n^{1/\delta}$ there is $T_i\in\{ T_1\kdots T_{t_3}\}$ with
\[
\{ \x\in\OQq^n:\, H_Q(\x )\leq Q^{-\delta}\}\subset T_i.
\]
\end{thm}

A new feature of our paper is the following \emph{interval result}.

\begin{thm}\label{th:1.2}
Assume again \eqref{1.5}, \eqref{1.6}, $0<\delta\leq 1$.
Put
\[
m:= \left[10^5 2^{2n}n^{10}\delta^{-2}\log (6n\delta^{-1})\right],\ \ \
\omega :=  \delta^{-1}\log 6n .
\]
Then there are an effectively computable proper linear subspace $T$ of $\OQq^n$, defined over $K$, and reals
$Q_1\kdots Q_m$ with $n^{1/\delta}\leq Q_1<\cdots <Q_m$, such that
for every $Q\geq 1$ with
\[
\{ \x\in\OQq^n:\, H_Q(\x )\leq Q^{-\delta}\}\not\subset T
\]
we have 
\[
Q\in \left[\left. 1, n^{1/\delta}\right)\right.\cup 
\left[\left. Q_1,Q_1^{\omega}\right)\right.\cup\cdots\cup\left[\left. Q_m,Q_m^{\omega}\right)\right. .
\]
\end{thm}

The reals $Q_1\kdots Q_m$ cannot be determined effectively from our proof.
Theorem \ref{th:1.1} is deduced from Theorem \ref{th:1.2} and a gap principle. The precise
definition of $T$ is given in Section \ref{2}. We show that in the case
considered here, i.e., with \eqref{1.6}, 
the space $T$ is the set of $\x =(x_1\kdots x_n)\in\OQq^n$
with
\begin{equation}\label{1.7}
\sum_{j\in I_i} x_j=0\ \mbox{for } i=1\kdots p,
\end{equation}   
where $I_1\kdots I_p$ ($p=n-\dim T$) are certain pairwise disjoint
subsets of $\{ 1\kdots n\}$
which can be determined effectively.
 
As an application, we give a refinement of the Theorem of Faltings and 
W\"{u}stholz on \eqref{1.2} mentioned above, again under assumption \eqref{1.6}.

\begin{cor}\label{co:1.3}
Let $K,S$ be as above, let $L_i^{(v)}$ $(v\in S,\, i=1\kdots n)$ be linear forms with
\eqref{1.6} and let $d_{iv}$ ($v\in S,\, i=1\kdots n$) be reals with
\[
d_{iv}\leq 0\ \mbox{for } v\in S,\, i=1\kdots n,\ \ \ \ \sum_{v\in S}\sum_{i=1}^n d_{iv}=-n-\eps\ 
\mbox{with } 0<\eps\leq 1.
\]
Put
\[
m{'}:= \left[10^6 2^{2n}n^{12}\eps^{-2}\log (6n\eps^{-1})\right],\ \ \  
\omega{'} :=  2n\eps^{-1}\log 6n .
\]
Then there are an effectively computable linear subspace $T'$ of $K^n$,
and reals $H_1\kdots H_{m'}$ with $n^{n/\eps}\leq H_1<H_2<\cdots <H_{m'}$
such that for every solution $\x\in K^n$ of \eqref{1.2} we have
\[
\x\in T'\ \ \mbox{or  } H(\x )\in 
\big[ 1, n^{n/\eps}\big)\cup 
\big[ H_1,H_1^{\omega{'}}\big)\cup
\cdots\cup\big[ H_{m'},H_{m'}^{\omega{'}}\big) .
\]
\end{cor}

Corollary \ref{co:1.3} follows by applying Theorem \ref{th:1.2} with 
\begin{eqnarray*}
&&c_{iv}:=\frac{n}{n+\eps}\Big( d_{iv}-\frac{1}{n}\sum_{j=1}^n d_{jv}\Big)\ \, (v\in S,\, i=1\kdots n),
\\
&&\delta := \frac{\eps}{n+\eps},\ \ \ Q:=H(\x )^{1+\eps /n}.
\end{eqnarray*}
The exceptional subspace $T'$ is the set of $\x \in K^n$ with \eqref{1.7}
for certain pairwise disjoint subsets $I_1\kdots I_p$ of $\{ 1\kdots n\}$.

It is an open problem to estimate from above the number of solutions
of \eqref{1.2} outside $T'$.  
\end{prg}

\begin{prg}
In Sections \ref{2}, \ref{3} we formulate our generalizations of the above
stated results to arbitrary linear forms.
In particular, in Theorem \ref{th:2.1} we give our general quantitative version
of the Absolute Parametric Subspace Theorem, which 
improves the result of Evertse and Schlickewei from
\cite{Evertse-Schlickewei2002},
and in Theorem \ref{th:2.3} we give our general interval result,
dealing with points $\x\in\OQq^n$ outside an exceptional subspace $T$.
Further, in Theorem \ref{th:2.2} we give an ``addendum" to Theorem \ref{th:2.1}
where we consider \eqref{1.4} for small values of $Q$.
In Section \ref{3} we give some applications to the Absolute Subspace Theorem,
i.e., we consider absolute generalizations of \eqref{1.2}, \eqref{1.1}, 
with solutions $\x$ taken from $\OQq^n$ instead of $K^n$.
Our central result is Theorem \ref{th:2.3}
from which the other results are deduced.
\end{prg}

\begin{prg}
We briefly discuss the proof of Theorem \ref{th:2.3}.
Recall that Schmidt's proof of his 1972 version of the Subspace Theorem
\cite{Schmidt1972}, \cite{Schmidt1980}
is based on geometry of numbers and
``Roth machinery," i.e., the construction of an auxiliary multi-homogeneous
polynomial and an application of Roth's Lemma.
The proofs of the quantitative versions of the Subspace Theorem and 
Parametric Subspace Theorem published since,
including that of Evertse and Schlickewei, essentially follow the same lines.
In 1994, Faltings and W\"{u}stholz \cite{Faltings-Wuestholz1994} 
came with a very different
proof of the Subspace Theorem.
Their proof is an inductive argument which involves 
constructions of auxiliary global line bundle sections
on products of projective varieties of very large degrees, and an application
of Faltings' Product Theorem.
Ferretti observed that with their method,
it is possible to
prove quantitative results like ours,
but with much larger bounds, 
due to the highly non-linear projective varieties
that occur in the course of the argument.

In our proof of Theorem \ref{th:2.3} we use ideas from
both Schmidt and Faltings and W\"{u}stholz.
In fact, similarly to Schmidt, we pass from $\OQq^n$ to an
exterior power $\wedge^p \OQq^n$ by means of techniques from the
geometry of numbers, and apply the Roth machinery to the exterior power.
But there, we replace Schmidt's construction of an auxiliary polynomial 
by that of Faltings and W\"{u}stholz.

A price we have to pay is, that our Roth machinery works only
in the so-called \emph{semistable case}
(terminology from \cite{Faltings-Wuestholz1994})
where the exceptional space $T$ in Theorem \ref{th:2.3} 
is equal to $\{ {\bf 0}\}$.
Thus, we need an involved additional argument to
reduce the general case
where $T$ can be arbitrary to the semistable case.

In this reduction we obtain, as a by-product of some independent interest,
a result on the limit behaviour of the successive infima 
$\lambda_1(Q)\kdots \lambda_n(Q)$ of $H_Q$ as $Q\to\infty$,
see Theorem \ref{th:16.2}. 
Here, $\lambda_i(Q)$ is the infimum of all $\lambda >0$,
such that the set of $\x\in\OQq^n$ with $H_Q(\x )\leq\lambda$ 
contains at least $i$
linearly independent points. Our limit result may be viewed as the ``algebraic" analogue
of recent work of Schmidt and Summerer \cite{Schmidt-Summerer2009}.  
\end{prg}

\begin{prg}
Our paper is organized as follows.
In Sections \ref{2}, \ref{3} we state our results.
In Sections \ref{4}, \ref{5}
we deduce from Theorem \ref{th:2.3} the other theorems 
stated in Sections \ref{2}, \ref{3}. 
In Sections \ref{6}, \ref{9} we have collected 
some notation and simple facts 
used throughout the paper.
In Section \ref{8} we state the semistable case of Theorem \ref{th:2.3}.
This is proved in Sections \ref{10}--\ref{14}.
Here we follow \cite{Evertse-Schlickewei2002}, 
except that we use the auxiliary polynomial
of Faltings and W\"{u}stholz instead of Schmidt's.
In Sections \ref{15}--\ref{18} 
we deduce the general case of Theorem \ref{th:2.3} from the semistable case. 
\end{prg}

\section{Results for twisted heights}\label{2}

\begin{prg}\label{notation}
All number fields considered in this paper
are contained in a given algebraic closure $\OQq$ of $\Qq$. 
Given a field $F$, we denote by $F[X_1\kdots X_n]^{\lin}$ the $F$-vector space of
linear forms $\alpha_1X_1+\cdots +\alpha_nX_n$ with 
$\alpha_1\kdots\alpha_n\in F$.

Let $K\subset\OQq$ be an algebraic number field.
Recall that the normalized absolute values $\|\cdot\|_v$ ($v\in M_K$)
introduced in Section \ref{1}
satisfy the Product Formula
\begin{equation}\label{2.1}
\prod_{v\in M_K} \|x\|_v=1\ \ \mbox{for } x\in K^*.
\end{equation}
Further, if $E$ is any finite extension of $K$ and we define normalized
absolute values $\|\cdot \|_w$ ($w\in M_E$) in the same manner as those
for $K$, we have for every place $v\in M_K$ and each place
$w\in M_E$ lying above $v$,
\begin{equation}\label{2.2}
\| x\|_w=\|x\|_v^{d(w|v)}\ \mbox{for } x\in K,\ \mbox{where }
d(w|v):=\frac{[E_w:K_v]}{[E:K]}
\end{equation}
and $K_v,\, E_w$ denote the completions of $K$ at $v$, $E$ at $w$, respectively.
Notice that
\begin{equation}\label{2.3}
\sum_{w|v} d(w|v)=1,
\end{equation}
where '$w|v$' indicates that $w$ is running through all places of $E$
that lie above $v$.
\end{prg}

\begin{prg}\label{twisted-height}
We list the definitions and technical assumptions
needed in the statements of our theorems. In particular,
we define our twisted heights.

Let again $K\subset\OQq$ be an algebraic number field.
Further, let $n$ be an integer, 
$\LL =(L_i^{(v)}:\, v\in M_K,\, i=1\kdots n)$ a tuple of linear forms,
and $\cc =(c_{iv}:\, v\in M_K,\, i=1\kdots n)$ a tuple of reals satisfying
\begin{eqnarray} 
\label{2.10a}
&n\geq 2,\ \ \ \mbox{$L_i^{(v)}\in K[X_1\kdots X_n]^{\lin}$ for $v\in M_K$,
$i=1\kdots n$,}&
\\
&
\label{2.6a}
\{ L_1^{(v)}\kdots L_n^{(v)}\}\ \mbox{is linearly independent for $v\in M_K$,}& 
\\
\label{2.7}
&\displaystyle{\bigcup_{v\in M_K} \{ L_1^{(v)}\kdots L_n^{(v)}\}=:\{ L_1\kdots L_r\}\ 
\mbox{is finite,}}&
\end{eqnarray}
\begin{eqnarray}
\label{2.7a}
&c_{1v}=\cdots =c_{nv}=0\ \mbox{for all but finitely many $v\in M_K$,}&
\\
\label{2.8}
&\displaystyle{\sum_{i=1}^n c_{iv}=0\ \mbox{for } v\in M_K,}
\\
\label{2.9}
&
\displaystyle{\sum_{v\in M_K}\max (c_{1v}\kdots c_{nv})\leq 1.}
\end{eqnarray}
In addition, let $\delta ,R$ be reals with 
\begin{equation}\label{2.10}
0<\delta\leq 1,\ \ \
R\geq r=\#\left(\bigcup_{v\in M_K} \{ L_1^{(v)}\kdots L_n^{(v)}\}\right),
\end{equation}
and put
\begin{eqnarray}
\label{2.3a}
&\displaystyle{\DL := \prod_{v\in M_K}\|\det (L_1^{(v)}\kdots L_n^{(v)})\|_v,}&
\\[0.15cm]
\label{2.9a}
&\displaystyle\HL :=\prod_{v\in M_K}\max_{1\leq i_1<\cdots <i_n\leq r}
\|\det (L_{i_1}\kdots L_{i_n})\|_v,&
\end{eqnarray}
where the maxima are taken over all $n$-element subsets of
$\{ 1\kdots r\}$.

For $Q\geq 1$ we define the twisted height
$H_{\mathcal{L},{\bf c},Q}:\, K^n\to \Rr$
by
\begin{equation}\label{2.4}
\HLcQ (\x ):=\prod_{v\in M_K}
\max_{1\leq i\leq n} \Big(\| L_i^{(v)}({\bf x})\|_v\cdot Q^{-c_{iv}}\Big).
\end{equation}
In case that $\x={\bf 0}$ we have $\HLcQ (\x )=0$.
If $\x\not={\bf 0}$, it follows from \eqref{2.10a}--\eqref{2.7a}
that all factors in the product are
non-zero and equal to $1$ for all but finitely many $v$;
hence the twisted height is well-defined and non-zero.

Now let $\x\in \OQq^n$. Then there is a finite extension $E$ of $K$ such that
$\x\in E^n$. For $w\in M_E$, $i=1\kdots n$, define
\begin{equation}\label{2.5}
L_i^{(w)}:=L_i^{(v)},\ \ \ c_{iw}:=c_{iv}\cdot d(w|v)
\end{equation}
if $v$ is the place of $K$ lying below $w$, and put
\begin{equation}\label{2.6}
\HLcQ (\x ):= \prod_{w\in M_E}\max_{1\leq i\leq n}
\Big(\| L_i^{(w)}({\bf x})\|_w\cdot Q^{-c_{iw}}\Big).
\end{equation}
It follows from \eqref{2.5}, \eqref{2.2}, \eqref{2.3} that this is
independent of the choice of $E$. Further, by \eqref{2.1},
we have $\HLcQ (\alpha \x )=\HLcQ (\x )$ for $\x\in\OQq^n$, $\alpha\in\OQq^*$.

To define $\HLcQ$, we needed only \eqref{2.10a}--\eqref{2.7a};
properties \eqref{2.8}, \eqref{2.9} are merely convenient normalizations.
\end{prg}

\begin{prg}
Under the above hypotheses, Evertse and Schlickewei
\cite[Theorem 2.1]{Evertse-Schlickewei2002}
obtained the following quantitative version of the Absolute Parametric Subspace Theorem:
\\[0.15cm]
There is a collection $\{ T_1\kdots T_{t_0}\}$ of proper linear subspaces of
$\OQq^n$, all defined over $K$, with
\[
t_0\leq 4^{(n+8)^2}\delta^{-n-4}\log (2R)\log\log (2R)
\] 
such that for every real $Q\geq \max (\HL^{1/R} ,n^{2/\delta})$
there is $T_i\in\{ T_1\kdots T_{t_0}\}$ for which 
\begin{equation}\label{2.12b}
\big\{ \x\in\OQq^n :\,\HLcQ (\x )\leq \DL^{1/n}Q^{-\delta}\big\}\,\subset\, T_i.
\end{equation}
We improve this as follows.

\begin{thm}\label{th:2.1}
Let $n,\LL ,\cc ,\delta ,R$ satisfy \eqref{2.10a}--\eqref{2.10},
and let $\DL ,\HL$ be given by \eqref{2.3a}, \eqref{2.9a}.
\\
Then there are proper linear subspaces $T_1\kdots T_{t_0}$ of $\OQq^n$,
all defined over $K$, with
\begin{equation}\label{2.11}
t_0\leq 10^6 2^{2n}n^{10}\delta^{-3}\log (3\delta^{-1}R)\log (\delta^{-1}\log 3R),
\end{equation}
such that for every real $Q$ with
\begin{equation}\label{2.12}
Q\geq C_0:= \max\big(H_{\mathcal{L}}^{1/R}, n^{1/\delta}\big)
\end{equation}
there is $T_i\in\{ T_1\kdots T_{t_0}\}$ with \eqref{2.12b}.
\end{thm}

\noindent
Notice that in terms of $n,\delta$, our upper bound for $t_0$ 
improves that of Evertse and Schlickewei
from $c_1^{n^2}\delta^{-n-4}$ to $c_2^n\delta^{-3}(\log\delta^{-1})^2$,
while it has the same dependence on $R$.

The lower bound $C_0$ in \eqref{2.12} still has an exponential dependence on $\delta^{-1}$.
We do not know of a method how to reduce it in our general absolute setting.
If we restrict to solutions ${\bf x}$ in $K^n$,
the following can be proved.

\begin{thm}\label{th:2.2}
Let again $n,\LL ,\cc ,\delta ,R$ satisfy \eqref{2.10a}--\eqref{2.10}.
Assume in addition
that $K$ has degree $d$.
\\
Then there are proper linear subspaces
$U_1\kdots U_{t_1}$ of $K^n$, with
\[
t_1\leq \delta^{-1}\big( (90n)^{nd}+3\log\log 3\HL^{1/R}\big)
\]
such that for every $Q$ with $1\leq Q< C_0=\max (\HL^{1/R},n^{1/\delta})$,
there is $U_i\in\{ U_1\kdots U_{t_1}\}$ with
\[
\left\{ \x\in K^n :\HLcQ (\x )\leq \DL^{1/n}Q^{-\delta}\right\}\,\subset\, U_i.
\]
\end{thm}

We mention that in various special cases, by an ad-hoc approach
the upper bound for $t_1$ can be reduced.
Recent work of Schmidt \cite{Schmidt2009} on the number of 
``small solutions" in Roth's Theorem (essentially the case $n=2$ in our setting)
suggests that there should be
an upper bound for $t_1$ with a polynomial instead of exponential dependence on $d$.
\end{prg}

\begin{prg}
We now formulate our general interval result for twisted heights.
We first define an exceptional vector space.
We may view a linear form $L\in \OQq [X_1\kdots X_n]^{\lin}$
as a linear function on $\OQq^n$. Then its restriction to a linear
subspace $U$ of $\OQq^n$ is denoted by $L|_U$.
  
Let $n,\LL ,\cc ,\delta ,R$ satisfy \eqref{2.10a}--\eqref{2.10}.
Let $U$ be a $k$-dimensional linear subspace of $\OQq^n$. 
For $v\in M_K$ we define $w_v(U)=w_{\mathcal{L},\cc,v}(U):=0$ if $k=0$ and 
\begin{eqnarray}\label{2.weightv-space}
&&w_v(U)=w_{\mathcal{L},\cc ,v}(U):=\min\Big\{ c_{i_1,v}+\cdots +c_{i_k,v}:
\\
\nonumber
&&\hspace*{3cm} 
L_{i_1}^{(v)}|_U\kdots L_{i_k}^{(v)}|_U
\ \mbox{are linearly independent}\Big\}
\end{eqnarray}
if $k>0$, where the minimum is taken over all $k$-tuples $i_1\kdots i_k$
such that $L_{i_1}^{(v)}|_U\kdots L_{i_k}^{(v)}|_U$ are linearly
independent. 
Then the \emph{weight of $U$ with respect to $(\mathcal{L} ,\cc )$} 
is defined by 
\begin{equation}\label{2.weight-space}
w(U)=w_{\mathcal{L} ,\cc}(U):=\sum_{v\in M_K} w_v(U).
\end{equation}
This is well-defined since by \eqref{2.7a} at most finitely many of the
quantities $w_v(U)$ are non-zero.

By theory from, e.g., \cite{Faltings-Wuestholz1994}
(for a proof see Lemma \ref{le:15.between} below) 
there is a unique,
proper linear subspace $T=T(\mathcal{L} ,\cc )$ of $\OQq^n$ such that
\begin{equation}\label{2.scss}
\left\{\begin{array}{l}
\displaystyle{\frac{w(T)}{n-\dim T}\geq \frac{w(U)}{n-\dim U}}
\\[0.1cm] 
\hspace*{3cm}\mbox{for every proper linear subspace $U$ of $\OQq^n$;}
\\[0.2cm]
\mbox{subject to this condition, $\dim T$ is minimal.}
\end{array}\right.
\end{equation}
Moreover, this space $T$ is defined over $K$. 

In Proposition \ref{pr:17.5} below, we prove that
\[ 
H_2(T)\leq \left(\max_{v,i} H_2(L_i^{(v)})\right)^{4^n}
\]
with ``Euclidean" heights $H_2$ for subspaces and linear forms defined
in Section \ref{6} below.
Thus, $T$ is effectively computable and it belongs to a finite
collection depending only on $\mathcal{L}$. 
In Lemma \ref{le:15.special} below, we prove that in the special case
considered in Section \ref{1}, i.e.,
\[
\{ L_1^{(v)}\kdots L_n^{(v)}\}\subset\{ X_1\kdots X_n,\, X_1+\cdots +X_n\}\ \mbox{for } v\in M_K
\]
we have
\[
T=\{ \x\in\OQq^n:\, \sum_{j\in I_i} x_j=0\ \mbox{for } j=1\kdots p\}
\]
for certain pairwise disjoint subsets $I_1\kdots I_p$ of $\{ 1\kdots n\}$.

Now our interval result is as follows.

\begin{thm}\label{th:2.3}
Let $n,\LL ,\cc ,\delta ,R$ satisfy \eqref{2.10a}--\eqref{2.10},
and let the vector space $T$ be given by \eqref{2.scss}. 
Put
\begin{eqnarray}
\label{2.13}
&m_0:= \left[10^5 2^{2n}n^{10}\delta^{-2}\log (3\delta^{-1}R)\right],\ \
\omega_0 :=  \delta^{-1}\log 3R .&
\end{eqnarray}
Then there are reals $Q_1\kdots Q_{m_0}$ with
\begin{equation}
\label{2.14}
C_0:=\max (H_{\mathcal{L}}^{1/R},n^{1/\delta})\leq Q_1<\cdots <Q_{m_0}
\end{equation}
such that for every $Q\geq 1$ for which
\begin{equation}\label{2.14a}
\{ \x\in\OQq^n:\, \HLcQ (\x)\leq \DL^{1/n}Q^{-\delta}\}\, \not\subset T
\end{equation}
we have
\begin{equation}\label{2.15}
Q\in [1,C_0)\cup [Q_1,Q_1^{\omega_0})\cup\cdots\cup [Q_{m_0},Q_{m_0}^{\omega_0}).
\end{equation}
\end{thm}
\end{prg}

\section{Applications to Diophantine inequalities}\label{3}

\begin{prg}
We state some results for ``absolute" generalizations of \eqref{1.2}, \eqref{1.1}.
We fix some notation.
The absolute Galois group ${\rm Gal}(\OQq /K)$ 
of a number field $K\subset\OQq$ is denoted by $G_K$.
The absolute height $H(\x )$ of $\x\in\OQq^n$ is defined by choosing a number field
$K$ such that $\x\in K^n$ and taking $H(\x ):=\prod_{v\in M_K}\|\x\|_v$.
The \emph{inhomogeneous} height of
$L=\alpha_1X_1+\cdots +\alpha_nX_n\in\OQq [X_1\kdots X_n]^{\lin}$ is given by
$H^*(L):= H({\bf a})$,
where ${\bf a}=(1,\alpha_1\kdots\alpha_n)$.
Further, for a number field $K$, we define the field $K(L):=K(\alpha_1\kdots \alpha_n)$.

We fix an algebraic number field $K\subset\OQq$.
Further, for every place $v\in M_K$ we choose and then fix an extension of $\|\cdot\|_v$ to $\OQq$. 
For $\x =(x_1\kdots x_n)\in \OQq^n$, $\sigma\in G_K$, $v\in M_K$, we put 
$\sigma (\x ):=(\sigma (x_1)\kdots\sigma (x_n))$, $\|\x\|_v:=\max_{1\leq i\leq n}\|x_i\|_v$.
\end{prg}   

\begin{prg}
We list some technical assumptions and then state our results.
Let $n$ be an integer $\geq 2$, $R$ a real, $S$ a finite subset of $M_K$,
$L_i^{(v)}$ ($v\in S ,\, i=1\kdots n$) linear forms from
$\OQq [X_1\kdots X_n]^{\lin}$, 
and $d_{iv}$ ($v\in S,\, i=1\kdots n$) reals,
such that
\begin{eqnarray}
\label{3.1}
&\{ L_1^{(v)}\kdots L_n^{(v)}\}\ \mbox{is linearly independent for $v\in S$,}&
\\
\label{3.2}
&H^*(L_i^{(v)})\leq H^*,\ \ [K(L_i^{(v)}):K]\leq D\ \mbox{for $v\in S$, $i=1\kdots n$,}&
\\
\label{3.3}
&\displaystyle{
\#\left(\bigcup_{v\in S}\{ L_1^{(v)}\kdots L_n^{(v)}\}\right)\leq  R,}&
\end{eqnarray}
\begin{eqnarray}
\label{3.4}
&\displaystyle{
\sum_{v\in S}\sum_{i=1}^n d_{iv}=  -n-\varepsilon\
\mbox{with } 0<\varepsilon \leq 1,}&
\\
\label{3.5}
&d_{iv}\leq 0\ \mbox{for } v\in S,\,\, i=1\kdots n.&
\end{eqnarray}
Further, put
\begin{equation}\label{3.5a}
A_v:=\|\det (L_1^{(v)}\kdots L_n^{(v)})\|_v^{1/n}\ \ \mbox{for } v\in S.
\end{equation}
\end{prg}

\begin{prg}
We consider the system of inequalities
\begin{equation}\label{3.6}
\max_{\sigma\in G_K}\frac{\|L_i^{(v)}(\sigma (\x ))\|_v}{\|\sigma (\x )\|_v}\,
\leq A_vH(\x)^{d_{iv}}\
(v\in S,\, i=1\kdots n)\ \ \mbox{in } \x\in\OQq^n.
\end{equation}
According to \cite[Theorem 20.1]{Evertse-Schlickewei2002}, the set of solutions $\x\in\OQq^n$
of \eqref{3.6} with $H(\x )\geq \max (H^*, n^{2n/\eps})$ is contained in a union of at most
\begin{equation}\label{1.extra}
2^{3(n+9)^2}\eps^{-n-4}\log (4RD)\log\log (4RD)
\end{equation}
proper linear subspaces of $\OQq^n$ which are defined over $K$.
We improve this as follows.

\begin{thm}\label{th:3.1}
Assume \eqref{3.1}--\eqref{3.5a}. Then the set of solutions
$\x\in\OQq^n$ of system \eqref{3.6} with
\begin{equation}\label{3.7}
H(\x )\geq C_1:=\max ((H^*)^{1/3RD},\, n^{n/\varepsilon})
\end{equation}
is contained in a union of at most
\begin{equation}\label{3.8}
10^9 2^{2n}n^{14}\varepsilon^{-3}\log \big(3\varepsilon^{-1}RD\big)
\log \big( \varepsilon^{-1}\log 3RD\big)
\end{equation}
proper linear subspaces of $\OQq^n$ which are all defined over $K$.
\end{thm}
 
Apart from a factor $\log\eps^{-1}$, in terms of $\eps$ 
our bound is precisely the best known bound
for the number of ``large" approximants 
to a given algebraic number in Roth's Theorem
(see, e.g., \cite{Schmidt2009}).
%
%

Although for applications this seems to be of lesser importance now,
for the sake of completeness we give without proof a quantitative version
of an absolute generalization of \eqref{1.1}. 
We keep the notation and assumptions from \eqref{3.1}--\eqref{3.5a}.
In addition, we put 
\[
s:=\# S,\ \ \  \Delta :=\prod_{v\in S}\| \det (L_1^{(v)}\kdots L_n^{(v)})\|_v.
\]
Consider
\begin{equation}\label{3.x}
\prod_{v\in S}\prod_{i=1}^n 
\max_{\sigma\in G_K}\frac{\| L_i^{(v)}(\sigma (\x ))\|_v}{\|\sigma (\x )\|_v}
\leq \Delta H(\x )^{-n-\eps}.
\end{equation}

\begin{cor}\label{co:3.1b}
The set of solutions $\x\in\OQq^n$ of \eqref{3.x} with $H(\x )\geq H_0$
is contained in a union of at most
\[
\big(9n^2\eps^{-1}\big)^{ns}\cdot
10^{10} 2^{2n}n^{15}\varepsilon^{-3}\log \big(3\varepsilon^{-1}D\big)
\log \big( \varepsilon^{-1}\log 3D\big)
\]
proper linear subspaces of $\OQq^n$ which are all defined over $K$.
\end{cor}

Evertse and Schlickewei \cite[Theorem 3.1]{Evertse-Schlickewei2002}
obtained a similar result, with an upper bound for the number of subspaces
which is about $\big(9n^2\eps^{-1}\big)^{ns}$ times
the quantity in \eqref{1.extra}. 
So in terms of $n$, their bound is of the
order $c^{n^2}$ whereas ours is of the order $c^{n\log n}$. 
Our Corollary \ref{co:3.1b} can be deduced
by following the arguments of \cite[Section 21]{Evertse-Schlickewei2002},
except that instead of Theorem 20.1 of that paper, one has to use
our Theorem \ref{th:3.1}.  
\\[0.15cm]

We now state our interval result, making more precise
the result of Faltings and W\"{u}stholz on \eqref{1.2}.

\begin{thm}\label{th:3.2}
Assume again \eqref{3.1}--\eqref{3.5a}. Put
\begin{eqnarray*}
&&m_1 :=
\left[10^8 2^{2n}n^{14}\varepsilon^{-2}\log
\big( 3\varepsilon^{-1}RD\big)\right],
\\
&&\omega_1 := 3n\varepsilon^{-1}\log 3RD .
\end{eqnarray*}
There are a proper linear subspace $T$ of $\OQq^n$ defined over $K$
which is effectively computable and belongs to a finite collection
depending only on $\{ L_i^{(v)}:\, v\in S,\, i=1\kdots n\}$, as well as reals
$H_1\kdots H_{m_1}$ with
\[
C_1:=\max ((H^*)^{1/3RD},\, n^{n/\varepsilon})\leq H_1<\cdots <H_{m_1},
\]
such that for every solution $\x\in\OQq^n$ of \eqref{3.6}
we have
\[
\x\in T\ \ \ \mbox{or}\ \ \ H(\x )\in [1,C_1)\cup [H_1,H_1^{\omega_1})
\cup\cdots\cup [H_{m_1},H_{m_1}^{\omega_1}).
\]
\end{thm}

Our interval result implies that the solutions $\x\in\OQq^n$ of \eqref{3.6}
outside $T$ have bounded height. In particular, \eqref{1.2} 
has only finitely many solutions $\x\in K^n$ outside $T$.
\end{prg}

\section{Proofs of Theorems \ref{th:2.1} and \ref{th:2.2}}\label{4}

We deduce Theorem \ref{th:2.1} from Theorem \ref{th:2.3},
and prove Theorem \ref{th:2.2}. For this purpose,
we need some gap principles.
We use the notation introduced in Section \ref{2}.
In particular, $K$ is a number field, $n\geq 2$,
$\mathcal{L} =(L_i^{(v)}:\, v\in M_K ,\, i=1\kdots n)$ 
a tuple from $K[X_1\kdots X_n]^{\lin}$, and
$\cc =(c_{iv}:\, v\in M_K:\, i=1\kdots n)$ a tuple of reals.
The linear forms $L_i^{(w)}$ and reals $c_{iw}$, where $w$ is a place
on some finite extension $E$ of $K$, are given by \eqref{2.5}.

We start with a simple lemma.

\begin{lemma}\label{le:4.1}
Suppose that $\mathcal{L} ,\cc$ satisfy
\eqref{2.10a}--\eqref{2.7a}. 
Let $\x \in\OQq^n$, $\sigma\in G_K$, $Q\geq 1$. Then
$\HLcQ (\sigma (\x ))=\HLcQ (\x )$.
\end{lemma}

\begin{proof}
Let $E$ be a finite Galois extension of $K$ such that $\x\in E^n$.
For any place $v$ of $K$ and any place $w$ of $E$ lying above $v$,
there is a unique place $w_{\sigma}$ of $E$ lying above $v$
such that $\|\cdot\|_{w_\sigma} =\|\sigma (\cdot )\|_w$.
By \eqref{2.5} and
$[E_{w_{\sigma}}:K_v]=[E_w:K_v]$ we have $L_i^{(w_{\sigma})}=L_i^{(w)}$,
$c_{i,w_{\sigma}}=c_{iw}$ for $i=1\kdots n$.
Thus,
\begin{eqnarray*}
\HLcQ (\sigma (\x ))&=&\prod_{v\in M_K}\prod_{w|v} \left(\max_{1\leq i\leq n}
\|L_i^{(w)}(\sigma (\x ))\|_wQ^{-c_{iw}}\right)
\\
\nonumber
&=&\prod_{v\in M_K}\prod_{w|v} \left(\max_{1\leq i\leq n}
\|L_i^{(w_{\sigma})}(\x )\|_wQ^{-c_{i,w_{\sigma}}}\right)=\HLcQ (\x ).
\end{eqnarray*}
\end{proof}

We assume henceforth that $n,\mathcal{L} ,\cc ,\delta ,R$ 
satisfy \eqref{2.10a}--\eqref{2.10}. Let $\mu$, $\DL$, $\HL$ 
be given by \eqref{2.9}, \eqref{2.3a}, \eqref{2.9a}.
Notice that
\eqref{2.2}, \eqref{2.3}, \eqref{2.5} imply that
\eqref{2.10a}--\eqref{2.9}
remain valid if we replace $K$ by $E$
and the index $v\in M_K$ by the index $w\in M_E$.
Likewise, in the definitions of $\mu ,\DL ,\HL$ we may replace $K$ by $E$
and $v\in M_K$ by $w\in M_E$.
This will be used frequently in the sequel.

We start with our first gap principle.
For ${\bf a}=(a_1\kdots a_n)\in\Cc^n$ we put
$\|{\bf a}\|:=\max (|a_1|\kdots |a_n|)$.

\begin{prop}\label{pr:4.2}
Let
\begin{equation}\label{4.2}
A\geq n^{1/\delta}.
\end{equation}
Then there is a single proper linear subspace $T_0$ of $\OQq^n$,
defined over $K$, such that for every $Q$ with
\[
A\leq Q<A^{1+\delta /2}
\]
we have $\{ \x\in\OQq^n:\, \HLcQ (\x )\leq\DL^{1/n}Q^{-\delta}\}
\subset T_0$.
\end{prop}

\begin{proof}
Let $Q\in [A,A^{1+\delta /2})$,
and let $\x\in\OQq^n$ with ${\bf x}\not= {\bf 0}$ and
$\HLcQ (\x )\leq \DL^{1/n}Q^{-\delta}$.
Take a finite extension $E$ of $K$ such that $\x\in E^n$.
For $w\in M_E$, put
\[
\theta_w:=\max_{1\leq i\leq n} c_{iw}.
\]
By \eqref{2.5}, \eqref{2.8}, \eqref{2.9} we have
\begin{equation}\label{4.3}
\sum_{i=1}^n c_{iw}=0\ \mbox{for } w\in M_E ,\ \  \sum_{w\in M_E}\theta_w\leq 1.
\end{equation}
Let $w\in M_E$ with $\theta_w>0$.
Using $A\leq Q<A^{1+\delta /2}$ we have 
\[
\max_{1\leq i\leq n}\| L_i^{(w)}(\x )\|_wQ^{-c_{iw}}>
\left(\max_{1\leq i\leq n}\| L_i^{(w)}(\x )\|_wA^{-c_{iw}}\right)\cdot
A^{-\theta_w\delta /2}.
\]
If $w\in M_E$ with $\theta_w=0$ then $c_{iw}=0$ for $i=1\kdots n$
and so we trivially have an equality instead of a strict inequality. 
By taking the product over $w$ and using \eqref{4.3}, we obtain
\begin{eqnarray*}
\HLcQ (\x )&>& H_{\mathcal{L},{\bf c},A}(\x )A^{-\delta /2}\ \
\mbox{if $\theta_w>0$ for some $w\in M_E$,}
\\
\HLcQ (\x )&=& H_{\mathcal{L},{\bf c},A}(\x )> H_{\mathcal{L},{\bf c},A}(\x )A^{-\delta /2}\ \
\mbox{otherwise.}
\end{eqnarray*}
Hence
\begin{equation}\label{4.4}
H_{\mathcal{L},{\bf c},A}(\x )< \DL^{1/n}A^{-\delta /2}.
\end{equation}
This is clearly true for ${\bf x}={\bf 0}$ as well.

Let $T_0$ be the $\OQq$-vector space spanned by the vectors $\x\in\OQq^n$
with \eqref{4.4}. By Lemma \ref{le:4.1}, if $\x$ satisfies \eqref{4.4}
then so does $\sigma (\x )$ for every $\sigma\in G_K$.
Hence $T_0$ is defined over $K$. 
Our Proposition follows once we have shown that $T_0\not= \OQq^n$,
and for this, it suffices to show that $\det (\x_1\kdots\x_n)= 0$
for any $\x_1\kdots \x_n\in\OQq^n$ with \eqref{4.4}.

So take $\x_1\kdots\x_n\in\OQq^n$ with \eqref{4.4}.
Let $E$ be a finite extension of $K$ with $\x_1\kdots\x_n\in E^n$.
We estimate from above
$\|\det (\x_1\kdots \x_n)\|_w$ for $w\in M_E$.
For $w\in M_E$, $j=1\kdots n$, put
\[
\Delta_w :=\|\det (L_1^{(w)}\kdots L_n^{(w)})\|_w,\ \
H_{jw}:=\max_{1\leq i\leq n} \|L_i^{(w)}(\x_j)\|_wA^{-c_{iw}}.
\]

First, let $w$ be an infinite place of $E$.
Put $s(w):=[E_w:\Rr ]/[E:\Qq ]$.
Then there is an embedding $\sigma_w :\, E\hookrightarrow\Cc$
such that $\|\cdot \|_w=|\sigma_w(\cdot )|^{s(w)}$. Put
\begin{equation}\label{4.4a}
{\bf a}_{jw}:= \left( A^{-c_{1w}/s(w)}\sigma_w(L_1^{(w)}(\x_j))\kdots
A^{-c_{nw}/s(w)}\sigma_w(L_n^{(w)}(\x_j))\right)
\end{equation}
for $j=1\kdots n$.
Then $H_{jw}=\|{\bf a}_{jw}\|^{s(w)}$.
So by Hadamard's inequality and \eqref{4.3},
\begin{eqnarray}
\label{4.4b}
\|\det (\x_1\kdots \x_n)\|_w&=& \Delta_w^{-1}\|\det\big(L_i^{(w)}(\x_j)\big)_{i,j}\|_w
\\
\nonumber
&=&
\Delta_w^{-1}A^{c_{1w}+\cdots +c_{nw }}
|\det ({\bf a}_{1w}\kdots {\bf a}_{nw})|^{s(w)}
\\
\nonumber
&\leq& \Delta_w^{-1}n^{ns(w)/2}H_{1w}\cdots H_{nw}.
\end{eqnarray}
Next, let $w$ be a finite place of $E$. Then by
the ultrametric inequality and \eqref{4.3},
\begin{eqnarray}
\label{4.4c}
\|\det (\x_1\kdots \x_n)\|_w&=& \Delta_w^{-1}\|\det\big(L_i^{(w)}(\x_j)\big)_{i,j}\|_w
\\
\nonumber
&\leq &
\Delta_w^{-1}\max_{\rho}
\|L_{\rho (1)}(\x_1)\|_w\cdots \|L_{\rho (n)}(\x_n)\|_w
\\
\nonumber
&\leq&
\Delta_w^{-1}A^{c_{1w}+\cdots +c_{nw }} H_{1w}\cdots H_{nw}
\\
\nonumber
&=&
\Delta_w^{-1}H_{1w}\cdots H_{nw},
\end{eqnarray}
where the maximum is taken over all permutations $\rho$ of $1\kdots n$.

We take the product over $w\in M_E$.
Then using $\prod_{w\in M_E} \Delta_w=\DL$
(by \eqref{2.2}, \eqref{2.5}, \eqref{2.3a}), 
$\sum_{w|\infty} s(w)=1$ (sum of local degrees is global degree),
\eqref{4.3}, \eqref{4.4},
and lastly our assumption $A\geq n^{1/\delta}$,
we obtain
\[
\prod_{w\in M_E} \|\det (\x_1\kdots \x_n)\|_w\leq\DL^{-1}n^{n/2}
\prod_{j=1}^n H_{\mathcal{L},{\bf c},A}(\x_j)< n^{n/2}A^{-n\delta /2}\leq 1.
\]
Now the product formula implies that $\det (\x_1\kdots\x_n)=0$, as required.
\end{proof}

For our second gap principle we need the following lemma.

\begin{lemma}\label{le:4.3}
Let $M\geq 1$. Then $\Cc^n$ is a union of at most $(20n)^nM^2$
subsets, such that for any ${\bf y}_1\kdots {\bf y}_n$ in the same subset,
\begin{equation}\label{4.5}
|\det ({\bf y}_1\kdots {\bf y}_n)|\leq M^{-1}\|{\bf y}_1\|\cdots \|{\bf y}_n\|.
\end{equation}
\end{lemma}

\begin{proof}\cite[Lemma 4.3]{Evertse2010}.
\end{proof}

\begin{prop}\label{pr:4.4}
Let $d:=[K:\Qq ]$  and $A\geq 1$.
Then there are proper linear subspaces $T_1\kdots T_t$ of $K^n$,
with
\[
t\leq (80n)^{nd}
\]
such that for every $Q$ with
\[
A\leq Q<2A^{1+\delta /2}
\]
there is $T_i\in\{ T_1\kdots T_t\}$ with
\[
\{ \x\in K^n: \HLcQ (\x )\leq \DL^{1/n}Q^{-\delta}\}\subset T_i.
\]
\end{prop}

\begin{proof}
We use the notation from the proof of Proposition \ref{pr:4.2}.
Temporarily, we index places of $K$ also by $w$.
Similarly as in the proof of Proposition \ref{pr:4.2} we infer
that if $\x\in K^n$ is such that there exists $Q$ with
$Q\in [A,2A^{1+\delta /2})$ and
$\HLcQ (\x )\leq \DL^{1/n}Q^{-\delta}$, then
\begin{equation}\label{4.6}
H_{\mathcal{L},{\bf c},A}(\x )< 2\DL^{1/n}A^{-\delta /2}.
\end{equation}

Put $M:= 2^n$. Let $w_1\kdots w_r$ be the infinite places of $K$,
and for $i=1\kdots r$ take an embedding $\sigma_{w_i}:\, K\hookrightarrow \Cc$ 
such that
$\|\cdot\|_{w_i}=|\sigma_{w_i}(\cdot )|^{s(w_i)}$.

For $\x\in K^n$ with \eqref{4.6} and $w\in\{ w_1\kdots w_r\}$ put
\[
{\bf a}_{w}(\x ):=\left( A^{-c_{1w}/s(w)}\sigma_{w}(L_1^{(w)}(\x))\kdots
A^{-c_{nw}/s(w)}\sigma_{w}(L_n^{(w)}(\x ))\right).
\]
By Lemma \ref{le:4.3}, the set of vectors
$\x\in K^n$ with \eqref{4.6} is a union of at most
\[
((20n)^nM^2)^r\leq (80n)^{nd}
\]
classes, such that for any $n$ vectors $\x_1\kdots \x_n$ in the same class,
\begin{equation}\label{4.7}
|\det ({\bf a}_w(\x_1)\kdots {\bf a}_w(\x_n))|\leq M^{-1}\ \mbox{for $w=w_1\kdots w_r$.}
\end{equation}

We prove that the vectors $\x\in K^n$ with \eqref{4.6}
belonging to the same class
lie in a single proper linear subspace of $K^n$,
i.e., that any $n$ such vectors have zero determinant. This clearly suffices.

Let $\x_1\kdots\x_n$ be vectors from $K^n$ that satisfy \eqref{4.6}
and lie in the same class.
Let $w$ be an infinite place of $K$. 
Then using \eqref{4.7} instead of Hadamard's inequality,
we obtain, instead of \eqref{4.4b},
\[
\|\det (\x_1\kdots \x_n)\|_w
\leq  \Delta_w^{-1}M^{-s(w)}H_{1w}\cdots H_{nw}.
\]
For the finite places $w$ of $K$ we still have \eqref{4.4c}.
Then by taking the product over $w\in M_K$, we obtain, 
with a similar computation
as in the proof of Proposition \ref{pr:4.2}, employing our choice $M=2^n$,
\[
\prod_{w\in M_K}\|\det (\x_1\kdots\x_n)\|_w< M^{-1}(2A^{-\delta /2})^n\leq 1.
\]
Hence $\det (x_1\kdots x_n)=0$. This completes our proof.
\end{proof}

In the proofs of Theorems \ref{th:2.1} and \ref{th:2.3} we keep the assumptions
\eqref{2.10a}--\eqref{2.10}.

\begin{proof}[Deduction of Theorem \ref{th:2.1} from Theorem \ref{th:2.3}]
Define
\[
\mathcal{S}_Q:=\{\x\in\OQq^n:\, \HLcQ (\x)\leq\DL^{1/n}Q^{-\delta }\}.
\]
Theorem \ref{th:2.3} implies that if $Q$ is a real such that
\[
Q\geq C_0=\max (\HL^{1/R},n^{1/\delta}),\ \ \mathcal{S}_Q\not\subset T
\]
then
\[
Q\in\bigcup_{h=1}^{m_0}\bigcup_{k=1}^s \left[\left. Q_h^{(1+\delta /2)^{k-1}},Q_h^{(1+\delta /2)^k}
\right)\right. ,
\]
where $s$ is the integer with
$(1+\delta /2)^{s-1}<\omega_0\leq (1+\delta /2)^s$.
Notice that we have a union of at most
\[
m_0s\leq m_0\left(1+\frac{\log\omega_0}{\log (1+\delta /2)}\right)\leq 3\delta^{-1}m_0(1+\log \omega_0)
\]
intervals. By Proposition \ref{pr:4.2}, for each of these intervals $I$,
the set $\bigcup_{Q\in I}\mathcal{S}_Q$ 
lies in a proper linear subspace of $\OQq^n$,
which is defined over $K$. Taking into consideration also the exceptional subspace $T$,
it follows that for the number $t_0$ of subspaces in Theorem \ref{th:2.1}
we have
\begin{eqnarray*}
t_0&\leq& 1+3\delta^{-1}m_0 (1+\log\omega_0)
\\
&\leq& 10^6 2^{2n}n^{10}\delta^{-3}\log (3\delta^{-1}R)\log (\delta^{-1}\log 3R).
\end{eqnarray*}
This proves Theorem \ref{th:2.1}.
\end{proof}

\begin{proof}[Proof of Theorem \ref{th:2.2}]
We distinguish between 
$Q\in [ n^{1/\delta},\, C_0)$
and $Q\in [1,n^{1/\delta})$.

Completely similarly as above, we have
\[
[ n^{1/\delta},\, C_0)\,\subseteq\,\bigcup_{j=1}^{s_1}
[n^{(1+\delta/2)^{j-1}/\delta},\, n^{(1+\delta/2)^j/\delta})\ \ (j=1\kdots s_1),
\] 
where $n^{(1+\delta/2)^{s_1-1}/\delta}< C_0\leq n^{(1+\delta/2)^{s_1}/\delta}$,
i.e.,
\begin{equation}\label{4.x}
s_1=1+\left[\frac{\log (\delta\log C_0/\log n)}{\log (1+\delta /2)}\right]
\leq 2+3\delta^{-1}\log\log 3\HL^{1/R}.
\end{equation}
By Proposition \ref{pr:4.2}, for each of the $s_1$ intervals $I$ on the right-hand side, the set
$\left(\bigcup_{Q\in I}\mathcal{S}_Q\right)\cap K^n$ 
lies in a proper linear subspace of $K^n$.

Next consider $Q$ with $1\leq Q<n^{1/\delta}$. 
Define $\gamma_0:=0$,  $\gamma_k:=1+\gamma_{k-1}(1+\delta /2)$ for $k=1,2,\ldots$, 
i.e.,
\[
\gamma_k := \frac{(1+\delta /2)^k-1}{\delta /2}\ \ \mbox{for $k=0,1,2,\ldots$.}
\]
Then
\[
[1,n^{1/\delta})\subseteq \bigcup_{k=1}^{s_2} \left[\left. 2^{\gamma_{k-1}},2^{\gamma_k}\right)\right.
\]
where $(1+\delta /2)^{s_2-1}<\frac{\log (2n^{1/2})}{\log 2}\leq (1+\delta /2)^{s_2}$, i.e.,
\begin{equation}\label{4.y}
s_2= 1+\left[\frac{\log \big(\log (2n^{1/2})/\log 2)}{\log (1+\delta /2)}\right]
<4\delta^{-1}\log\log 4n^{1/2}.
\end{equation}
Applying Proposition \ref{pr:4.4} with $A=2^{\gamma_{k-1}}$ ($k=1\kdots s_2$), 
we see that for each of the $s_2$ intervals $I$ on the right-hand side, 
there is
a collection of at most $(80n)^{nd}$ proper linear subspaces of $K^n$,
such that for every $Q\in I$, the set $\mathcal{S}_Q\cap K^n$ 
is contained in one of these subspaces.

Taking into consideration \eqref{4.x}, \eqref{4.y},
it follows that for the number of subspaces $t_1$ in Theorem \ref{th:2.2}
we have
\begin{eqnarray*}
t_1&\leq& s_1+(80n)^{nd}s_2 \leq 2+ 3\delta^{-1}\log\log 3\HL^{1/R}+ 
(80n)^{nd}\cdot 4\delta^{-1}\log\log 4n^{1/2}
\\
&<& \delta^{-1}\big( (90n)^{nd}+3\log\log 3\HL^{1/R}\big).
\end{eqnarray*}
This proves Theorem \ref{th:2.2}.
\end{proof}

\section{Proofs of Theorems \ref{th:3.1} and \ref{th:3.2}}\label{5}

\begin{prg}
We use the notation introduced in Section \ref{3} and keep the assumptions
\eqref{3.1}--\eqref{3.5a}.
Further, for $L=\sum_{i=1}^n \alpha_iX_i\in\OQq [X_1\kdots X_n]^{\lin}$
and $\sigma\in G_K$, we put
$\sigma (L):= \sum_{i=1}^n \sigma (\alpha_i)X_i$.

Fix
a finite Galois extension $K'\subset\OQq$ of $K$
such that all linear forms
$L_i^{(v)}$ ($v\in S$, $i=1\kdots n$) have their coefficients in $K'$.
Recall that for every $v\in M_K$ we have chosen a continuation of
$\|\cdot \|_v$ to $\OQq$.
Thus, for every $v'\in M_{K'}$ there is
$\tau_{v'}\in{\rm Gal}(K'/K)$
such that
$\|\alpha\|_{v'}=\|\tau_{v'}(\alpha )\|_v^{d(v'|v)}$ for $\alpha\in K'$,
where $v$ is the place of $K$ lying below $v'$.
Put
\begin{equation}\label{5.1}
L_i^{(v)}:=X_i,\  d_{iv}:=0\ \mbox{for $v\in M_K\setminus S$, $i=1\kdots n$}
\end{equation}
and then,
\begin{eqnarray}
\label{5.2}
&&L_i^{(v')}:= \tau_{v'}^{-1}(L_i^{(v)}),\ \ \ 
c_{i,v'}:= 
\displaystyle{d(v'|v)\cdot \frac{n}{n+\varepsilon}\left( d_{iv}-\frac{1}{n}\sum_{j=1}^n d_{jv}\right)\ }
\\[0.1cm]
\nonumber
&&\hspace*{3cm}\mbox{for $v'\in M_{K'}$, $i=1\kdots n$},
\end{eqnarray}
\begin{equation}\label{5.2a}
\left\{
\begin{array}{l}
\mathcal{L}:= (L_i^{(v')}:\, v'\in M_{K'},\, i=1\kdots n),
\\
{\bf c}:= (c_{i,v'}:\, v'\in M_{K'},\, i=1\kdots n),
\end{array}\right.
\end{equation}
and finally,
\begin{equation}\label{5.3}
\delta :=\frac{\varepsilon}{n+\varepsilon}.
\end{equation}
Clearly, 
\begin{eqnarray*}
&&c_{1,v'}=\cdots =c_{n,v'}=0\ \mbox{for all but finitely many $v'\in M_{K'}$,}
\\
&&\sum_{j=1}^n c_{j,v'}=0\ \mbox{for } v\in M_{K'}.
\end{eqnarray*}
Moreover, by \eqref{5.1}, \eqref{5.2},
\eqref{3.5}, \eqref{3.4},
\begin{equation}\label{5.2y}
\Big(\sum_{v'\in M_{K{'}}}\max_{1\leq i\leq n} c_{i,{v'}}\Big)\leq 1.
\end{equation} 
By \eqref{5.1}, \eqref{3.2} we have
\begin{equation}\label{5.4}
\# \bigcup_{v'\in M_{K'}}\{ L_1^{(v')}\kdots L_n^{(v')}\}\leq RD+n.
\end{equation}
These considerations show that \eqref{2.10a}--\eqref{2.10} are satisfied
with $K'$ in place of $K$, 
with the choices of $\LL ,\cc ,\delta$ from \eqref{5.1}--\eqref{5.3},
and with $RD+n$ in place of $R$. 
Further,
\begin{equation}\label{5.5}
\DL =\prod_{v'\in M_{K'}}\|\det (L_1^{(v')}\kdots L_n^{(v')})\|_{v'}
=\prod_{v\in S} \| \det (L_1^{(v)}\kdots L_n^{(v)})\|_v,
\end{equation}
\begin{equation}\label{5.5b}
\HL =\prod_{v'\in M_{K'}}\max_{1\leq i_1<\cdots <i_n\leq r}\|\det (L_{i_1}\kdots L_{i_n})\|_{v'},
\end{equation}
where 
$\bigcup_{v'\in M_{K'}}\{ L_1^{(v')}\kdots L_n^{(v')}\}=:
\{ L_1\kdots L_r\}$.
 
By \eqref{3.2} and the fact that conjugate linear forms
have the same inhomogeneous height, we have
\begin{equation}\label{5.5a}
\max_{1\leq i\leq r} H^*(L_i)=H^*.
\end{equation}
For $v'\in M_{K'}$, $1\leq i_1<\cdots <i_n\leq r$ we have,
by Hadamard's inequality if $v'$ is infinite and
the ultrametric inequality if $v'$ is finite, that
\[
\|\det (L_{i_1}\kdots L_{i_n})\|_{v'}
\leq D_{v'}\prod_{i=1}^r \max (1,\| L_i\|_{v'})
\]
where $D_{v'}:=n^{n[K'_{v'}:\Rr ]/2[K':\Qq ]}$ if $v'$ is infinite and
$D_{v'}:=1$ if $v'$ is finite. 
Taking the product over $v'\in M_{K'}$,
noting that by \eqref{5.1}, \eqref{5.5a}, the set 
$\{ L_1\kdots L_r\}$ contains
$X_1\kdots X_n$, which have inhomogeneous height $1$, and at most $DR$
other linear forms of inhomogeneous height $\leq H^*$,
we obtain
\begin{equation}\label{5.6}
\HL \leq n^{n/2}H^*(L_1)\cdots H^*(L_r)\leq n^{n/2}(H^*)^{DR}.
\end{equation}

The next lemma links system \eqref{3.6}
to a twisted height inequality.
\end{prg}

\begin{lemma}\label{le:5.1}
Let $\x\in\OQq^n$ be a solution of \eqref{3.6}. Then with
$\mathcal{L}$, ${\bf c}$, $\delta$ as defined by \eqref{5.1}--\eqref{5.3}
and with
\[
Q:= H(\x )^{1+\varepsilon /n}
\]
we have
\[
\HLcQ (\sigma (\x ))\leq\DL^{1/n}Q^{ -\delta}\ \mbox{for } \sigma\in G_K.
\]
\end{lemma}

\begin{proof}
Let $\sigma\in G_K$. Put $A_v:=1$ for $v\in M_K\setminus S$. 
Pick a finite Galois extension $E$ of $K$ containing $K'$ and the coordinates
of $\sigma (\x )$.
Let $w\in M_E$ lie above $v'\in M_{K'}$ and the latter in turn above $v\in M_K$.
In accordance with \eqref{2.5}
we define $L_i^{(w)}:=L_i^{(v')}$,
$c_{iw}:=d(w|v')c_{i,v'}$ for $i=1\kdots n$.
Further, we put
$d_{iw}:= d(w|v)d_{iv}$, 
$A_w:= A_v^{d(w|v)}$, 
and we choose $\tau_w\in{\rm Gal}(\OQq /K)$ 
such that $\tau_w|_{K'}=\tau_v$ and
\begin{equation}\label{5.x0}
\|\alpha \|_w=\| \tau_w(\alpha )\|_v^{d(w|v)}\ \mbox{for } \alpha\in E.
\end{equation}
Then \eqref{5.1}, \eqref{5.2} imply for $i=1\kdots n$,
\begin{equation}\label{5.x1}
L_i^{(w)}=\tau_w^{-1}(L_i^{(v)}),\ \ \ \
c_{iw}=\frac{n}{n+\eps}\left( d_{iw}-\frac{1}{n}\sum_{j=1}^n d_{jw}\right).
\end{equation}
If $v\in S$, then from \eqref{3.6} it follows that
\begin{equation}\label{5.x2}
\frac{\|L_i^{(w)}(\sigma (\x ))\|_w}{\|\sigma (\x )\|_w}= 
\left(\frac{\|L_i^{(v)}(\tau_w\sigma (\x ))\|_v}{\|\tau_w\sigma (\x )\|_v}\right)^{d(w|v)}
\leq A_wH(\x )^{d_{iw}},
\end{equation}
while if $v\not\in S$, we have $A_w=1$ and $L_i^{(w)}=X_i$, $d_{iw}=0$
for $i=1\kdots n$,  
and so the inequality is trivially true. 
Finally, \eqref{3.5}, \eqref{3.5a}, \eqref{5.5} imply
\begin{equation}\label{5.x3}
\sum_{w\in M_E}\sum_{i=1}^n d_{iw}=-n-\eps ,\ \ \prod_{w\in M_E} A_w=\DL^{1/n}. 
\end{equation}

By our choice of $Q$ and by \eqref{5.x1}, \eqref{5.x2}, we have
\begin{eqnarray*}
\|L_i^{(w)}(\sigma (\x ))\|_w Q^{-c_{iw}}&=&
\|L_i^{(w)}(\sigma (\x ))\|_wH(\x )^{-d_{iw}+\frac{1}{n}\sum_{j=1}^n d_{jw}}
\\
&\leq& A_w\|\sigma (\x )\|_wH(\x)^{\frac{1}{n}\sum_{j=1}^n d_{jw}}.
\end{eqnarray*}
By taking the product over $w$, using $H(\sigma (\x ))=H(\x )$,
\eqref{5.x3} and again our choice of $Q$
we arrive at
\[
\HLcQ (\sigma (\x ))\leq \DL^{1/n} H(\x )^{1-1-\eps/n}=\DL^{1/n}Q^{-\delta}.
\]
\end{proof} 

In addition we need the following easy observation
which is stated as a lemma for convenient reference.

\begin{lemma}\label{le:5.2}
Let $m,m'$ be integers and $A_0,B_0, \omega ,\omega'$ reals with
$B_0\geq A_0\geq 1$, $\omega '\geq \omega >1$ and $m'\geq m>0$, and let
$A_1\kdots A_m$ be reals with $A_0\leq A_1<\cdots <A_m$.
Then there are reals $B_1\kdots B_{m'}$ with $B_0\leq B_1<\cdots <B_{m'}$
such that
\[
[1,A_0)\cup\left(\bigcup_{h=1}^m\left[\left. A_h,A_h^{\omega}\right)\right.\right)\subseteq
[1,B_0)\cup\left(\bigcup_{h=1}^{m'}\left[\left. B_h,B_h^{\omega ' }\right)\right.\right).
\]
\end{lemma}

\begin{proof} Let 
$S:=\bigcup_{h=1}^m\left[\left. A_h,A_h^{\omega}\right)\right.\cup [A_m^{\omega},\infty )$.
It is easy to see that the lemma is satisfied with 
$B_1$ the smallest real in $S$ with $B_1\geq B_0$ and
$B_j$ the smallest real in $S$ outside 
$\bigcup_{h=1}^{j-1}\left[\left. B_h,B_h^{\omega ' }\right)\right.$
for $j=2\kdots m'$.    
\end{proof}

\begin{proof}[Proof of Theorem \ref{th:3.1}]
We apply Theorem \ref{th:2.1} with $K'$ instead of $K$,
and with $\mathcal{L},{\bf c},\delta$ as in
\eqref{5.1}--\eqref{5.3}; according to \eqref{5.4} we could have taken $n+DR$,
but instead we take $6(DR)^2$ instead of $R$.
Then by \eqref{5.6} the quantity $C_0$ in Theorem \ref{th:2.1} becomes 
\begin{eqnarray*}
C_0':=\max (\HL^{1/6(RD)^2},n^{1/\delta})&\leq&
\max\left( \big(n^{n/2}(H^*)^{RD}\big)^{1/6(RD)^2},\, n^{1/\delta}\right)
\\
&\leq&
\left(\max ((H^*)^{1/3RD}, n^{n/\eps})\right)^{1+\varepsilon/n}=H_0^{1+\eps /n}
\end{eqnarray*}
and the upper bound for the number of subspaces $t_0$ 
in Theorem \ref{th:2.1} becomes
\begin{eqnarray*}
&&10^6 2^{2n}n^{10}(1+n\varepsilon^{-1})^3\times
\\
&&\qquad\qquad\times\log \big(18(1+n\varepsilon^{-1})(RD)^2\big)
\log\big( (1+n\varepsilon^{-1})\log (18(RD)^2)\big)
\\[0.1cm]
&&\leq
10^9 2^{2n}n^{14}\varepsilon^{-3}\log \big(3\varepsilon^{-1}RD\big)
\log\big( \varepsilon^{-1}\log 3RD)\big)
\end{eqnarray*}
which is precisely the upper bound for the number of subspaces in Theorem \ref{th:3.1}.

Let $\x\in\OQq^n$ be a solution to \eqref{3.6} with $H(\x )\geq H_0$
and put $Q:= H(\x )^{1+\varepsilon /n}$.
Then $Q\geq C_0'$. 
Moreover,
by Lemma \ref{le:5.1} and Theorem \ref{th:2.1} we have
\[
\{ \sigma (\x ):\, \sigma\in G_K\}\subseteq \{ {\bf y}\in\OQq^n:\,
\HLcQ({\bf y})\leq\DL^{1/n}Q^{-\delta}\}\subset T_i
\]
for some $T_i\in\{ T_1\kdots T_{t_0}\}$. But then we have in fact,
that $\x\in T_i':= \bigcap_{\sigma\in G_K} \sigma (T_i)$,
which is a proper linear subspace of $\OQq^n$ defined over $K$.
We infer that the solutions $\x\in\OQq^n$ of \eqref{3.6} with
$H(\x )\geq H_0$ lie in a union $T_1'\cup\cdots\cup T_{t_0}'$
of proper linear subspaces of $\OQq^n$, defined over $K$.
This completes our proof.
\end{proof}


\begin{proof}[Proof of Theorem \ref{th:3.2}]
We apply Theorem \ref{th:2.3} with $K'$ instead of $K$,
with $\mathcal{L},{\bf c},\delta$ as in
\eqref{5.1}--\eqref{5.3} and with $6(DR)^2$ instead of $R$.
An easy computation shows that with these choices, the expressions
for $m_0,\omega_0$ in Theorem \ref{th:2.3},
are bounded above by the quantities $m_1,\omega_1$
from the statement of Theorem \ref{th:3.2}.
Further, $C_0$ becomes a quantity bounded above by $C_1^{1+\eps /n}$.
Now according to Theorem \ref{th:2.3} and Lemmas \ref{le:5.1}, \ref{le:5.2}, 
there are reals $Q_1\kdots Q_{m_1}$ with $C_1^{1+\eps /n}\leq Q_1<\cdots <Q_{m_1}$
such that if $\x\in\OQq^n$ is a solution to \eqref{3.6} outside the subspace
$T=T(\mathcal L ,{\bf c})$ from Theorem \ref{th:2.3},
then
\[
Q:= H(\x )^{1+\varepsilon /n}\in \left[\left. 1,C_1^{1+\eps /n}\right)\right.\cup
\bigcup_{h=1}^{m_1} \left[\left. Q_h,Q_h^{\omega_1}\right)\right. .
\]
So with $H_i:= Q_i^{(1+\eps /n)^{-1}}$ ($i=1\kdots m_1$), we have
\[
H(\x )\in [1,C_1)\cup
\bigcup_{h=1}^{m_1} \left[\left. H_h,H_h^{\omega_1}\right)\right. .
\]
In fact, $H(\x )$ belongs to the above union of intervals
if $\sigma (\x )\not\in T$ for any $\sigma\in G_K$, so in fact already
if $\x\not\in T':=\bigcap_{\sigma\in G_K} \sigma (T)$.
Now $T'$ is a proper $\OQq$-linear subspace
of $\OQq^n$ defined over $K$ and $T'$ is effectively determinable
in terms of $T$. The space $T$ in turn is effectively determinable
and belongs to a finite collection
depending only on $\{ L_i^{(v')}:\, v'\in M_{K'},\,\,i=1\kdots n\}$,
so ultimately only on $\{ L_i^{(v)}:\, v\in S,\,\, i=1\kdots n\}$.
Hence the same must apply to $T'$. This completes our proof.
\end{proof}

\section{Notation and simple facts}\label{6}

We have collected some notation and simple facts for later reference.
We fix an algebraic number field $K\subset\OQq$ 
and use $v$ to index places on $K$. 
We have to deal with 
varying finite extensions $E\subset\OQq$ of $K$ 
and sometimes with varying towers
$K\subset F\subset E\subset \OQq$; 
then places on $E$ are indexed by $w$
and places on $F$ by $u$. 
Completions are denoted by $K_v,E_w,F_u$, etc.
We use notation $w|u$, $u|v$ to indicate
that $w$ lies above $u$, $u$ above $v$.
If $w|v$ we put $d(w|v):=[E_w:K_v]/[E:K]$. 

\begin{prg} {\bf Norms and heights.}
Let $E$ be any algebraic number field.
If $w$ is an infinite place of $E$, there is an embedding $\sigma_w:\, E\hookrightarrow\Cc$
such that $\|\cdot\|_w=|\sigma_w(\cdot )|^{[E_w:\Rr ]/[E:\Qq ]}$.
If $w$ is a finite place of $E$ lying above the prime $p$, then
$\|\cdot\|_w$ is an extension of $|\cdot |_p^{[E_w:\Qq_p]/[E:\Qq ]}$ to $E$.

To handle infinite and finite places simultaneously, we introduce
\begin{equation}\label{6.1}
s(w):=\textstyle{\frac{[E_w:\Rr ]}{[E:\Qq ]}}\ \mbox{if $w$ is infinite,  }
s(w):=0\ \mbox{if $w$ is finite.}
\end{equation}
Thus, for $x_1\kdots x_n\in E$, $a_1\kdots a_n\in\Zz$, $w\in M_E$ we have
\begin{equation}\label{6.2}
\| a_1x_1+\cdots +a_nx_n\|_w\leq \big(\sum_{i=1}^n |a_i|\big)^{s(w)}\cdot
\max (\|x_1\|_w\kdots\|x_n\|_w).
\end{equation}

Let $\x =(x_1\kdots x_n)\in E^n$. Put
\begin{eqnarray*}
&&\| \x \|_w:=\max (\| x_1\|_w\kdots \| x_n\|_w)\ \mbox{for } w\in M_E,
\\[0.15cm]
&&\left.\begin{array}{l}
\displaystyle{\| \x\|_{w,1}:=\big(\sum_{i=1}^n |\sigma_w(x_i)|\big)^{s(w)}}\\
\displaystyle{\| \x\|_{w,2}:=\big(\sum_{i=1}^n |\sigma_w(x_i)|^2\big)^{s(w)/2}}
\end{array}\right\}\ \mbox{for $w\in M_E$, $w$ infinite,}
\\[0.15cm]
&&\|\x\|_{w,1}=\|\x\|_{w,2}:=\|\x\|_w\ \ \mbox{for $w\in M_E$, $w$ finite.}
\end{eqnarray*}
Now for $\x\in\OQq^n$ we define
\[
H(\x ):=\prod_{w\in M_E}\|\x\|_w,\ \ 
H_1(\x ):= \prod_{w\in M_E} \|\x\|_{w,1},\ \ 
H_2(\x ):= \prod_{w\in M_E} \|\x\|_{w,2},
\]
where $E$ is any number field such that $\x\in E^n$.
This is independent of the choice of $E$. Then 
\begin{equation}\label{6.3}
n^{-1}H_1(\x )\leq n^{-1/2}H_2(\x )\leq H(\x )\leq H_2 (\x )\leq H_1(\x )\ 
\mbox{for $\x\in\OQq^n$.}
\end{equation}
The standard inner product of ${\bf x}=(x_1\kdots x_n)$,
${\bf y}=(y_1\kdots y_n)\in\OQq^n$ 
is defined by ${\bf x}\cdot {\bf y}=\sum_{i=1}^n x_iy_i$.
Let again $E$ be an arbitrary number field and $w\in M_E$.
Then by the Cauchy-Schwarz inequality for the infinite places
and the ultrametric inequality for the finite places,
\begin{equation}
\label{6.Cauchy-Schwarz}
\|{\bf x}\cdot{\bf y}\|_w\leq \|{\bf x}\|_{w,2}\cdot\|{\bf y}\|_{w,2}\ \ 
\mbox{for ${\bf x},{\bf y}\in E^n$, $w\in M_E$}.
\end{equation}

If $P$ is a polynomial with coefficients in a number field $E$
or in $\OQq$, we define $\|P\|_w$, $\|P\|_{w,1}$, $\|P\|_{w,2}$,
$H(P)$, $H_1(P)$, $H_2(P)$ by applying the above definitions
to the vector $\x$ of coefficients of $P$.
Then for $P_1\kdots P_r\in E[X_1\kdots X_m]$, $w\in M_E$ we have
\begin{equation}\label{6.4}
\left\{\begin{array}{l}
\| P_1+\cdots +P_r\|_{w,1}\leq r^{s(w)}\max(\| P_1\|_{w,1}\kdots \| P_r\|_{w,1}),
\\ 
\| P_1\cdots P_r\|_{w,1}\leq \| P_1\|_{w,1}\cdots \| P_r\|_{w,1}. 
\end{array}\right.
\end{equation}
\end{prg}

\begin{prg} {\bf Exterior products.}
Let $n$ be an integer $\geq 2$ and $p$ an integer with $1\leq p<n$.
Put $N:=\binom{n}{p}$.
Denote by $C(n,p)$ the sequence of $p$-element subsets of $\{ 1\kdots n\}$,
ordered lexicographically, i.e., $C(n,p)=(I_1\kdots I_N)$,
where
\begin{eqnarray*}
&&I_1=\{ 1\kdots p\},\ I_2=\{ 1\kdots p-1,p+1\}\kdots
\\
&&I_{N-1}=\{ n-p,n-p+2\kdots n\},\ I_N=\{ n-p+1\kdots n\}.
\end{eqnarray*}
We use short-hand notation $I=\{ i_1<\cdots <i_p\}$ for a set 
$I=\{ i_1\kdots i_p\}$ with $i_1<\cdots <i_p$.

We denote by $\det (a_{ij})_{i,j=1\kdots p}$ the $p\times p$-determinant 
with $a_{ij}$ on the $i$-th row and $j$-th column.
The exterior product of 
$\x_1=(x_{11}\kdots x_{1n})$$\kdots$
$\x_p=(x_{p1}\kdots x_{pn})\in\OQq^n$ is given by
\[
\x_1\wedge\cdots\wedge \x_p :=(A_1\kdots A_N),
\]
where 
\[
A_l:=\det (x_{i,i_j})_{i,j=1\kdots p},
\]
with $\{ i_1<\cdots <i_p\}=I_l$ the $l$-th set in the sequence $C(n,p)$,
for $l=1\kdots N$.

Let $\x_1\kdots \x_n$ be linearly independent vectors from $\OQq^n$.
For $l=1\kdots N$, define $\widehat{\x_l}:=\x_{i_1}\wedge\cdots\wedge\x_{i_p}$,
where $I_l=\{ i_1<\cdots<i_p\}$ is the $l$-th set in $C(n,p)$. 
Then
\begin{equation}\label{6.4b}
\det (\widehat{\x}_1\kdots\widehat{\x}_N)=
\pm \Big(\det (\x_1\kdots\x_n)\Big)^{\binom{n-1}{p-1}}.
\end{equation}

Given a number field $E$ such that $\x_1\kdots \x_p\in E^n$ we have,
by Hadamard's inequality for the infinite places and the ultrametric inequality
for the finite places,
\begin{equation}\label{6.4c}
\|\x_1\wedge\cdots\wedge \x_p\|_{w,2}\leq 
\|\x_1\|_{w,2}\cdots\|\x_p\|_{w,2}\ \mbox{for $w\in M_E$.}
\end{equation}
Hence
\begin{equation}\label{6.5}
H_2(\x_1\wedge\cdots\wedge \x_p)\leq H_2(\x_1)\cdots H_2(\x_p)\ \mbox{for }
\x_1\kdots\x_p\in\OQq^n.
\end{equation}

The above definitions and inequalities are carried over to linear forms
by identifying a linear form 
$L=\sum_{j=1}^n a_jX_j ={\bf a}\cdot {\bf X}\in \OQq [X_1\kdots X_n]^{\lin}$
with its coefficient vector ${\bf a}=(a_1\kdots a_n)$,
e.g., $\|L\|_w:=\|{\bf a}\|_w$, $H(L):=H({\bf a})$.
The exterior product of
$L_i=\sum_{j=1}^n a_{ij}X_j={\bf a}_i\cdot{\bf X}
\in \OQq [X_1\kdots X_n]^{\lin}$ ($i=1\kdots p$)
is defined by
\[
L_1\wedge\cdots \wedge L_p := A_1X_1+\cdots +A_NX_N,
\]
where $(A_1\kdots A_N)={\bf a}_1\wedge\cdots\wedge {\bf a}_p$. 
Analogously to \eqref{6.5} we have for any linear forms
$L_1\kdots L_p\in \OQq [X_1\kdots X_n]^{\lin}$ $(1\leq p \leq n)$,
\begin{equation}\label{6.6}
H_2(L_1\wedge\cdots \wedge L_p)\leq H_2(L_1)\cdots H_2(L_p).
\end{equation}
Finally, for any $L_1\kdots L_p\in \OQq [X_1\kdots X_n]^{\lin}$,
$\x_1\kdots \x_p\in\OQq^n$, we have
\begin{equation}\label{6.7}
(L_1\wedge\cdots\wedge L_p)(\x_1\wedge\cdots\wedge\x_p )=
\det (L_i(\x_j))_{1\leq i,j\leq p}.
\end{equation} 
\end{prg}

\begin{prg} {\bf Heights of subspaces.}
Let $T$ be a linear subspace of $\OQq^n$.
The height $H_2(T)$ of $T$ is given by 
$H_2(T):=1$ if $T=\{{\bf 0}\}$ or $\OQq^n$ and
\[
H_2(T):=H_2(\x_1\wedge\cdots\wedge\x_p)
\]
if $T$ has dimension $p$ with $0<p<n$ and $\{ \x_1\kdots \x_p \}$
is any basis of $T$. This is independent of the choice of the basis.
Thus, by \eqref{6.5}, if $\{ \x_1\kdots\x_p\}$ is any basis of $T$,
\begin{equation}\label{6.8}
H_2(T)\leq H_2(\x_1)\cdots H_2(\x_p).
\end{equation}
By a result of Struppeck and Vaaler 
\cite{Struppeck-Vaaler1991}
we have for any two linear subspaces $T_1,T_2$ of $\OQq^n$,
\begin{eqnarray}\label{6.9}
\max\big( H_2(T_1\cap T_2),\, H_2(T_1+T_2)\big)&\leq& 
H_2(T_1\cap T_2)H_2(T_1+T_2)
\\
\nonumber
&\leq& H_2(T_1)H_2(T_2).
\end{eqnarray}

Given a linear subspace $V$ of $\OQq [X_1\kdots X_n]^{\lin}$,
we define $H_2(V):=1$ if $V=\{{\bf 0}\}$ or $\OQq [X_1\kdots X_n]^{\lin}$ 
and $H_2(V):=H_2(L_1\wedge\cdots\wedge L_p)$ otherwise,
where $\{L_1\kdots L_p\}$ is any basis of $V$.

Let $T$ be a linear subspace of $\OQq^n$. Denote by $T^{\bot}$
the $\OQq$-vector space of linear forms $L\in \OQq [X_1\kdots X_n]^{\lin}$
such that $L(\x )=0$ for all $\x\in T$.
Then (\cite[p. 433]{Schmidt1967})
\begin{equation}\label{6.10}
H_2(T^{\bot})=H_2(T).
\end{equation}

We finish with the following lemma.

\begin{lemma}\label{le:6.1}
Let $T$ be a $k$-dimensional linear subspace of $\OQq^n$. Put $p:= n-k$.
Let $\{ {\bf g}_1\kdots {\bf g}_n\}$ be a basis of $\OQq^n$ such that
$\{ {\bf g}_1\kdots {\bf g}_k\}$ is a basis of $T$. 
\\
For $j=1\kdots N$,
put $\widehat{{\bf g}}_j:= {\bf g}_{i_1}\wedge\cdots\wedge {\bf g}_{i_p}$,
where $\{ i_1<\cdots <i_p\}=I_j$ is the $j$-th set in the sequence $C(n,p)$.
Let $\widehat{T}$ be the linear subspace of $\OQq^N$ spanned by
$\widehat{{\bf g}}_1\kdots\widehat{{\bf g}}_{N-1}$. Then
\[
H_2(\widehat{T})=H_2(T).
\]
\end{lemma}

\begin{proof}
Let $L_1\kdots L_n\in\OQq [X_1\kdots X_n]^{\lin}$ such that 
for $i,j=1\kdots n$ we have $L_i({\bf g}_j)=1$
if $i=j$ and $0$ otherwise. 
Then $\{ L_{k+1}\kdots L_n\}$ is a basis of $T^{\bot}$.
Moreover, by \eqref{6.7}, we have
\[
(L_{k+1}\wedge\cdots\wedge L_n)(\widehat{{\bf g}}_j)=0
\]
for $j=1\kdots N-1$. 
Hence $L_{k+1}\wedge\cdots\wedge L_n$ spans $\widehat{T}^{\bot}$.
Now a repeated application of \eqref{6.10} gives
\[
H_2(\widehat{T})=H_2(\widehat{T}^{\bot})=H_2(L_{k+1}\wedge\cdots\wedge L_n) 
=H_2(T^{\bot})=H_2(T).
\]
\end{proof} 
  
\end{prg}

\section{Simple properties of twisted heights}\label{9}

We fix tuples $\LL =(L_i^{(v)}:\, v\in M_K,\, i=1\kdots n)$, 
$\cc =(c_{iv}:\, v\in M_K,\, i=1\kdots n)$ 
satisfying the minimal requirements needed to define the twisted height
$\HLcQ$, that is, \eqref{2.10a}--\eqref{2.7a}.
Further, $\DL$, $\HL$ are defined by \eqref{2.3a}, \eqref{2.9a}, respectively.
Write
$\bigcup_{v\in M_K}\{ L_1^{(v)}\kdots L_n^{(v)}\}=\{ L_1\kdots L_r\}$, and
let $d_1\kdots d_t$ be the non-zero numbers among
\begin{equation}\label{9.-1}
\det (L_{i_1}\kdots L_{i_n})\ \ (1\leq i_1<\cdots <i_n\leq n).
\end{equation}
Then
\begin{equation}\label{9.1}
\prod_{v\in M_K}\max (\|d_1\|_v\kdots \|d_t\|_v)=\HL .
\end{equation}
Clearly,
\[
\prod_{v\in M_K}\min (\|d_1\|_v\kdots \|d_t\|_v)\geq
\prod_{v\in M_K}\frac{\| d_1\cdots d_t\|_v}
{(\max (\|d_1\|_v\kdots \|d_t\|_v))^{t-1}}
\]
and so, invoking the product formula and $t\leq \binom{r}{n}$,
\begin{equation}
\label{9.2}
\prod_{v\in M_K}\min (\|d_1\|_v\kdots \|d_t\|_v)\geq
\HL^{1-\binom{r}{n}}.
\end{equation}
Consequently, for the quantity $\DL$ given by \eqref{2.6} we have
\begin{equation}\label{9.3}
\HL^{1-\binom{r}{n}}\leq\DL \leq \HL .
\end{equation}

\begin{lemma}\label{le:9.1}
Put $\theta :=\sum_{v\in M_K}\max (c_{1v}\kdots c_{nv})$.
Let $Q\geq 1$, $\x\in\OQq^n$, $\x\not={\bf 0}$. Then
\[
\HLcQ (\x) \geq n^{-1}\HL^{-\binom{r}{n}}Q^{-\theta}.
\]
\end{lemma}

\begin{proof}
Let $E$ be a finite extension of $K$ with $\x\in E^n$.
Assume without loss of generality that $L_1\kdots L_n$ (from
$\{ L_1\kdots L_r\}$ defined above) are linearly independent,
and put $\delta_w :=\det (L_1^{(w)}\kdots L_n^{(w)})$
for $w\in M_E$. Note that also $\sum_{w\in M_E} \max_i c_{iw}=\theta$. 
We may write
\[
L_i =\sum_{j=1}^n \gamma_{ijw}L_j^{(w)}\ \ \mbox{for $w\in M_E$, $i=1\kdots n$,}
\]
with $\gamma_{ijw}\in K$. By Cramer's rule, we have
$\gamma_{ijw}=\delta_{ijw}/\delta_w$, 
where $\delta_{ijw}$ is the determinant
obtained from $\delta_w$ by replacing $L_j^{(w)}$ by $L_i$.
So $\delta_{ijw}$ belongs to the set of numbers in \eqref{9.-1}.
Further, $\prod_{w\in M_E}\|\delta_w\|_w =\DL$.
Now \eqref{9.3} gives
\[
\prod_{w\in M_E}\max_{1\leq i,j\leq n}\|\gamma_{ijw}\|_w\leq \DL^{-1}\HL\leq \HL^{\binom{r}{n}}.
\]
Put ${\bf y}:=(L_1({\bf x})\kdots L_n({\bf x}))$. Then,
noting that ${\bf y}\not={\bf 0}$,  
\begin{eqnarray*}
1&\leq& H({\bf y})\leq n\HL^{\binom{r}{n}}
\prod_{w\in M_E}\max_{1\leq i\leq n}\| L_i^{(w)}(\x )\|_w
\\
&\leq&
n\HL^{\binom{r}{n}}Q^{\theta}
\prod_{w\in M_E}\max_{1\leq i\leq n}\| L_i^{(w)}(\x )\|_wQ^{-c_{iw}}
=n\HL^{\binom{r}{n}}Q^{\theta}\HLcQ(\x ).
\end{eqnarray*}
This proves our lemma.
\end{proof}

\begin{lemma}\label{le:9.3}
let $\theta_v$ ($v\in M_K$) be reals, at most finitely many of which are 
non-zero. Put $\Theta := \sum_{v\in M_K} \theta_v$.
Define ${\bf d}=(d_{iv}:\, v\in M_K,\, i=1\kdots n)$ by
$d_{iv}:= c_{iv}-\theta_v$ for $v\in M_K$, $i=1\kdots n$.
\\
(i) Let $\x\in\OQq^n$, $Q\geq 1$. Then
\[
H_{\mathcal{L},{\bf d},Q}(\x )=Q^{\Theta}\HLcQ (\x ).
\]
(ii) Let $U$ be a linear subspace of $\OQq^n$. Then
\[ 
w_{\mathcal{L},\dd}(U)=w_{\mathcal{L},\cc}(U)-\Theta\dim U.
\]
(iii) $T(\LL ,{\bf d})=T(\LL ,\cc )$.
\end{lemma}

\begin{proof}
(i) Choose a finite extension $E$ of $K$ with $\x\in E^n$.
In accordance with our conventions, we put $\theta_w:=d(w|v)\theta_v$
if $w\in M_E$ lies above $v\in M_K$; 
thus, $\sum_{w\in M_E}\theta_w=\sum_{v\in M_K} \theta_v$. 
The lemma now follows trivially
by considering the factors of the twisted heights for $w\in M_E$
and taking the product.

(ii) is obvious, and (iii) is an immediate consequence of (ii).
\end{proof}

For $L\in\OQq [X_1\kdots X_n]^{\lin}$ and a linear map
\[
\varphi :\, \OQq^m\to\OQq^n:\,  
(x_1\kdots x_m)\mapsto (\sum_{j=1}^m a_{1j}x_j\kdots \sum_{j=1}^m a_{nj}x_j)
\]
we define $L\circ\varphi\in \OQq [X_1\kdots X_m]^{\lin}$ by
\[
L\circ\varphi := 
L(\sum_{j=1}^m a_{1j}X_j\kdots \sum_{j=1}^m a_{nj}X_j).
\]
If $L\in K[X_1\kdots X_n]^{\lin}$ and $\varphi$ is defined over $K$,
i.e., $a_{ij}\in K$ for all $i,j$, we have $L\circ\varphi\in K[X_1\kdots X_m]^{\lin}$.
More generally,  for a system of linear forms 
$\mathcal{L}=(L_i^{(v)}:\, v\in M_K,\, i=1\kdots n)$
we put
$\mathcal{L}\circ\varphi := 
(L_i^{(v)}\circ\varphi :\, v\in M_K ,\, i=1\kdots n)$.

\begin{lemma}\label{le:9.2}
Let $(\LL ,\cc )$ be a pair satisfying \eqref{2.10a}--\eqref{2.7a}, 
and $\varphi :\OQq^n\to \OQq^n$
an invertible linear map defined over $K$.
\\[0.1cm]
(i) Let $\x\in \OQq^n$, $Q\geq 1$. Then 
$H_{\mathcal{L}\circ\varphi ,\cc ,Q}(\x )=\HLcQ (\varphi (\x ))$.
\\[0.1cm]
(ii) Let $U$ be a proper linear subspace of $\OQq^n$. Then 
$w_{\mathcal{L}\circ\varphi ,\cc}(U)=w_{\mathcal{L},\cc}(\varphi (U))$.
\\[0.1cm]
(iii) Let $T(\LL\circ\varphi ,\cc )$ be the subspace defined by \eqref{2.scss},
but with $\LL\circ\varphi$ instead of $\varphi$. 
Then $T(\LL\circ\varphi ,\cc )=\varphi^{-1}(T(\LL ,\cc ))$.
\\[0.1cm]
(iv) $\Delta_{\mathcal{L}\circ\varphi}=\DL$, 
$H_{\mathcal{L}\circ\varphi}=\HL$.
\end{lemma}

\begin{proof}
(i), (ii) are trivial. (iii) is a consequence of (ii).
As for (iv), we have by the product formula that
\[
\Delta_{\mathcal{L}\circ\varphi}=
\prod_{v\in M_K}\left(\|\det (\varphi )\|_v\cdot\|\det (L_1^{(v)}\kdots L_n^{(v)}\|_v\right)=\DL
\]
and likewise, $H_{\mathcal{L}\circ\varphi}=\HL$.
\end{proof}

\noindent
{\bf Remark.} A consequence of this lemma is, that in order to prove Theorem \ref{th:2.3},
it suffices to prove it for $\mathcal{L}\circ\varphi$ instead of $\mathcal{L}$
where $\varphi$ is any linear transformation of $\OQq^n$ defined over $K$.
For instance, pick any $v_0\in M_K$ and choose $\varphi$ such that $L_i^{(v_0)}\circ\varphi =X_i$
for $i=1\kdots n$. Thus, we see that in the proof of Theorem \ref{th:2.3} we may assume 
without loss of generality that $L_i^{(v_0)}=X_i$ for $i=1\kdots n$.
It will be convenient to choose $v_0$ such that $v_0$ is non-archimedean, and $c_{i,v_0}=0$
for $i=1\kdots n$. 

\section{An interval result in the semistable case}\label{8}

We formulate an interval result like Theorem \ref{th:2.3},
but under some additional constraints.

We keep the notation and assumptions from Section \ref{2}.
Thus $K$ is an algebraic number field, and $n$,
$\mathcal{L}=(L_i^{(v)}:\, v\in M_K,\, i=1\kdots n)$, 
${\bf c}=(c_{iv}:\, v\in M_K,\, i=1\kdots n)$, $\delta ,R$ satisfy
\eqref{2.10a}--\eqref{2.10}. Further, we add the condition as discussed
in the above remark.

The weight $w(U)=w_{\mathcal{L},{\bf c}}(U)$ 
of a $\OQq$-linear subspace $U$ of $\OQq^n$ 
is defined by \eqref{2.weight-space}.
In addition to the above, we assume that
the pair $(\mathcal{L},{\bf c})$ is \emph{semistable}, that is,
the exceptional space $T=T(\mathcal{L} ,\cc )$ defined by \eqref{2.scss} is
equal to $\{{\bf 0}\}$.

For reference purposes, we have listed all our conditions below.
Thus, $K$ is an algebraic number field, 
and $n$ is a positive integer,
$\delta ,R$ are reals, 
$\mathcal{L}=(L_i^{(v)}:\, v\in M_K,\, i=1\kdots n)$ is a tuple
of linear forms and 
${\bf c}=(c_{iv}:\, v\in M_K,\, i=1\kdots n)$
a tuple of reals such that
\begin{equation}\label{8.1} 
R\geq n\geq 2,\ \  0<\delta\leq 1,
\end{equation}
\begin{eqnarray}
\label{8.5}
&c_{1v}=\cdots =c_{nv}=0\ 
\mbox{for all but finitely many $v\in M_K$,}&
\\
\label{8.6}
&\displaystyle{\sum_{i=1}^n c_{iv}=0\ \mbox{for $v\in M_K$,}}&
\\
\label{8.7}
&\displaystyle{\sum_{v\in M_K}\max (c_{1v}\kdots c_{nv})\leq 1,}&
\end{eqnarray}
\begin{eqnarray}
\label{8.7a}
&L_i^{(v)}\in K[X_1\kdots X_n]^{\lin}\ \mbox{for $v\in M_K$, $i=1\kdots n$,}&
\\
\label{8.3}
&\{ L_1^{(v)}\kdots L_n^{(v)}\}\  \mbox{is linearly independent for $v\in M_K$,}&
\\
\label{8.4}
&\displaystyle{\# \bigcup_{v\in M_K} \{ L_1^{(v)}\kdots L_n^{(v)}\}\,\leq R,}&
\\
\label{8.2}
&
\begin{array}{l}
\mbox{there is a non-archimedean place $v_0\in M_K$ such that}
\\
\mbox{$c_{i,v_0}=0$, $L_i^{(v_0)}=X_i$ for $i=1\kdots n$,}
\end{array}
&
\end{eqnarray}
\begin{equation}\label{8.8}
w(U)\leq 0\ 
\mbox{for every proper linear subspace $U$ of $\OQq^n$.}
\end{equation}
Notice that
\eqref{8.8} is equivalent to the assumption that the space
$T$ defined by
\eqref{2.scss} is $\{ {\bf 0}\}$.

\begin{thm}\label{th:8.1}
Assume \eqref{8.1}--\eqref{8.8}. Put
\begin{equation}
\label{8.definitions}
\left\{\begin{array}{l}
m_2:=\left[ 61n^6 2^{2n}\delta^{-2}\log (22n^2 2^nR/\delta )\right],
\\[0.15cm]
\omega_2:= m_2^{5/2} ,\ \ \ C_2:= (2\HL )^{m_2^{2m_2}}.
\end{array}\right.
\end{equation}
Then there are reals $Q_1\kdots Q_{m_2}$ with 
\[
C_2\leq Q_1<\cdots <Q_{m_2}
\]
such that for every $Q\geq 1$ with
\begin{equation}\label{8.9}
\left\{ \x\in\OQq^n:\, \HLcQ (\x )\leq Q^{-\delta}\right\}
\not=\{ {\bf 0}\}
\end{equation}
we have $Q\in [1,C_2)\cup \bigcup_{h=1}^{m_2}
\left[\left. Q_h,Q_h^{\omega_2}\right)\right.$.
\end{thm}
\vskip 0.1cm\noindent
The factor $\DL^{1/n}$ occurring in \eqref{2.14a} has been absorbed into $C_2$.
Theorem \ref{th:8.1} may be viewed as an extension and refinement of a result
of Schmidt on general Roth systems \cite[Theorem 2]{Schmidt1971}.

Theorem \ref{th:8.1} is proved in Sections
\ref{10}--\ref{14}.
In Sections \ref{15}--\ref{18}
we deduce Theorem \ref{th:2.3}.

We outline how Theorem \ref{th:2.3} is deduced from Theorem \ref{th:8.1}.
Let again $T=T(\mathcal{L} ,\cc )$ 
be the exceptional subspace for $(\mathcal{L} ,\cc )$.
Put $k:= \dim T$.
With the notation used in Sections \ref{15}--\ref{18},
we construct a surjective homomorphism $\varphi{''} :\, \OQq^n\to\OQq^{n-k}$
defined over $K$
with kernel $T$, a tuple 
$\mathcal{L}{''}:=(L_i^{(v)}{''}:\, v\in M_K ,\, i=1\kdots n-k)$ 
in $K[X_1\kdots X_{n-k}]^{\lin}$
and a tuple of reals
${\bf d}=(d_{iv}:\, v\in M_K ,\, i=1\kdots n-k)$
such that $(\mathcal{L}{''},{\bf d})$
satisfies conditions analogous to \eqref{8.1}--\eqref{8.8}
and 
\[
H_{\mathcal{L}{''},{\bf d},Q'}(\varphi{''} (\x ))\ll \HLcQ (\x )
\ \ \mbox{for $\x\in\OQq^n$, $Q\geq C_2$,}
\]
where $Q'=Q^n$. 
Then Theorem \ref{th:8.1} is applied with 
$\mathcal{L}{''}$ and ${\bf d}$.

An important ingredient in the deduction of Theorem \ref{th:2.3} is an
upper bound for the height $H_2(T)$ of $T$.
In fact, in Sections \ref{15},\ref{16} we prove a limit result for the
sucessive infima for $\HLcQ$ (Theorem \ref{th:16.2}) 
where we need Theorem \ref{th:8.1}.
We use this limit result in Section \ref{17} 
to compute an upper bound for $H_2(T)$.
In Section \ref{18} we complete the proof of Theorem \ref{th:2.3}.

\section{Geometry of numbers for twisted heights}\label{10}

We start with some generalities on twisted heights.
Let $K$ be a number field and $n\geq 2$.
Let $(\mathcal{L} ,\cc )$ be a pair for which for the moment  
we require only \eqref{2.10a}--\eqref{2.7a}.

For $\lambda\in\Rr_{\geq 0}$ define 
$T(Q,\lambda )=T(\mathcal{L} ,\cc ,Q,\lambda )$
to be the $\OQq$-vector space generated by 
\[
\{ \x\in\OQq^n:\, \HLcQ (\x )\leq\lambda\}.
\]
We define the successive infima $\lambda_i(Q)=\lambda_i(\mathcal{L} ,\cc ,Q)$
$(i=1\kdots n)$ of $\HLcQ$ by
\[
\lambda_i(Q):=\inf\{\lambda\in\Rr_{\geq 0}:\, \dim T(Q,\lambda )\geq i\}.
\]
Since we are working over $\OQq$, the successive infima need not be 
minima. For $i=1\kdots n$, we define
\[
T_i(Q)=T_i(\mathcal{L} ,\cc ,Q)=\bigcap_{\lambda >\lambda_i(Q)}T(Q,\lambda ).
\]
We insert the following simple lemma.

\begin{lemma}\label{le:9.0}
Let $(\mathcal{L} ,\cc )$ be any pair with \eqref{2.10a}--\eqref{2.7a}, 
and let $Q\geq 1$.
\\
(i) The spaces $T_1(Q)\kdots T_n(Q)$ are defined over $K$.
\\
(ii)
Let $k\in\{ 1\kdots n-1\}$ 
and suppose that $\lambda_k(Q)<\lambda_{k+1}(Q)$.
Then $\dim T_k(Q)=k$ and $T(Q,\lambda )=T_k(Q)$ for all $\lambda$
with $\lambda_k(Q)<\lambda<\lambda_{k+1}(Q)$.
\end{lemma}

\begin{proof}
(i) Lemma \ref{le:4.1} implies that for any $\lambda\in\Rr_{>0}$
and any $\sigma\in G_K$ we have
$\sigma (T(Q,\lambda ))=T(Q,\lambda )$. Hence $T(Q,\lambda )$ is 
defined over $K$. This implies (i) at once.

(ii) 
From the definition of the successive infima it follows
at once that $\dim T(Q,\lambda )=k$ for all $\lambda$
with $\lambda_k(Q)<\lambda <\lambda_{k+1}(Q)$.
Since also $T(Q,\lambda )\subseteq T(Q,\lambda ')$ if $\lambda\leq\lambda '$
this implies (ii). 
\end{proof} 

The quantity $\DL$ is defined by \eqref{2.3a}.
We recall the following analogue of Minkowski's Theorem.

\begin{prop}\label{pr:9.1}
Let again $(\mathcal{L} ,\cc )$ be any pair with \eqref{2.10a}--\eqref{2.7a}. 
Put 
\[
\alpha :=\sum_{v\in M_K}\sum_{i=1}^n c_{iv}.
\]
Then for $Q\geq 1$ we have
\begin{equation}\label{10.minkowski-general}
n^{-n/2}\DL Q^{-\alpha}\leq \lambda_1(Q)\cdots\lambda_n(Q)\leq 2^{n(n-1)/2}\DL Q^{-\alpha}.
\end{equation}
In particular, if $\alpha =0$, then
\begin{equation}\label{10.minkowski}
n^{-n/2}\DL \leq \lambda_1(Q)\cdots\lambda_n(Q)\leq 2^{n(n-1)/2}\DL .
\end{equation}  
\end{prop}

\begin{proof}
This is a reformulation of \cite[Corollary 7.2]{Evertse-Schlickewei2002}.
In fact, this result is an easy consequence of an analogue over $\OQq$
of Minkowski's Theorem on successive minima, due to Roy and Thunder
\cite{Roy-Thunder1996}. Using instead an Arakelov type result of S. Zhang
\cite{Zhang1995}, it is possible to improve $2^{n(n-1)/2}$ to $(cn)^n$
for some absolute constant $c$, but such a strengthening
would not have any effect on our final result.
\end{proof}

From now on, we assume that $n,\delta ,R ,\mathcal{L} ,\cc$ 
satisfy \eqref{8.1}--\eqref{8.8}.
We consider reals $Q$ with
\begin{equation}\label{10.2}
Q\geq C_2,
\end{equation}
where $C_2$ is given by \eqref{8.definitions},
and with \eqref{8.9}, i.e.,
\begin{equation}\label{10.3}
\lambda_1(Q)\leq Q^{-\delta}.
\end{equation}
Our assumptions imply $\alpha =0$, and so \eqref{10.minkowski} holds.
We deduce some consequences.

\begin{lemma}\label{le:10.2}
Suppose $n,\delta ,R, \mathcal{L} ,\cc$ satisfy \eqref{8.1}--\eqref{8.8}
and $Q$ satisfies \eqref{10.2}, \eqref{10.3}.
Let $i_1\kdots i_p$ be distinct indices from $\{ 1\kdots n\}$.
Then
\[
Q^{-p-\half}\leq \lambda_{i_1}(Q)\cdots\lambda_{i_p}(Q)\leq Q^{n-p+\half}.
\]
\end{lemma}

\begin{proof}
Write $\lambda_i$ for $\lambda_i(Q)$.
Lemma \ref{le:9.1} and the conditions \eqref{8.4} (i.e., $r\leq R$), \eqref{8.7}
and \eqref{10.2} imply
\[
\lambda_1\geq n^{-1}\HL^{-\binom{R}{n}}Q^{-1}\geq Q^{-1-1/(3n)}.
\]
This implies at once the lower bound for $\lambda_{i_1}\cdots\lambda_{i_p}$.
Further, by \eqref{10.minkowski}, the upper bound for $\DL$ in \eqref{10.3}
and again \eqref{10.2}, 
\[
\lambda_{i_1}\cdots\lambda_{i_p}\leq 2^{n(n-1)/2}\DL \lambda_1^{p-n}
\leq 2^{n(n-1)/2}\HL \lambda_1^{p-n}\leq Q^{n-p+\half}.
\]
\end{proof} 

\begin{lemma}\label{le:10.3}
Suppose again that $n,R,\delta ,\mathcal{L} ,\cc$ satisfy \eqref{8.1}--\eqref{8.8}, and
that $Q$ satisfies \eqref{10.2}, \eqref{10.3}. Then
there is $k\in\{ 1\kdots n-1\}$ such that
\[
\lambda_k(Q)\leq Q^{-\delta /(n-1)}\lambda_{k+1}(Q).
\]
\end{lemma}

\begin{proof}
Fix $Q$ with \eqref{10.2},\eqref{10.3}. 
Write $\lambda_i$ for $\lambda_i(Q)$, for $i=1\kdots n$.
Then by \eqref{10.minkowski}, the lower bound for $\DL$ in \eqref{9.3} 
and \eqref{10.3}, \eqref{10.2}, 
\[
\lambda_n\geq \left(n^{-n/2}\DL\lambda_1^{-1}\right)^{1/(n-1)}
\geq \left( n^{-n/2}\HL^{1-\binom{R}{n}}Q^{\delta}\right)^{1/(n-1)}\geq 1. 
\]
Take $k\in\{ 1\kdots n-1\}$ such that $\lambda_k/\lambda_{k+1}$ is minimal. Then
\[
\frac{\lambda_k}{\lambda_{k+1}}\leq
\left(\frac{\lambda_1}{\lambda_n}\right)^{1/(n-1)}\leq \lambda_1^{1/(n-1)}\leq 
Q^{-\delta /(n-1)}.
\]
\end{proof}

\section{A lower bound for the height of the $k$-th infimum subspace}\label{7}

Our aim is to deduce a useful lower bound for the height of the vector space $T_k(Q)$,
where $k$ is the index from Lemma \ref{le:10.3}.
It is only at this point that we have to use
our semistability assumption \eqref{8.8}.

We need some lemmas, which are used also elsewhere.
We write in the usual manner
\[
\bigcup_{v\in M_K}\{ L_1^{(v)}\kdots L_n^{(v)}\}=\{ L_1\kdots L_r\}.
\]
The quantity $\HL$ is given by \eqref{2.9a}.

\begin{lemma}\label{le:9.5}
Assume that $\LL$ contains $X_1\kdots X_n$.
Let $\{ d_1\kdots d_m\}$ be the set consisting of $1$, 
all determinants $\det (L_{i_1}\kdots L_{i_n})$ ($1\leq i_1<\cdots <i_n\leq r$),
and all subdeterminants of order $\leq n$ of these determinants. Then
\[
\prod_{v\in M_K}\max (\| d_1\|_v\kdots \|d_m\|_v)=\HL .
\]
\end{lemma}

\begin{proof}
Pick indices $1\leq i_1<\cdots <i_n\leq r$.
Each of the subdeterminants of $\det (L_{i_1}\kdots L_{i_n})$
can be expressed
as a determinant of $n$ linear forms from $L_{i_1}\kdots L_{i_n},X_1\kdots X_n$.
Since $X_1\kdots X_n\in\{ L_1\kdots L_r\}$, these subdeterminants 
are up to sign in the set of determinants
$\det (L_{i_1}\kdots L_{i_n})$ ($1\leq i_1<\cdots <i_n\leq r$).
Now the lemma is clear from \eqref{2.9a}.
\end{proof}

\begin{lemma}\label{le:10.3a}
Let $\LL$, $\cc$ satisfy \eqref{2.10a}--\eqref{2.7a}, 
and suppose in addition that
$\LL$ contains $X_1\kdots X_n$.
Let $T$ be a $k$-dimensional linear subspace of $\OQq^n$,
and $\{ {\bf g}_1\kdots {\bf g}_k\}$ a basis of $T$.
Let $E$ be a finite extension of $K$ such that 
${\bf g}_i\in E^n$ for $i=1\kdots k$.
\\
Let $\theta_1\kdots \theta_u$ be the distinct non-zero numbers among
\[
\det (L_{i_l}({\bf g}_j))_{l,j=1\kdots k}\ \ (1\leq i_1<\cdots <i_k\leq r).
\]
Then
\begin{eqnarray}
\label{10.a}
&&\prod_{w\in M_E}\max (\| \theta_1\|_w\kdots \|\theta_u\|_w)\leq \binom{n}{k}^{1/2}\HL\cdot H_2(T),
\\
\label{10.b}
&&\prod_{w\in M_E}\min (\| \theta_1\|_w\kdots \|\theta_u\|_w)\geq 
\left(\binom{n}{k}^{1/2}\HL\cdot H_2(T)\right)^{1-\binom{r}{k}}.
\end{eqnarray}
\end{lemma}

\begin{proof}
For $w\in M_E$, put 
\[
G_w:=\|{\bf g}_1\wedge\cdots\wedge{\bf g}_k\|_{w,2},\ \ 
H_w:=\max (\|d_1\|_w\kdots \|d_m\|_w),
\]
where
$\{ d_1\kdots d_m\}$ is the set from Lemma \ref{le:9.5}. 
Thus, 
\begin{equation}
\label{10.x}
\prod_{w\in M_E} G_w=H_2(T),\ \ 
\prod_{w\in M_E}H_w=\HL .
\end{equation}

Let $\{ L_{i_1}\kdots L_{i_k}\}$ be a $k$-element subset of
$\{ L_1\kdots L_r\}$.
Then the coefficients of $L_{i_1}\wedge\cdots\wedge L_{i_k}$
(being subdeterminants of order $k$) belong to $\{ d_1\kdots d_m\}$.
 
Now \eqref{6.7}, \eqref{6.Cauchy-Schwarz} imply for $w\in M_E$,
\begin{eqnarray*}
\|\det \left(L_{i_l}({\bf g}_j\right)_{1\leq l,j\leq k}\|_w&=&
\| (L_{i_1}\wedge\cdots\wedge L_{i_k})\cdot 
({\bf g}_1\wedge\cdots\wedge {\bf g}_k)\|_w
\\
&\leq& \|L_{i_1}\wedge\cdots\wedge L_{i_k}\|_{w,2}\cdot 
\|{\bf g}_1\wedge\cdots\wedge {\bf g}_k\|_{w,2}
\\
&\leq& \binom{n}{k}^{s(w)/2}H_wG_w.
\end{eqnarray*}
By taking the maximum over all tuples $i_1\kdots i_k$ 
and then the product over $w\in M_E$, and using \eqref{10.x},
inequality \eqref{10.a} follows.

By the product formula,
\begin{eqnarray*}
\prod_{w\in M_E}\min (\| \theta_1\|_w\kdots \|\theta_u\|_w)&\geq& 
\prod_{w\in M_E}
\frac{\| \theta_1\cdots\theta_u\|_w}{\max (\| \theta_1\|_w\kdots \|\theta_u\|_w)^{u-1}}
\\
&=&\left(\prod_{w\in M_E}\max (\| \theta_1\|_w\kdots \|\theta_u\|_w)\right)^{1-u}
\end{eqnarray*}
and together with \eqref{10.a}, $u\leq \binom{r}{k}$ this implies \eqref{10.b}.
\end{proof}

We now deduce our lower bound for the height of the vector space $T_k(Q)$.

\begin{lemma}\label{le:10.4}
Let $n,R,\delta ,\mathcal{L} ,\cc$ satisfy \eqref{8.1}--\eqref{8.8}
and let 
$Q$ satisfy \eqref{10.2}, \eqref{10.3}.
Further, let $k$ be the index from Lemma \ref{le:10.3}. Then
\[
H_2(T_k(Q))\geq Q^{\delta /3R^n}.
\]
\end{lemma}

\begin{proof}
Put $T:=T_k(Q)$ and $\lambda_i:=\lambda_i(Q)$ for $i=1\kdots n$.
 
Let $v\in M_K$. 
Choose $\{ i_1(v)\kdots i_k(v)\}\subset\{ 1\kdots n\}$
such that the linear\\ forms
$L_{i_1(v)}^{(v)}\kdots L_{i_k(v)}^{(v)}$ are linearly 
independent on $T$ and
\[
w_v(T)=\sum_{l=1}^k c_{i_l(v),v}.
\]
Then by assumption \eqref{8.8},
\[
\sum_{v\in M_K}\sum_{l=1}^k c_{i_l(v),v}=w(T)\leq 0.
\]
Given any finite extension $E$ of $K$ and $w\in M_E$, define
$i_l(w):=i_l(v)$ for $l=1\kdots k$, 
where $v$ is the place of $K$ below $w$. 
Then by \eqref{2.3}, \eqref{2.5} we have
\begin{equation}\label{10.5}
\sum_{w\in M_E}\sum_{l=1}^k c_{i_l(w),w}\leq 0.
\end{equation}

Choose $\eps$ such that
\begin{equation}\label{10.5a}
0<\eps <1,\qquad (1+\varepsilon )\lambda_k<\lambda_{k+1}.
\end{equation}
Then there are linearly independent vectors ${\bf g}_1\kdots{\bf g}_k\in T$ such that
\begin{equation}\label{10.6}
\HLcQ({\bf g}_j)\leq (1+\varepsilon )\lambda_j\ \ \mbox{for } j=1\kdots k.
\end{equation}
Let $E$ be a finite extension of $K$ such that ${\bf g}_j\in E^n$ 
for $j=1\kdots k$.
Put
\[
H_{jw}:=\max_{1\leq i\leq n}\|L_i^{(w)}({\bf g}_j)\|_wQ^{-c_{iw}}\ 
\mbox{for $w\in M_E$, $j=1\kdots k$.}
\]
Thus,
\begin{equation}\label{10.7}
\| L_i^{(w)}({\bf g}_j)\|_w\leq H_{jw}Q^{c_{iw}}\ 
\mbox{for $w\in M_E$, $i=1\kdots n$, $j=1\kdots k$.}
\end{equation}

For $w\in M_E$, put
\[
\theta_w:=\det\big(L_{i_l(w)}^{(w)}({\bf g}_j)\big)_{1\leq l,j\leq k}.
\]
We estimate from above and below $\prod_{w\in M_E}\|\theta_w\|_w$.
We start with the upper bound.
Let $w\in M_E$. First, by \eqref{10.7}, the triangle inequality if $w$
is infinite and the ultrametric inequality if $w$ is finite,
\[
\|\theta_w\|_w\leq (k!)^{s(w)}H_{1w}\cdots H_{kw}Q^{\sum_{l=1}^k c_{i_l(w),w}}.
\]
By taking the product over $w\in M_E$ and inserting \eqref{10.6}, \eqref{10.5},
\eqref{10.5a},
\begin{eqnarray}\label{10.8}
\prod_{w\in M_E}\|\theta_w\|_w &\leq& 
k! \HLcQ({\bf g}_1)\cdots \HLcQ({\bf g}_k)Q^{\sum_{w\in M_E}\sum_{l=1}^k c_{i_l(w),w}}
\\
\nonumber
&\leq& k!(1+\varepsilon )^k\lambda_1\cdots\lambda_k\leq 2^k k! \lambda_1\cdots\lambda_k .
\end{eqnarray}
By Lemma \ref{le:10.3} we have
\begin{eqnarray*}
\lambda_1\cdots\lambda_k&\leq&(\lambda_1\cdots\lambda_k)^{k/n}
\big(Q^{-\delta/(n-1)}\lambda_{k+1}\big)^{k(n-k)/n}
\\
&\leq&
Q^{-k(n-k)\delta /n(n-1)}(\lambda_1\cdots\lambda_n)^{k/n}.
\end{eqnarray*}
Applying \eqref{10.minkowski} and using the upper bound in \eqref{9.3} for $\DL$ we obtain
\[
\lambda_1\cdots\lambda_k \leq 2^{k(n-1)/2}\DL^{k/n}Q^{-k(n-k)\delta /n(n-1)}
\leq 2^{k(n-1)/2}\HL^{k/n}Q^{-k(n-k)\delta /n(n-1)},
\]
and inserting the latter into \eqref{10.8} and using assumption \eqref{10.2}   
leads us to the upper bound
\[
\prod_{w\in M_E}\|\theta_w\|_w\leq Q^{-\delta /2n}.
\] 

From \eqref{10.b} we conclude at once
\[
\prod_{w\in M_E}\|\theta_w\|_w\geq\left(\binom{n}{k}^{1/2}\HL\cdot H_2(T)\right)^{1-\binom{R}{k}}
\]
and a combination with the upper bound just established and again our assumption \eqref{10.2}
gives $H_2(T)\geq Q^{\delta /3R^n}$, as required.
\end{proof}

\section{Inequalities in an exterior power}\label{11}

Letting $Q$ be a real with \eqref{10.2}, \eqref{10.3}, 
$k$ the index from Lemma \ref{le:10.3},
and $N:=\binom{n}{k}$, 
we construct $N-1$ linearly independent vectors 
$\widehat{\hh}_1(Q)\kdots\\
\widehat{\hh}_{N-1}(Q)\in\wedge^{n-k}\OQq^n\cong \OQq^N$
satisfying an appropriate system of inequalities.
The construction is similar to that of \cite{Evertse-Schlickewei2002};
the basic tool is Davenport's Lemma.

In the subsequent sections, 
Theorem \ref{th:8.1} is proved by applying  
the Roth machinery to our system of inequalities.
More precisely, we recall a non-vanishing result in Section \ref{12},
and construct a suitable auxiliary polynomial $P$ in Section \ref{13}.
Assuming Theorem \ref{th:8.1} is false, we show that the non-vanishing
result is applicable to $P$, and with the inequalities derived in the present
section and the properties of $P$ we derive a contradiction.

We start with recalling \cite[Lemma 6.3]{Evertse-Schlickewei2002}.

\begin{lemma}\label{le:11.0}
Let $F$ be any algebraic number field and $A_u$ ($u\in M_F$) positive
reals such that
\[
A_u=1\ \mbox{for all but finitely many } u\in M_F;\ \ \ 
\prod_{u\in M_F} A_u\, >1. 
\]
Then there exist a finite extension $E$ of $F$, and $\alpha\in E^*$,
such that
\[
\|\alpha\|_w\leq A_w\ \mbox{for $w\in M_E$,}
\]
where we have written $A_w:=A_u^{d(w|u)}$, with $u$ the place of $F$ below $w$.
\end{lemma}

We keep the notation and assumptions from sections \ref{8},\ref{10}.
Thus, $n\geq 2$, $K$ is an algebraic number field and $\mathcal{L}$, $\cc$, $R$,
$\delta$ satisfy \eqref{8.1}--\eqref{8.8}. 
We fix a real number $Q\geq 1$.
Temporarily, we write $\lambda_i$ for the $i$-th successive infimum $\lambda_i(Q)$
of $\HLcQ$ ($i=1\kdots n$).
For a subset $\mathcal{S}$ of $\OQq^n$, we denote by $\span{\mathcal{S}}$
the $\OQq$-vector space generated by $\mathcal{S}$.

Let $v_0$ be the place from \eqref{8.4}.
Given a finite extension $E$ of $K$,
we write $w\in M_E,\, w|v_0$ to indicate that we let $w$ run through all places 
of $E$ lying above $v_0$, and $w\in M_E,\, w\nmid v_0$ to indicate that we let
$w$ run through all places of $E$ not lying above $v_0$.

Choose $\eps >0$ such that
\begin{equation}\label{11.2}
\left\{\begin{array}{l}
(1+\eps)^2\lambda_i<\lambda_{i+1}\ 
\mbox{for each $i$ with $\lambda_i<\lambda_{i+1}$,}
\\
(1+\eps )^{n+1}\cdot n\cdot 2^{n^2}<3^{n^2}.
\end{array}\right.
\end{equation}
Then choose linearly independent vectors ${\bf g}_1\kdots{\bf g}_n$ of $\OQq^n$ 
such that
\begin{equation}\label{11.3}
\HLcQ({\bf g}_i)\leq (1+\half\eps )\lambda_i\ \mbox{for $i=1\kdots n$.}
\end{equation}

\begin{lemma}\label{le:11.1}
There exist a finite extension $E$ of $K$, and scalar multiples
${\bf g}_1'\kdots{\bf g}_n'$ of ${\bf g}_1\kdots{\bf g}_n$, respectively, having their
coordinates in $E$, such that
\begin{eqnarray}
\label{11.4}
&&\hspace*{0.25cm} \|L_i^{(w)}({\bf g}_j')\|_w\leq n^{-s(w)}Q^{c_{iw}}\
(i,j=1\kdots n,\, w\in M_E,\,w\nmid v_0),
\\
\label{11.5}
&&\hspace*{0.25cm} \|L_i^{(w)}({\bf g}_j')\|_w\leq \big( (1+\eps )n\lambda_j\big)^{d(w|v_0)}\
(i,j=1\kdots n,\, w\in M_E,\,w\mid v_0).
\end{eqnarray}
\end{lemma}

\begin{proof}
Choose a finite extension $F$ of $K$ such that
${\bf g}_1\kdots{\bf g}_n\in F^n$.
For $j\in\{ 1\kdots n\}$, put
\[
A_{ju}:=
\left\{\begin{array}{l}
\displaystyle{n^{-s(u)}\cdot
\left( \max_{1\leq i\leq n} \| L_i^{(u)}({\bf g}_j)\|_uQ^{-c_{iu}}\right)^{-1}\
(u\in M_F,\, u\nmid v_0),}
\\[0.2cm]
\displaystyle{\left( \frac{n(1+\eps)}{1+\half\eps}
\cdot \HLcQ({\bf g}_j)\right)^{d(u|v_0)}\cdot
\left( \max_{1\leq i\leq n} \| L_i^{(u)}({\bf g}_j)\|_u\right)^{-1}}
\\
\hspace*{8cm}(u\in M_F,\, u\mid v_0).
\end{array}\right.
\]
Notice that for $j=1\kdots n$, at most finitely many among the 
numbers $A_{ju}$ ($u\in M_F$) are $\not=1$, and $\prod_{u\in M_F} A_{ju}>1$.
So we can apply Lemma \ref{le:11.0} and obtain that there are
a finite extension $E$ of $F$, and $\alpha_1\kdots\alpha_n\in E^*$,
such that
\[
\|\alpha_j\|_w\leq A_{jw}\ \mbox{for $w\in M_E$, $j=1\kdots n$,}
\]
where we have written $A_{jw}:=A_{ju}^{d(w|u)}$, with $u$ the place of $F$
below $w$. As is easily seen, we have for $j=1\kdots n$, that
\[
A_{jw}:=
\left\{\begin{array}{l}
\displaystyle{n^{-s(w)}\cdot
\left( \max_{1\leq i\leq n} \| L_i^{(w)}({\bf g}_j)\|_wQ^{-c_{iw}}\right)^{-1}\
(w\in M_E,\, w\nmid v_0),}
\\[0.2cm]
\displaystyle{\left( \frac{n(1+\eps)}{1+\half\eps}
\cdot \HLcQ({\bf g}_j)\right)^{d(w|v_0)}\cdot
\left( \max_{1\leq i\leq n} \| L_i^{(w)}({\bf g}_j)\|_w\right)^{-1}}
\\
\hspace*{8cm}(w\in M_F,\, w\mid v_0).
\end{array}\right.
\]
Together with \eqref{11.3} this implies that 
${\bf g}_j':=\alpha_j{\bf g}_j$ ($j=1\kdots n$) satisfy \eqref{11.4}, \eqref{11.5}.
\end{proof}

\begin{lemma}[Davenport's Lemma]\label{le:11.2}
There exist a finite extension $E$ of $K$, a permutation $\pi$ of
$\{ 1\kdots n\}$, and vectors $\hh_j=\hh_j(Q)\in E^n$ ($j=1\kdots n$), 
with the following properties:
\begin{eqnarray}
\label{11.6}
&&\hspace*{0.25cm}\span\{\hh_1\kdots\hh_j\}=\span\{{\bf g}_1\kdots{\bf g}_j\}\ 
\mbox{for $j=1\kdots n$,}
\\
\label{11.7}
&&\hspace*{0.25cm}
\| L_i^{(w)}(\hh_j)\|_w\leq n^{-s(w)}Q^{c_{iw}}\ 
(i,j=1\kdots n,\, w\in M_E,\,
w\nmid v_0),
\\
\label{11.8}
&&\hspace*{0.25cm}
\| L_{\pi (i)}^{(w)}(\hh_j)\|_w
\leq \left(3^{n^2}\min(\lambda_i,\lambda_j)\right)^{d(w|v_0)}
\\
\nonumber
&&\hspace*{6cm} 
(i,j=1\kdots n,\, w\in M_E,\, w\mid v_0).
\end{eqnarray}
\end{lemma}

\begin{proof}
The proof is the same as that of
\cite[Lemma 9.2]{Evertse-Schlickewei2002}, except for some small modifications.

In fact, starting with ${\bf g}_1\kdots{\bf g}_n$, we construct 
scalar multiples ${\bf g}_1'\kdots{\bf g}_n'$
as in Lemma \ref{le:11.1}. Then 
\cite[(9.17),(9.18)]{Evertse-Schlickewei2002} hold,
but with the vectors ${\bf g}_1\kdots {\bf g}_n$
being replaced by ${\bf g}_1'\kdots{\bf g}_n'$, 
and the numbers $Q^{c_{iw}}$, $(1+\eps )\lambda_j$ by $n^{-s(w)}Q^{c_{iw}}$
and $(1+\eps )n\lambda_j$, respectively, for $i,j=1\kdots n$.
We then copy the proof
of \cite[Lemma 9.2]{Evertse-Schlickewei2002}.
Here we have to use \eqref{11.2} instead of \cite[(9.15)]{Evertse-Schlickewei2002}.
This yields vectors $\hh_1\kdots\hh_n$ satisfying \eqref{11.6}, \eqref{11.7}
and \eqref{11.8} with $2^{n^2}n(1+\eps )^{n+1}$ instead of $3^{n^2}$.   
Together with our assumption \eqref{11.2} this implies
our Lemma \ref{le:11.2}.

In the proof of \cite[Lemma 9.2]{Evertse-Schlickewei2002}, the tuples 
$\LL =(L_i^{(v)}:\, v\in M_K, i=1\kdots n)$ under consideration satisfy, in addition
to \eqref{8.4}, \eqref{8.2},  
the following conditions: $\|\det (L_1^{(v)}\kdots L_n^{(v)})\|_v=1$ for $v\in M_K$,
and $L_1^{(v)}=X_1\kdots L_n^{(v)}=X_n$ for all but finitely many $v\in M_K$.
But these conditions are not used anywhere.
\end{proof}

Let $Q$ be a real with \eqref{10.2},\eqref{10.3} and let $k\in\{ 1\kdots n-1\}$ 
be the index from Lemma \ref{le:10.3}. That is, $Q$ satisfies 
\begin{eqnarray}
\label{11.13a}
&&Q\geq C_2,\ \ \lambda_1(Q)\leq Q^{-\delta},
\\
\label{11.13b}
&&\lambda_k(Q)\leq Q^{-\delta /(n-1)}\lambda_{k+1}(Q).
\end{eqnarray}

Put $N:=\binom{n}{k}$.
Let $C(n,n-k)=(I_1\kdots I_N)$ be the sequence of $(n-k)$-elements subsets
of $\{ 1\kdots n\}$,
arranged in lexicographical order.
Thus, $I_1=\{ 1\kdots n-k\}$, $I_2=\{ 1\kdots n-k-1,n-k+1\}$$\kdots$$I_{N-1}=\{ k,k+1\kdots n\}$,
$I_N=\{ k+1\kdots n\}$.
 
Let $\hh_j=\hh_j(Q)$ ($j=1\kdots n$) be the vectors from Lemma \ref{le:11.2}.
For $v\in M_K$, $j=1\kdots N$, define
\begin{eqnarray}\label{11.9}
&&\widehat{L}_j^{(v)}:=L_{i_1}^{(v)}\wedge\cdots\wedge L_{i_{n-k}}^{(v)},\ \ 
\widehat{c}_{jv}:=c_{i_1,v}+\cdots +c_{i_{n-k},v},
\\
\label{11.10}
&&\widehat{\hh}_j=\widehat{\hh}_j(Q):=\hh_{i_1}(Q)\wedge\cdots\wedge\hh_{i_{n-k}}(Q),
\\
\label{11.11}
&&\nu_j=\nu_j(Q):=\lambda_{i_1}(Q)\cdots\lambda_{i_{n-k}}(Q)
\end{eqnarray}
where $I_j=\{ i_1<\cdots <i_{n-k}\}$. The permutation $\pi$ from Lemma \ref{le:11.2}
induces a permutation $\widehat{\pi}$ of $\{ 1\kdots N\}$, such that if
$I_j=\{ i_1\kdots i_{n-k}\}$, then $I_{\widehat{\pi}(j)}=\{\pi (i_1)\kdots \pi (i_{n-k})\}$.
In the usual manner, we write
\begin{equation}\label{11.11a} 
\widehat{L}_j^{(w)}=\widehat{L}_j^{(v)},\ \ \widehat{c}_{jw}=d(w|v)\widehat{c}_{jv}
\end{equation}
for any place $w$ of any finite extension of $K$,
where $v$ is the place of $K$ below $w$.

Let $E$ be the finite extension of $K$ from Lemma \ref{le:11.2}.
By \eqref{6.7}, \eqref{11.7}, \eqref{11.8} we have for $w\in M_E$,
$i,j=1\kdots N$,
\begin{eqnarray}
\label{11.12}
&&\|\widehat{L}_i^{(w)}(\widehat{\hh}_j)\|_w= 
\|\det \big(L_{i_p}^{(w)}(\hh_{j_q})\big)_{1\leq p,q\leq n-k}\|_w 
\\
\nonumber
&&\qquad\leq 
\left\{\begin{array}{l}
((n-k)!)^{s(w)}n^{-(n-k)s(w)}Q^{\widehat{c}_{iw}}\leq Q^{\widehat{c}_{iw}}\ 
\mbox{if $w\nmid v_0$,}
\\[0.15cm]
3^{n^3}\min\left(\nu_{\widehat{\pi}^{-1}(i)},\,\nu_{\widehat{\pi}^{-1}(j)}\right) 
\ \mbox{if $w\mid v_0$,}
\end{array}\right.
\end{eqnarray}
where $I_i=\{ i_1<\cdots <i_{n-k}\}$, $I_j=\{ j_1<\cdots <j_{n-k}\}$.

Our concern is about the points $\widehat{\hh}_1\kdots\widehat{\hh}_{N-1}$.
Define the quantities $\widehat{c}_{i,v_0}(Q)$ 
($i=1\kdots N$) (so depending on $Q$ (!)) by
\begin{equation}\label{11.12x}
Q^{\widehat{c}_{i,v_0}(Q)}:=
\left\{\begin{array}{l}
3^{n^3}\nu_{\widehat{\pi}^{-1}(i)}(Q)\ \mbox{if } \widehat{\pi}^{-1}(i)\not= N,
\\[0.15cm]
3^{n^3}\nu_{N-1}(Q)\ \mbox{if } \widehat{\pi}^{-1}(i)=N.
\end{array}\right.
\end{equation}
Next, define 
\begin{equation}\label{11.12a}
\widehat{c}_{iw}(Q):=d(w|v_0)\widehat{c}_{i,v_0}(Q)
\end{equation}
if $w$ is a place of some finite extension of $K$ lying above $v_0$.

Now \eqref{11.12} implies for $i=1\kdots N$, $j=1\kdots N-1$,
\begin{equation}\label{11.13}
\left\{
\begin{array}{l}
\| \widehat{L}_i^{(w)}(\widehat{\hh}_j)\|_w\leq Q^{\widehat{c}_{iw}}\ (w\in M_E,\, w\nmid v_0),
\\[0.15cm]
\|\widehat{L}_i^{(w)}(\widehat{\hh}_j)\|_w\leq Q^{\widehat{c}_{iw}(Q)}\ (w\in M_E,\, w\mid v_0).
\end{array}\right. 
\end{equation}
We may take the same finite extension $E$ of $K$ as in \eqref{11.12}
but in fact, in view of \eqref{11.11a}, \eqref{11.12a}, we may take for $E$ any
finite extension of $K$ that contains the coordinates of 
$\widehat{\hh}_1\kdots\widehat{\hh}_{N-1}$.
It is a feature of our new approach,
as opposed to \cite{Evertse-Schlickewei2002},
that it allows to handle exponents $\widehat{c}_{iw}(Q)$
which vary with $Q$.

We have collected some properties of the exponents $\widehat{c}_{iv}$, $\widehat{c}_{i,v_0}(Q)$.

\begin{lemma}\label{le:11.3} 
Let $Q$ be a real with \eqref{11.13a}, \eqref{11.13b}.
Put $N:=\binom{n}{k}$.
Then 
\begin{eqnarray}
\label{11.14}
&&\sum_{i=1}^N\widehat{c}_{iv}=0\ \ \mbox{for } v\in M_K\setminus\{ v_0\}, 
\\
\label{11.14a}
&&\max_{1\leq i\leq N}|\widehat{c}_{iv}|\leq (n-1)\max_{1\leq i\leq n} c_{iv}\ \
\mbox{for } v\in M_K\setminus\{ v_0\},
\\
\label{11.15}
&&\sum_{v\in M_K\setminus\{ v_0\}}\max_{1\leq i\leq N}|\widehat{c}_{iv}|\leq n-1,
\\
\label{11.16}
&&\sum_{i=1}^N \widehat{c}_{i,v_0}(Q)\leq -\delta /n,
\\
\label{11.17}
&&\max_{1\leq i\leq N}|\widehat{c}_{i,v_0}(Q)|\leq n.
\end{eqnarray}
\end{lemma}

\begin{proof}
\eqref{11.14}, \eqref{11.14a} and \eqref{11.15} are easy consequences of 
\eqref{11.9}, \eqref{8.5}--\eqref{8.7}
and the choice of $v_0$: 
\eqref{11.14} is immediate, for \eqref{11.14a} observe that
\[
|\widehat{c}_{jv}|=
\max\left(\sum_{i\in I_j} c_{iv},\sum_{i\not\in I_j} c_{iv}\right)
\leq 
(n-1)\max_{1\leq i\leq n} c_{iv},
\]
and for \eqref{11.15} take the sum over $v$ and apply \eqref{8.7}.
We prove \eqref{11.16}. Write again $\lambda_i,\nu_j$ for $\lambda_i(Q)$, $\nu_j(Q)$
and put $N':=\binom{n-1}{n-k-1}$.
Notice that by \eqref{11.11},
$\nu_{N-1}=\lambda_k\lambda_{k+2}\cdots\lambda_N$,
$\nu_N=\lambda_{k+1}\cdots\lambda_N$. 
Together with \eqref{11.12x}, \eqref{10.minkowski}, 
Lemma \ref{le:10.3}, \eqref{11.13b}, \eqref{11.13a} 
this implies 
\begin{eqnarray*}
Q^{\sum_{i=1}^N \widehat{c}_{i,v_0}(Q)}&=&3^{n^3N}\nu_1\cdots \nu_N(\nu_{N-1}/\nu_N)
\\
\nonumber
&=&3^{n^3N}(\lambda_1\cdots\lambda_n)^{N'}(\lambda_k/\lambda_{k+1})
\\
&\leq&3^{n^3N}2^{n(n-1)N'/2}Q^{-\delta /(n-1)}\leq Q^{-\delta /n}.
\end{eqnarray*}

We finish with proving \eqref{11.17}. let $i\in \{ 1\kdots N\}$. 
By \eqref{11.11}, \eqref{11.12x} we have
\[
Q^{\widehat{c}_{i,v_0}(Q)}=3^{n^3}\lambda_{i_1}\cdots\lambda_{i_{n-k}}
\]
for certain disctinct indices $i_1\kdots i_{n-k}\in\{ 1\kdots n\}$.
Together with Lemma \ref{le:10.2}, \eqref{11.13a},
this implies 
\[
Q^{|\widehat{c}_{i,v_0}(Q)|}\leq 3^{n^3}Q^{n-\half}\leq Q^n.
\]
\end{proof}

Next, we prove some properties of the linear forms $\widehat{L}_i^{(v)}$.
For $v\in M_K$, denote by $\widehat{A}_v$ the matrix of which the $j$-th row consists 
of the coefficients of $\widehat{L}_j^{(v)}$, for $j=1\kdots N$.
The inhomogeneous height of a set 
$\mathcal{S}=\{\alpha_1\kdots\alpha_s\}\subset\OQq$ is given by 
$H^*(\mathcal{S}):=\prod_{w\in M_E}\max (1,\|\alpha_1\|_w\kdots\|\alpha_s\|_w)$
where $E$ is any number field containing $\mathcal{S}$.
If $A_1\kdots A_m$ are matrices with elements from $\OQq$,
we denote by $H^*(A_1\kdots A_m)$ the inhomogeneous height of the set 
of elements of $A_1\kdots A_m$.

\begin{lemma}\label{le:11.4}
Let $\widehat{A}_1\kdots \widehat{A}_s$ be the distinct matrices among 
$\widehat{A}_v\ (v\in M_K)$.
Then
\[
H^*(\widehat{A}_1^{-1}\kdots \widehat{A}_s^{-1})\leq \HL^{R^n}.
\]
\end{lemma}

\begin{proof}
Write $\bigcup_{v\in M_K}\{ L_1^{(v)}\kdots L_n^{(v)}\}=\{ L_1\kdots L_r\}$;
then $r\leq R$. 
For $i=1\kdots s$, let $B_i:=(\det\widehat{A}_i)\widehat{A}_i^{-1}$.
For $v\in M_K$,
put $\delta_v:=\det (L_1^{(v)}\kdots L_n^{(v)})$, and let $\delta_1\kdots\delta_u$
be the distinct numbers among $\delta_v$ ($v\in M_K$).

Thanks to assumption \eqref{8.4}, we can apply Lemma \ref{le:9.5}.
For $v\in M_K$, the elements of the matrix $(\det \widehat{A}_v)\widehat{A}_v^{-1}$ 
are up to sign
the coefficients of $L_{i_1}^{(v)}\wedge\cdots\wedge L_{i_k}^{(v)}$
for all $k$-element subsets $\{ i_1<\cdots <i_k\}$ of $\{ 1\kdots r\}$,
and so are up to sign among the set $\{d_1\kdots d_m\}$ from Lemma \ref{le:9.5}.
Hence
\[
H^*(B_1\kdots B_s)\leq \HL .
\]
By \eqref{6.4b}, we have $\det \widehat{A}_v=\delta_v^{N{'}}$ for $v\in M_K$,
where $N':= \binom{n-1}{n-k-1}$.
Now a combination of \eqref{9.2} and the inequality
just established gives
\begin{eqnarray*}
H^*(\widehat{A}_1^{-1}\kdots \widehat{A}_s^{-1})&\leq&
\HL\cdot\prod_{v\in M_K}\max_{i\leq i\leq u}\|\delta_i\|_v^{-N{'}}
\\
&\leq& \HL^{1+N{'}(\binom{r}{n}-1)}\leq \HL^{R^n}.
\end{eqnarray*}
This proves our lemma.
\end{proof}

\begin{lemma}\label{le:11.5}
Suppose $Q$ satisfies \eqref{11.13a}, \eqref{11.13b} and put $N:=\binom{n}{k}$.
Let $\widehat{T}(Q)$ be the $\OQq$-vector space spanned by the
vectors $\widehat{\hh}_1(Q)\kdots\widehat{\hh}_{N-1}(Q)$. Then
\[
H_2(\widehat{T}(Q))\geq Q^{\delta/3R^n}.
\]  
\end{lemma}

\begin{proof}
Put $T:=T_k(Q)$, $\widehat{T}:=\widehat{T}(Q)$.
We have seen that $T$ is spanned by ${\bf h}_1(Q)\kdots {\bf h}_k(Q)$.
So we may apply Lemma \ref{le:6.1}. Now this lemma together with 
Lemma \ref{le:10.4} gives
$H_2(\widehat{T})=H_2(T)\geq Q^{\delta /3R^n}$.
\end{proof}

\section{A non-vanishing result}\label{12}

Let $N,m$ be integers $\geq 2$. 
Below, $\ii$, $\jj$ will denote $mN$-tuples
$(i_{hl}:\, h=1\kdots m,\, j=1\kdots N)$,
$(j_{hl}:\, h=1\kdots m,\, j=1\kdots N)$ of integers,
and $\ii \pm\jj$ their componentwise sum/difference.
 
We consider polynomials $P\in \OQq [\XX_1\kdots\XX_m ]=\OQq [X_{11}\kdots X_{mN}]$
in $m$ blocks of $N$ variables
$\XX_h=(X_{h1}\kdots X_{hN})$ ($h=1\kdots m)$.
Such a polynomial $P$ is expressed as
\begin{equation}\label{12.1}
P=\sum_{\jj} a(\jj )\XX^{\jj}\ \ \ \mbox{with } 
\XX^{\jj}:=\prod_{h=1}^m\prod_{l=1}^N X_{hl}^{j_{hl}}, 
\end{equation}
where the sum is over a finite set of tuples
$\jj\in\Zz_{\geq 0}^{mN}$, and where
$a (\jj )\in \OQq$.
For a polynomial $P$ as above and for
$\ii\in\Zz_{\geq 0}^{mN}$ we define
\[
P_{\ii}:=\left(
\prod_{h=1}^m\prod_{l=1}^N\frac{1}{i_{hl}!}\frac{\del^{i_{hj}}}{\del X_{hj}^{i_{hj}}}
\right)
P.
\]
Thus, if $P$ is given by \eqref{12.1}, then
\begin{eqnarray}\label{12.2}
P_{\ii}&=&\sum_{\jj} \binom{\ii +\jj}{\ii} a(\ii +\jj )\XX^{\jj},
\\
\nonumber
&&\mbox{where } 
\binom{\ii +\jj}{\ii}:=\prod_{h=1}^m\prod_{l=1}^N\binom{i_{hl}+j_{hl}}{i_{hl}}.
\end{eqnarray}
We say that $P\in \OQq [\XX_1\kdots \XX_m]$ is multihomogeneous of degree
$(r_1\kdots r_m)$ if it is homogeneous of degree $r_h$ in block $\XX_h$
for $h=1\kdots m$, i.e., if in \eqref{12.1} the sum is taken over tuples
$\jj\in\Zz_{\geq 0}^{mN}$ with $\sum_{l=1}^N j_{hl}=d_h$ for $h=1\kdots m$.

We write points in $\OQq^{mN}$ as $(\x_1\kdots \x_m)$, where
$\x_1\kdots\x_m\in \OQq^N$.

The height $H_2(P)$
of $P\in\OQq [\XX_1\kdots\XX_m]$ is defined as $H_2({\bf a}_P)$,
where ${\bf a}_P$ is a vector consisting of the non-zero coefficients of $P$.

Let $T$ be a finite dimensional $\OQq$-vector space and $B$ a positive
integer.
By a \emph{grid of size $B$} in $T$ we mean a set of the shape
\[
\left\{ 
\sum_{i=1}^d x_i{\bf a}_i:\, x_i\in\Zz,\, |x_i|\leq B\ \mbox{for } i=1\kdots d
\right\}
\]
where $d=\dim T$ and $\{ {\bf a}_1\kdots {\bf a}_d\}$ is any basis of $T$.

We recall \cite[Lemma 26]{Evertse1996}.
We note that this result was deduced from a sharp version
of Roth's Lemma, and ultimately goes back to Faltings' Product Theorem
\cite{Faltings1991}.

\begin{prop}\label{pr:12.1}
Let $m,N$ be integers $\geq 2$, $\eps$ a real with $0<\eps\leq 1$, and $r_1\kdots r_m$
positive integers such that
\begin{equation}\label{12.3}
\frac{r_h}{r_{h+1}}\geq\frac{2m^2}{\eps}\ \mbox{for } h=1\kdots m-1.
\end{equation}
Next, let $P$ be a non-zero polynomial in $\OQq [\XX_1\kdots\XX_m]$
which is homogeneous of degree $r_h$ in the block $\XX_h$ for $h=1\kdots m$,
and let $T_1\kdots T_m$ be $(N-1)$-dimensional linear subspaces of $\OQq^N$
such that
\begin{equation}\label{12.4}
H_2(T_h)^{r_h}\geq \left( e^{r_1+\cdots +r_m}H_2(P)\right)^{(N-1)(3m^2/\eps )^m}.
\end{equation}
Finally, for $h=1\kdots m$ let $\Gamma_h$ be a grid in $T_h$ of size $N/\eps$.
\\[0.2cm]
Then there are $\x_h\in\Gamma_h$ with $\x_h\not= {\bf 0}$ for $h=1\kdots m$
and $\ii\in\Zz_{\geq 0}^{mN}$ with
\begin{equation}\label{12.4a}
\sum_{h=1}^m\frac{1}{r_h}\left(\sum_{l=1}^N i_{hl}\right)\leq 2m\eps
\end{equation}
such that
\begin{equation}\label{12.5}
P_{\ii}(\x_1\kdots \x_m)\not= 0.
\end{equation}
\end{prop}

\section{Construction of the auxiliary polynomial}\label{13}

We start with recalling our main tools,
which are a version of Siegel's Lemma due to Bombieri and Vaaler 
and Hoeffding's inequality from probability theory.

For an algebraic number field $K$ we 
denote by $D_K$ the discriminant of $K$, and put
\[
C_K:=|D_K|^{1/2[K:\Qq ]}.
\]

\begin{lemma}\label{le:13.1}
Let $K$ be a number field, $U,V$ integers with $V>U>0$,
and $L_1\kdots L_U$ non-zero linear forms from $K[X_1\kdots X_V]^{{\rm lin}}$.
Then there exists $\x\in K^V\setminus\{ {\bf 0}\}$ such that
\begin{eqnarray}
\label{13.1}
&&L_1(\x )=0\kdots L_U(\x )=0,
\\[0.15cm]
\label{13.2}
&&H_2(\x )\leq V^{1/2}C_K\big( H_2(L_1)\cdots H_2(L_U)\big)^{1/(V-U)}.
\end{eqnarray}
\end{lemma}

\begin{proof}
This is a consequence of Bombieri and Vaaler \cite[Theorem 9]{Bombieri-Vaaler1983}.
\end{proof}

In the lemma below, all 
random variables under consideration are defined on a given probability
space with probability measure ${\rm Prob}$.
The expectation of a random variable $\Xx$ is denoted by $E(\Xx )$.

\begin{lemma}\label{le:13.2}
Let $\Xx_1\kdots\Xx_m$ be mutually independent random variables such that
\[
{\rm Prob} \big(\Xx_h\in [a_h,b_h]\big)=1,\ \ E(\Xx_h)=\mu_h\ \ \ \mbox{for } h=1\kdots m,
\]
where $a_h,b_h,\mu_h\in\Rr$, $a_h<b_h$ for $h=1\kdots m$.
Then for every $\eps >0$ we have
\begin{equation}\label{13.3}
{\rm Prob}\left(\sum_{h=1}^m(\Xx_h-\mu_h)\geq m\eps\right)\,\leq\,
\exp\left(-\frac{2m^2\eps^2}{\sum_{h=1}^m (b_h-a_h)^2}\right).
\end{equation}
\end{lemma}

\begin{proof}
See W. Hoeffding \cite[Theorem 2]{Hoeffding1963}.
\end{proof} 

For positive integers $m,N$ and a tuple of positive integers 
$\rr =(r_1\kdots r_m)$ 
define $\Ur$ to be the set of tuples
\[ 
\jj =(j_{hl}:\, h=1\kdots m,\, l=1\kdots N)\in\Zz_{\geq 0}^{mN}
\]
such that
\[
\sum_{l=1}^N j_{hl}=r_h\ \mbox{for } h=1\kdots m.  
\]
Put
\begin{equation}\label{13.7}
V:= \#\Ur =\prod_{h=1}^m\prod_{l=1}^N \binom{r_h+N-1}{N-1}.
\end{equation}
Using the inequality 
\[
\binom{a+b}{b}\leq\frac{(a+b)^{a+b}}{a^a b^b}=\left( 1+\frac{b}{a}\right)^a
\left( 1+\frac{a}{b}\right)^b
\leq \left( e\Big(1+\frac{b}{a}\Big)\right)^a
\]
for positive integers $a,b$, it follows that
\begin{equation}\label{13.8}
V\leq (eN)^{r_1+\cdots +r_m}.
\end{equation}

We deduce the following combinatorial lemma.

\begin{lemma}\label{le:13.3}
Let $N$ be a positive integer, $\rr =(r_1\kdots r_m)$ a tuple
of positive integers, $\eps ,\gamma$ reals with $0<\eps \leq 1$ and $\gamma >0$,
and $\widehat{\cc}_h=(\widehat{c}_{h1}\kdots \widehat{c}_{hN})$ ($h=1\kdots m$) tuples of reals such that
\begin{equation}\label{13.4}
|\widehat{c}_{hl}|\leq\gamma\ \ \mbox{for } h=1\kdots m,\,\, l=1\kdots N.
\end{equation}
Then the number of tuples $\jj=(j_{hl}:\, h=1\kdots m,\, l=1\kdots N)\in\Ur$
such that
\begin{equation}\label{13.5}
\sum_{h=1}^m \frac{1}{r_h}\left(\sum_{l=1}^N j_{hl}\widehat{c}_{hl}\right)\geq 
\frac{1}{N}\left(\sum_{h=1}^m\sum_{l=1}^N \widehat{c}_{hl}\right)+m\gamma\eps
\end{equation}
is at most
\begin{equation}\label{13.6}
e^{-m\eps^2/2}V.
\end{equation}
\end{lemma}

\begin{proof}
We assume without loss of generality that $\gamma =1$.
We view $\jj$ as a uniformly distributed random variable on $\Ur$,
i.e., each possible value of $\jj$ is given probability $1/V$.
Define random variables on $\Ur$ by
\[
\Xx_h:=\frac{1}{r_h}\sum_{l=1}^N j_{hl}\widehat{c}_{hl}\ \ (h=1\kdots m).
\]
Notice that $\Xx_1\kdots\Xx_m$ are mutually independent and for $h=1\kdots m$,
\begin{eqnarray*}
&&{\rm Prob}\big(\Xx_h\in [-1,1]\big)=1,\ \ \ \mbox{(by \eqref{13.4} and $\gamma =1$),}
\\
&&E(\Xx_h)=\mu_h:=\frac{1}{N}\sum_{l=1}^N \widehat{c}_{hl}.
\end{eqnarray*}
Now the number of tuples $\jj\in\Ur$ with \eqref{13.5} is precisely
\[
V\cdot{\rm Prob}\left(\sum_{h=1}^m(\Xx_h-\mu_h)\geq m\eps \right),
\]
and by Lemma \ref{le:13.2} this is at most $V\cdot e^{-m\eps^2/2}$.
\end{proof}

Let $K$ be an algebraic number field and $m,N, r_1\kdots r_m$ integers $\geq 2$.
We keep the notation introduced in Section \ref{12}.
In particular, by $\ii$ we denote an $mN$-tuple of non-negative integers
$\ii =(i_{hl}:\, h=1\kdots m,\, l=1\kdots N)$, and similarly
for $\jj$, $\kk$.
Further, $K[\XX_1\kdots\XX_m]$ denotes the ring of polynomials with coefficients in $K$
in the blocks of variables $\XX_h=(X_{h1}\kdots X_{hN})$ ($h=1\kdots m$).

We consider polynomials in this ring which are homogeneous of degree $r_h$
in $\XX_h$, for $h=1\kdots m$. In analogy to \eqref{12.1}, such a polynomial $P$
can be expressed as
\begin{equation}\label{13.9}
P=\sum_{\jj\in\Ur} a(\jj )\XX^{\jj}\ \mbox{with }
{\bf a}_P:=\big( a(\jj ):\, \jj\in\Ur\big)\in K^V.
\end{equation}

We prove a simple auxiliary result.

\begin{lemma}\label{le:13.4}
Let $P$ be a non-zero polynomial with \eqref{13.9}.
Further, let
\[
\widehat{L}_i=\sum_{j=1}^N \alpha_{ij}X_j\ (i=1\kdots N)
\] 
be linearly independent linear forms with coefficients in $K$
and 
\[
\big(\beta_{ij}\big)_{i,j=1\kdots N}=\Big(\big( \alpha_{ij}\big)_{i,j=1\kdots N}\Big)^{-1}
\]
the inverse of the coefficient matrix of $\widehat{L}_1\kdots\widehat{L}_N$.
Lastly, put
\[
C_v:=\max_{i,j=1\kdots N}\|\beta_{ij}\|_v\ \mbox{for } v\in M_K.
\]
Then for every $\ii\in\Zz_{\geq 0}^{mN}$ we have
\begin{eqnarray}\label{13.9a}
&&P_{\ii}=\sum_{\jj\in\Uri} d_{\ii ,\jj}({\bf a}_P)\prod_{h=1}^m\prod_{l=1}^N \widehat{L}_l(\XX_h)^{j_{hl}}
\\
\nonumber
&&\qquad
\mbox{with } \Uri:=\{ \jj\in\Zz_{\geq 0}^{mN}:\, \jj +\ii\in\Ur\},
\end{eqnarray}
where $d_{\ii ,\jj}$ is a linear form with coefficients in $K$ in $V$ variables satisfying
\begin{equation}\label{13.10}
\|d_{\ii ,\jj}\|_{v,1}\leq \left( (6N^2)^{s(v)}C_v\right)^{r_1+\cdots +r_m}
\ \ \mbox{for } \jj\in\Ur ,\, v\in M_K.
\end{equation}
\end{lemma}

\begin{proof}
Define new variables $Y_{hl}:=\widehat{L}_l(\XX_h)$ for $h=1\kdots m$, $l=1\kdots N$.
Then by \eqref{12.2},
\begin{eqnarray*}
P_{\ii}&=&\sum_{\jj\in\Uri}\binom{\ii +\jj}{\ii} a(\ii +\jj )\XX^{\jj}
\\
&=&\sum_{\jj\in\Uri} a(\ii +\jj )\prod_{h=1}^m\prod_{l=1}^N 
\left(\binom{i_{hl}+j_{hl}}{i_{hl}}\Big(\sum_{j=1}^N \beta_{lj}Y_{lj}\Big)^{j_{hl}}\right)
\\
&=:&\sum_{\jj\in\Uri} a(\ii +\jj )D_{\ii,\jj}({\bf Y}).
\end{eqnarray*}
Let $v\in M_K$. Then by \eqref{6.4} we have for $\jj\in\Uri$, 
on noting
$\binom{\ii +\jj}{\ii}\leq 2^{\sum_{h,l}(i_{hl}+ j_{hl})}=2^{r_1+\cdots +r_m}$,
\[
\| D_{\ii, \jj}\|_{v,1}\leq 
\binom{\ii +\jj}{\ii}^{s(v)}\big( N^{s(v)}C_v\big)^{\sum_{h,l}j_{hl}}
\\
\leq (2NC_v)^{r_1+\cdots +r_m}.
\]
Together with \eqref{6.4}, \eqref{13.8}, this implies for $\jj\in\Uri$,
\[
\| d_{\ii ,\jj}\|_{v,1}\leq V^{s(v)}\max_{{\bf k}\in\Ur}\| D_{\ii ,{\bf k}}\|_{v,1}
\leq (6N^2C_v)^{r_1+\cdots +r_m}.
\]  
\end{proof}

As before, let $\mathcal{L} ,\cc ,n,R,\delta$ satisfy \eqref{8.1}--\eqref{8.8}.
We fix $k\in\{ 1\kdots n-1\}$, and consider 
all reals $Q$ satisfying \eqref{11.13a}, \eqref{11.13b}.

Let $v_0$ be the place from \eqref{8.2}, 
and $\widehat{L}_i^{(v)}$ ($v\in M_K,\, i=1\kdots N$) the linear forms
and $\widehat{c_{iv}}$ ($v\in M_K\setminus\{ v_0\},\, i=1\kdots N$),
$\widehat{c}_{i,v_0}(Q)$ ($i=1\kdots N$) the reals from Section \ref{11}.

We want to construct a suitable non-zero polynomial $P$
of the shape \eqref{13.9}. 
The next lemma is our first step. For $v\in M_K$ we write
\begin{equation}\label{13.11}
P=\sum_{\jj\in\Ur} d_{\jj}^{(v)}({\bf a}_P)
\prod_{h=1}^m\prod_{l=1}^N \widehat{L}_l^{(v)}(\XX_h)^{j_{hl}}
\end{equation}
where $d_{\jj}^{(v)}$ is a linear form with coefficients in $K$ in $V$ variables
in the coefficient vector ${\bf a}_P$ of $P$.

\begin{lemma}\label{le:13.5}
Let $S_0$ be a subset of
\[
S_1:=\{ v\in M_K:\, {\bf c}_v:=(c_{1v}\kdots c_{nv})\not= 0\}
\]
and put $s_0:=\# S_0$. 
\\
Let $\eps$ be a real with $0<\eps <1$,
$m$ an integer with
\begin{equation}\label{13.13}
m\, >\, 2\eps^{-2}\log (2s_0+2)
\end{equation}
and $r_1\kdots r_m$ positive integers.
\\
Lastly, let $Q_1\kdots Q_m$ be reals with \eqref{11.13a}, \eqref{11.13b}.
\\[0.15cm]
Then there exists a non-zero polynomial $P$ of the type \eqref{13.9}
with the following properties:
\\[0.15cm]
(i) For every $v\in S_0$ and each $\jj\in\Ur$ with
\begin{equation}\label{13.14}
\sum_{h=1}^m \frac{1}{r_h}\left( \sum_{l=1}^N j_{hl}\widehat{c}_{lv}\right)\geq
mn\eps\cdot 
\left(\max_{1\leq i\leq n} c_{iv}\right)
\end{equation}
we have
\begin{equation}\label{13.15}
d_{\jj}^{(v)}({\bf a}_P)=0.
\end{equation}
(ii) For each $\jj\in\Ur$ with
\begin{equation}\label{13.16}
\sum_{h=1}^m \frac{1}{r_h}\left( \sum_{l=1}^N j_{hl}\widehat{c}_{l,v_0}(Q_h)\right)\geq
-\frac{m\delta}{nN}\,+mn\eps 
\end{equation}
we have
\begin{equation}\label{13.17}
d_{\jj}^{(v_0)}({\bf a}_P)=0.
\end{equation}
(iii) For the height of $P$ we have
\begin{equation}\label{13.18}
H_2(P)\leq C_K\big(2^{3n}\HL^{R^n}\big)^{r_1+\cdots +r_m}.
\end{equation}
\end{lemma}

\noindent
We recall here that by \eqref{8.5} the set $S_1$ is finite
and that the place $v_0$ given by \eqref{8.2} does not belong to $S_1$.

\begin{proof}
We prove that there exists a non-zero polynomial $P$ of the type \eqref{13.9}
such that for every $v\in S_0$, and each $\jj\in\Zz_{\geq 0}^{mN}$ with 
\begin{equation}\label{13.14a}
\sum_{h=1}^m \frac{1}{r_h}\left( \sum_{l=1}^N j_{hl}\widehat{c}_{lv}\right)\geq
\left(\frac{m}{N}\sum_{l=1}^N \widehat{c}_{lv}\right)+mn\eps\cdot 
\left(\max_{1\leq i\leq n} c_{iv}\right)
\end{equation}
we have \eqref{13.15}, and such that for each $\jj\in\Zz_{\geq 0}^{mN}$
with
\begin{equation}\label{13.16a}
\sum_{h=1}^m \frac{1}{r_h}\left( \sum_{l=1}^N j_{hl}\widehat{c}_{l,v_0}(Q_h)\right)\geq
\left(\frac{1}{N}\sum_{h=1}^m\sum_{l=1}^N 
\widehat{c}_{l,v_0}(Q_h)\right)+mn\eps 
\end{equation}
we have \eqref{13.17}.
This suffices, since by \eqref{11.14}, \eqref{11.16}, the conditions
\eqref{13.14}, \eqref{13.16} imply \eqref{13.14a}, \eqref{13.16a}. 

We may view \eqref{13.15} with \eqref{13.14a} and \eqref{13.17} with \eqref{13.16a}
as a system of linear equations
in the unknown vector ${\bf a}_P\in K^V$, where $V=\#\Ur$.
By \eqref{11.14a}, \eqref{11.17}, Lemma \ref{le:13.3}, and assumption \eqref{13.13},
the number of equations, i.e., the number of $\jj$ with \eqref{13.14a}, \eqref{13.16a},
is
\[
U\leq (s_0+1)Ve^{-m\eps^2/2}\leq \half V.
\]
Combining Lemma \ref{le:11.4} with Lemma \ref{le:13.4} gives us
\[
H_2(d_{\jj}^{(v)})\leq \big( 6N^2\HL^{R^n}\big)^{r_1+\cdots +r_m}
\]
for $v\in S_0\cup\{ v_0\}$, $\jj\in\Ur$. 
Now Lemma \ref{le:13.1} implies
that there is a non-zero ${\bf a}_P\in K^V$ with
\[
H_2({\bf a}_P)\leq C_KV^{1/2}\left( 6N^2\HL^{R^n}\right)^{(r_1+\cdots +r_m)U/(V-U)}.
\]
By inserting \eqref{13.8} and $N=\binom{n}{k}\leq 2^{n-1}$ we arrive at
\begin{eqnarray*}
H_2(P)=H_2({\bf a}_P)&\leq& 
C_K\left( 6e^{1/2}N^{5/2}\HL^{R^n}\right)^{r_1+\cdots +r_m}
\\
&\leq& C_K\left( 2^{3n}\HL^{R^n}\right)^{r_1+\cdots +r_m}.
\end{eqnarray*}
Our Lemma follows.
\end{proof}

The next proposition lists the properties of our final auxiliary polynomial.
For $v\in M_K$, $\ii\in\Zz_{\geq 0}^{mN}$, we write,
analogously to \eqref{13.9a},
\begin{equation}\label{13.12}
P_{\ii}=\sum_{\jj\in\Uri} d_{\ii ,\jj}^{(v)}({\bf a}_P)
\prod_{h=1}^m\prod_{l=1}^N \widehat{L}_l^{(v)}(\XX_h)^{j_{hl}}
\end{equation}
where $\Uri =\{ \jj\in\Zz_{\geq 0}^{mN}:\, \ii +\jj\in\Ur\}$ and where 
$d_{\ii ,\jj}^{(v)}$ is a linear form in $V$ variables with coefficients in $K$.

\begin{prop}\label{pr:13.6}
Let $\eps$ be a real with $0<\eps \leq 1$, $m$ an integer with
\begin{equation}\label{13.19}
m\geq 2n\eps^{-2}\log (4R/\eps )
\end{equation}
and $r_1\kdots r_m$ positive integers.
\\
Further, let $Q_1\kdots Q_m$ be reals with  \eqref{11.13a}, \eqref{11.13b}.
\\[0.15cm]
Then there exists a non-zero polynomial $P$ of the type \eqref{13.9}
with the following properties:
\\[0.15cm]
(i) For every $v\in M_K\setminus\{ v_0\}$, each tuple $\ii\in\Zz_{\geq 0}^{mN}$
with
\begin{equation}\label{13.20} 
\sum_{h=1}^m \frac{1}{r_h}\left(\sum_{l=1}^N i_{hl}\right)\leq 2m\eps
\end{equation}
and each $\jj\in\Uri$ with
\begin{equation}\label{13.22}
\sum_{h=1}^m \frac{1}{r_h}\left(\sum_{l=1}^N \widehat{c}_{lv}j_{hl}\right)
> 4mn\eps\max_{1\leq i\leq n}c_{iv}
\end{equation}
we have
\begin{equation}\label{13.23}
d_{\ii ,\jj}^{(v)}({\bf a}_P)=0.
\end{equation}
(ii) For each $\ii$ with \eqref{13.20} and each $\jj\in\Uri$ with
\begin{equation}\label{13.24}
\sum_{h=1}^m \frac{1}{r_h}\left(\sum_{l=1}^N \widehat{c}_{l,v_0}(Q_h)j_{hl}\right)
> -\frac{m\delta}{nN}\, +4mn\eps
\end{equation}
we have
\begin{equation}\label{13.25}
d_{\ii ,\jj}^{(v_0)}({\bf a}_P)=0.
\end{equation}
(iii) For the height of $P$ we have
\begin{equation}\label{13.26}
H_2(P)\leq C_K\left( 2^{3n}\HL^{R^n}\right)^{r_1+\cdots +r_m}.
\end{equation}
(iv) For all $\ii\in\Zz_{\geq 0}^{mN}$ we have
\begin{equation}\label{13.27}
\prod_{v\in M_K}\left(\max_{\jj\in\Uri}\|d_{\ii ,\jj}^{(v)}({\bf a}_P)\|_v\right)
\leq C_K\left( 2^{6n}\HL^{2R^n}\right)^{r_1+\cdots +r_m}.
\end{equation}
\end{prop}

\begin{proof}
We construct a subset $S_0$ of
\[
S_1:=\{ v\in M_K:\, {\bf c}_v=(c_{1v}\kdots c_{nv})\not= {\bf 0}\}
\]
and apply Lemma \ref{le:13.5} with this set.
The set $S_0$ is obtained by
dividing $S_1$ into subsets and picking one element
from each subset.
For $v\in M_K$, we put $\gamma_v:=\max_{1\leq i\leq n} c_{iv}$.

First, we divide $S_1$ into $t_1$ subsets
$S_{11}\kdots S_{1,t_1}$ in such a way that two places $v_1,v_2$ belong to
the same subset if and only if
\[
L_i^{(v_1)}=L_i^{(v_2)}\ \mbox{for } i=1\kdots n.
\]
By \eqref{8.4}, we have $t_1\leq R^n$.

We further subdivide the subsets $S_{1j}$. Let $j\in\{ 1\kdots t_1\}$.
Divide the cube $[-1,1]^n$ into $t_2:=\big( [2/\eps ]+1\big)^n$
small subcubes of sidelength
\[
\frac{2}{[2/\eps ]+1}\leq \eps .
\]
Now divide $S_{1,j}$ into $t_2$ subsets $S_{1j1}\kdots S_{1j,t_2}$
such that two places $v_1,v_2$ belong to the same subset if
the two points
\[
\left( \frac{c_{1,v_1}}{\gamma_{v_1}}\kdots \frac{c_{n,v_1}}{\gamma_{v_1}}\right),\ \ 
\left( \frac{c_{1,v_2}}{\gamma_{v_2}}\kdots \frac{c_{n,v_2}}{\gamma_{v_2}}\right)
\]
belong to the same small subcube. In this way, we have divided $S_1$ into
\[
t_1t_2\leq R^n\left([2/\eps ]+1\right)^n\leq\left( 3R/\eps\right)^n
\]
subsets. Let $S_0$ consist of one element from each of the subsets.
Thus,
\begin{equation}\label{13.28}
s_0:=\# S_0\leq (3R/\eps )^n.
\end{equation}
Further, for each $v\in S_1$, there is $v_1\in S_0$ with
\[
L_i^{(v)}=L_i^{(v_1)},\ \ \ \
\left|\frac{c_{iv}}{\gamma_v}-\frac{c_{i,v_1}}{\gamma_{v_1}}\right|\leq \eps\ \ 
\mbox{for } i=1\kdots n.
\]
This implies that for every $v\in S_1$ there is $v_1\in S_0$ such that
\begin{eqnarray}
\label{13.29}
&&\widehat{L}_l^{(v)}=\widehat{L}_l^{(v_1)}\ \mbox{for } l=1\kdots N,
\\
\label{13.30}
&&\left|\frac{\widehat{c}_{lv}}{n\gamma_v}-\frac{\widehat{c}_{l,v_1}}{n\gamma_{v_1}}\right|\leq
\eps\ \mbox{for } l=1\kdots N.
\end{eqnarray}

We apply Lemma \ref{le:13.5} with the subset $S_0$ constructed above.
Condition \eqref{13.13} of this lemma is satisfied, in view 
of our assumption \eqref{13.19} on $m$, and of \eqref{13.28}.
Let $P$ be the non-zero polynomial from Lemma \ref{le:13.5}.
We show that this polynomial has all properties listed in our Proposition.

To prove (i), we first show that for every $\jj\in\Ur$, $v\in M_K\setminus\{ v_0\}$ with
\begin{equation}\label{13.31}
\sum_{h=1}^m\frac{1}{r_h}\left(\sum_{l=1}^N \widehat{c}_{lv}j_{hl}\right)\,>
2mn\eps\gamma_v
\end{equation}
we have 
\begin{equation}\label{13.32}
d_{\jj}^{(v)}({\bf a}_P)=0.
\end{equation}
For $v\in M_K\setminus (S_1\cup\{ v_0\})$ 
we have $c_{iv}=0$ for $i=1\kdots n$, whence $\gamma_v=0$ and $\widehat{c}_{lv}=0$
for $l=1\kdots N$, so there are no $\jj$ with \eqref{13.31}.
For $v\in S_0$ we have \eqref{13.32} for all $\jj$ with \eqref{13.14}, 
and so certainly for all $\jj$ with the weaker condition \eqref{13.31}.
Finally, let $v\in S_1\setminus S_0$ and take $\jj\in\Ur$ with
\eqref{13.31}. 
Take $v_1\in S_0$ with \eqref{13.29}, \eqref{13.30}.
Condition \eqref{13.29} implies that $d_{\jj}^{(v)}({\bf a}_P)=d_{\jj}^{(v_1)}({\bf a}_P)$,
hence it suffices to show that $d_{\jj}^{(v_1)}({\bf a}_P)=0$.
Now condition \eqref{13.31} together with \eqref{13.30} implies
\[
\sum_{h=1}^m\frac{1}{r_h}
\left(\sum_{l=1}^N \frac{\widehat{c}_{l,v_1}}{n\gamma_{v_1}}\cdot j_{hl}\right)
\geq
\sum_{h=1}^m\frac{1}{r_h}\left(\sum_{l=1}^N \frac{\widehat{c}_{lv}}{n\gamma_{v}}
\cdot j_{hl}\right)-\eps\displaystyle{\sum_{h=1}^m\sum_{l=1}^N\frac{j_{hl}}{r_h}}
\,> m\eps .
\]
Hence $\jj ,v_1$ satisfy \eqref{13.14} and so $d_{\jj}^{(v_1)}({\bf a}_P)=0$ by
Lemma \ref{le:13.5}. This shows \eqref{13.32} for $v\in M_K\setminus\{ v_0\}$.

We now prove (i). Let $\ii\in\Zz_{\geq 0}^{mN}$ be a tuple with \eqref{13.20}
and let $v\in M_K\setminus\{ v_0\}$. Using expression \eqref{13.11} for $P$,
we infer that $P_{\ii}$ is a $K$-linear combination of polynomials
\[
\sum_{\jj\in\Ur} d_{\jj}^{(v)}({\bf a}_P)\binom{\jj}{\kk}
\prod_{h=1}^M\prod_{l=1}^N \widehat{L}_{l}(\XX_h)^{j_{hl}-k_{hl}}
\]
taken over tuples $\kk\in\Zz_{\geq 0}^{mN}$ with
\begin{equation}\label{13.33}
\sum_{h=1}^m\frac{1}{r_h}\left(\sum_{j=1}^N k_{hl}\right)\leq 2m\eps .
\end{equation}
Hence if $\jj\in\Uri$,
then $d_{\ii ,\jj}^{(v)}({\bf a}_P)$ is a $K$-linear
combination of terms $d_{\jj +\kk}^{(v)}({\bf a}_P)$, over tuples $\kk$
with \eqref{13.33}. Now take $\jj\in\Uri$ and suppose that $\jj$ satisfies 
\eqref{13.22}. 
Then for all $\kk$ with \eqref{13.33} we have
$\jj +\kk\in\Ur$ and moreover, by \eqref{11.14a},
\[  
\sum_{h=1}^m \frac{1}{r_h}\left(\sum_{l=1}^N \widehat{c}_{lv}(j_{hl}+k_{hl})\right)
> 4mn\eps\gamma_v-2mn\eps\gamma_v=2mn\eps\gamma_v.
\]
i.e., $\jj +\kk$ satisfies \eqref{13.31}. So for all
$\kk$ with \eqref{13.33} we have that
$\jj +\kk$ satisfies \eqref{13.32}, i.e., $d_{\jj +\kk}^{(v)}({\bf a}_P)=0$.
This implies that $d_{\ii ,\jj}^{(v)}({\bf a}_P)=0$.
This proves (i).

The proof of (ii) follows the same lines, using part (ii) of Lemma \ref{le:13.5}
instead of \eqref{13.32}.
(iii) is merely a copy of part (iii) of Lemma \ref{le:13.5}.

It remains to prove (iv). Let $\ii$ satisfy \eqref{13.20} and let $\jj\in\Uri$.
Then by Lemma \ref{le:13.4},
\[
\| d_{\ii ,\jj}^{(v)}\|_{v,1}\leq \big(\, (6N^2)^{s(v)}C_v \big)^{r_1+\cdots +r_m},
\]
and so
\[
\| d_{\ii ,\jj}^{(v)}({\bf a}_P)\|_v\leq \| d_{\ii ,\jj}^{(v)}\|_{v,1}\cdot \|{\bf a}_P\|_v
\leq 
\|{\bf a}_P\|_v\cdot\big(\, (6N^2)^{s(v)}C_v \big)^{r_1+\cdots +r_m}
\]
for $v\in M_K$, where by Lemma \ref{le:11.4} we have
\[
\prod_{v\in M_K} C_v\leq\HL^{R^n} .
\]
By taking the product over $v\in M_K$, using (iii), $N\leq 2^{n-1}$
we obtain
\begin{eqnarray*}
\prod_{v\in M_K} \max_{\jj\in\Uri}\| d_{\ii ,\jj}^{(v)}({\bf a}_P)\|_v &\leq& 
C_K\big( 2^{3n}\HL^{R^n}\cdot 6N^2\HL^{R^n}\big)^{r_1+\cdots +r_m}
\\
&\leq& C_K\big( 2^{6n}\HL^{2R^n}\big)^{r_1+\cdots +r_m}.
\end{eqnarray*}
This proves (iv). 
\end{proof}

\section{Proof of Theorem \ref{th:8.1}}\label{14}

We keep the notation and definitions from the previous sections.
Assume that Theorem \ref{th:8.1} is false. 
Define the following parameters:
\begin{equation}\label{14.1}
\eps := \frac{\delta}{11n^2 2^{n-1}},\ \ m:=\left[ 2n\eps^{-2}\log (4R/\eps )\right]+1. 
\end{equation}
Notice that
\begin{equation}\label{14.1a}
nm\leq n+\,2\cdot 11^2 n^6 2^{2n-2}\delta^{-2}\log (4\cdot 11n^2 2^{n-1} R/\delta )\leq m_2.
\end{equation}
Hence by Lemma \ref{le:10.3}, there exist $k\in\{ 1\kdots n-1\}$,
and reals $Q_1\kdots Q_m$, such that
\begin{eqnarray}
\label{14.2}
&\hspace*{0.5cm} Q_1\geq C_2,&
\\
\label{14.3}
&\hspace*{0.5cm} Q_{h+1}>Q_h^{\omega_2}\ \ (h=1\kdots m-1),&
\\
\label{14.4}
&\hspace*{0.5cm} \lambda_1(Q_h)\leq Q_h^{-\delta},\ \lambda_k(Q_h)\leq Q_h^{-\delta /(n-1)}\lambda_{k+1}(Q_h)
\ \ (h=1\kdots m).&
\end{eqnarray}
Put
\[
N:=\binom{n}{k}.
\]
For $h=1\kdots m$, let $\widehat{\hh}_{h1}:=\widehat{\hh}_1(Q_h)\kdots 
\widehat{\hh}_{h,N-1}:=\widehat{\hh}_{N-1}(Q_h)$ be linearly independent
vectors from $\OQq^N$ satisfying \eqref{11.13} with $Q=Q_h$.
By the remark following \eqref{11.13}, we may take for the field $E$ any finite extension
of $K$ containing the coordinates of $\widehat{\hh}_{hj}$ for $h=1\kdots m$,
$j=1\kdots N-1$. Thus, we have for $h=1\kdots m$, $l=1\kdots N$, $j=1\kdots N-1$,
\begin{equation}\label{14.5}
\left\{
\begin{array}{l}
\| \widehat{L}_l^{(w)}(\widehat{\hh}_{hj})\|_w\leq Q_h^{\widehat{c}_{lw}}\ 
(w\in M_E,\, w\nmid v_0),
\\[0.15cm]
\|\widehat{L}_l^{(w)}(\widehat{\hh}_{hj})\|_w\leq Q_h^{\widehat{c}_{lw}(Q_h)}\ (w\in M_E,\, w\mid v_0).
\end{array}\right. 
\end{equation}
For $h=1\kdots m$, denote by $\widehat{T}_h$ the $\OQq$-vector space generated by
$\widehat{\hh}_{h1}$$\kdots$$\widehat{\hh}_{h,N-1}$, and define the grid
\begin{equation}\label{14.grid}
\Gamma_h:=\left\{ \sum_{j=1}^N x_j\widehat{\hh}_{hj}:\, x_j\in\Zz ,\, |x_j|\leq N/\eps\ \mbox{for }
j=1\kdots N-1\right\}.
\end{equation}

Now choose a positive integer $r_1$ such that
\[
r_1>\frac{\eps^{-1}\log Q_m}{\log Q_1}
\]
and then integers $r_2\kdots r_m$ such that
\[
\frac{r_1\log Q_1}{\log Q_h}\leq r_h <1+\, \frac{r_1\log Q_1}{\log Q_h}\ \mbox{for } h=2\kdots m.
\]
Thus, $r_1\kdots r_m$ are all positive integers with
\begin{equation}\label{14.6}
Q_1^{r_1}\leq Q_h^{r_h}< Q_1^{r_1(1+\eps )}\ \mbox{for } h=2\kdots m.
\end{equation}
Further, by choosing $r_1$ sufficiently large as we may, we can guarantee that
\begin{equation}\label{14.7}
1.1^{r_1}>C_K.
\end{equation}

With our choice of $m$ in \eqref{14.1}, there exists a non-zero polynomial
$P$ with the properties listed in Proposition \ref{pr:13.6}.
We apply our non-vanishing result Proposition
\ref{pr:12.1} to $P$.
We verify the conditions of that proposition. 
Condition \eqref{12.3} is satisfied since by \eqref{14.6}, \eqref{14.3}, 
\eqref{8.definitions}, \eqref{14.1a},
\[
\frac{r_{h+1}}{r_h}\geq (1+\eps )^{-1}\frac{\log Q_{h+1}}{\log Q_h}
\geq (1+\eps )^{-1}m_2^{5/2}\geq 2m^2/\eps .
\]
\eqref{12.4} follows by combining the lower bound for $H_2(\widehat{T}_h)$
from Lemma \ref{le:11.5} with the lower bound $Q_1\geq C_2$ from \eqref{14.2}
and the upper bound for $H_2(P)$ from \eqref{13.26}. 
More precisely, we have for $h=1\kdots m$,
\begin{eqnarray*}
H_2(\widehat{T}_h)^{r_h}&\geq& Q_h^{r_h\delta/3R^n}\geq Q_1^{r_1\delta/3R^n}\ \ 
\mbox{by Lemma \ref{le:11.5}, \eqref{14.6}}
\\
&\geq& C_2^{r_1\delta/3R^n}=(2\HL )^{r_1\cdot m_2^{2m_2}\delta/3R^n}\ \ 
\mbox{by \eqref{14.2}, \eqref{8.definitions}}
\\
&\geq& \left((2\HL )^{(nm)^{2nm}\delta/3mR^n}\right)^{r_1+\cdots +r_m}
\ \ \mbox{by \eqref{14.1a}}
\\
&\geq&
\left(e\cdot 2^{4n}\HL^{R^n}\right)^{(N-1)(3m^2/\eps )^m(r_1+\cdots +r_m)}
\ \ \mbox{by \eqref{14.1}}
\\
&\geq& \left(e^{r_1+\cdots +r_m}H_2(P)\right)^{(N-1)(3m^2/\eps )^m}\ \ 
\mbox{by \eqref{13.26}, \eqref{14.7},}
\end{eqnarray*}
which is condition \eqref{12.4}.
 
Now we conclude from Proposition \ref{pr:12.1} that there exist a tuple
$\ii\in\Zz_{\geq 0}^{mN}$ such that
\[
\sum_{h=1}^m \frac{1}{r_h}\left(\sum_{l=1}^N i_{hl}\right)\leq 2m\eps
\]
and non-zero points ${\bf x}_h\in\Gamma_h$ ($h=1\kdots m$), such that
\[
P_{\ii}({\bf x}_1\kdots {\bf x}_m)\not= 0.
\]

We finish by showing that $\prod_{w\in M_E}\| P_{\ii}(\x_1\kdots\x_m)\|_w<1$.
Then by the Product Formula, $P_{\ii}(\x_1\kdots\x_m)=0$
which is against what we just proved.
Thus, our assumption that Theorem \ref{th:8.1} is false leads
to a contradiction.

We express $P_{\ii}$ as in \eqref{13.12} for $v\in M_K$.
We define in the usual manner, where in all cases $w\in M_E$
and $v$ is the place of $K$ below $w$:
\begin{eqnarray*}
&&\widehat{L}_l^{(w)}:=\widehat{L}_l^{(v)}\ \ (l=1\kdots N),
\\
&&\widehat{c}_{lw}:=d(w|v)\widehat{c}_{lv}\ \ 
(w\nmid v_0,\,\, l=1\kdots N),
\\
&&\widehat{c}_{lw}(Q_h):=d(w|v_0)\widehat{c}_{l,v_0}(Q_h)\ \ 
(w|v_0 ,\,\, l=1\kdots N),
\\ 
&&d_{\ii ,\jj}^{(w)}({\bf a}_P):=d_{\ii ,\jj}^{(v)}({\bf a}_P)\ \
(\jj\in\Uri ),
\end{eqnarray*}
and also
\[
\gamma_w:=\max_{1\leq i\leq n} c_{iw}.
\]
Then $\gamma_w=d(w|v)\max_{1\leq i\leq n} c_{iv}$ if $v$ is the place
of $K$ below $w$ and moreover, by \eqref{8.7} and $\sum_{w|v} d(w|v)=1$
for $v\in M_K$,
\begin{equation}\label{14.8}
\sum_{w\in M_E}\gamma_w \leq 1.
\end{equation} 
Now \eqref{13.12}, \eqref{13.22}, \eqref{13.24} imply that for $w\in M_E$
we have 
\begin{equation}\label{14.8a}
P_{\ii}=\sum_{\jj\in\Uw} d_{\ii ,\jj}^{(w)}({\bf a}_P)
\prod_{h=1}^m\prod_{l=1}^N \widehat{L}_l^{(w)}(\XX_h)^{j_{hl}},
\end{equation}
where for $w\in M_E$ with $w\nmid v_0$, $\Uw$ is the set of $\jj\in\Uri$
with
\begin{equation}\label{14.8b}
\sum_{h=1}^m \frac{1}{r_h}\left(\sum_{l=1}^N \widehat{c}_{lw}j_{hl}\right)
\leq 4mn\eps\gamma_w
\end{equation}
and for $w\in M_E$ with $w|v_0$, $\Uw$ is the set of $\jj\in\Uri$ with
\begin{equation}\label{14.8c}
\sum_{h=1}^m \frac{1}{r_h}\left(\sum_{l=1}^N \widehat{c}_{l,w}(Q_h)j_{hl}\right)
\leq  d(w|v_0)\left(-\frac{m\delta}{nN}\, +4mn\eps \right).
\end{equation}
Further, by \eqref{13.27}, \eqref{14.7} we have
\begin{eqnarray}\label{14.9}
&&\prod_{w\in M_E} A_w\leq \left( 2^{7n}\HL^{2R^n}\right)^{r_1+\cdots +r_m},
\\
\nonumber
&&
\qquad\mbox{with } A_w:=\max_{\jj\in\Uw}\| d_{\ii ,\jj}^{w)}({\bf a}_P)\|_w\
\mbox{for } w\in M_E.
\end{eqnarray}
Finally, we observe that by \eqref{14.5} we have for the points
$\x_h\in\Gamma_h$ ($h=1\kdots N$) and for $l=1\kdots N$,
\begin{equation}\label{14.10}
\left\{
\begin{array}{l}
\| \widehat{L}_l^{(w)}(\x_h)\|_w\leq N^{s(w)}Q_h^{\widehat{c}_{lw}}\ 
(w\in M_E,\, w\nmid v_0),
\\[0.15cm]
\|\widehat{L}_l^{(w)}(\x_h)\|_w\leq Q_h^{\widehat{c}_{lw}(Q_h)}\ (w\in M_E,\, w\mid v_0).
\end{array}\right. 
\end{equation}
where we have used that $w$ with $w|v_0$ is non-archimedean.

First, take $w\in M_E$ with $w\nmid v_0$. Then we have, 
in view of \eqref{14.8a}, \eqref{13.8}, \eqref{14.6}, \eqref{11.14a},
\eqref{14.8b},
\begin{eqnarray*}
&&
\| P_{\ii}(\x_1\kdots\x_m)\|_w
\leq V^{s(w)}A_w\cdot\max_{\jj\in\Uw}
\prod_{h=1}^m\prod_{l=1}^N \|\widehat{L}_l^{(w)}(\x_h)^{j_{hl}}\|_w
\\
&&\qquad
\leq A_w (eN^2)^{s(w)(r_1+\cdots +r_m)}\prod_{h=1}^m Q_h^{\sum_{l=1}^N \widehat{c}_{lw}j_{lw}}
\\
&&\qquad\leq
A_w(eN^2)^{s(w)(r_1+\cdots +r_m)}(Q_1^{r_1})^{\alpha_w}
\end{eqnarray*}
with
\begin{eqnarray*}
\alpha_w&\leq& \sum_{h=1}^m
\frac{1}{r_h}\left(\sum_{l=1}^N \widehat{c}_{lw}j_{lw}\right)
+\eps m\max_l |\widehat{c}_{lw}|
\\
&\leq& 
5\gamma_wmn\eps .
\end{eqnarray*}
So altogether, we have for $w\in M_E$ with $w\nmid v_0$,
\begin{equation}\label{14.11}
\| P_{\ii}(\x_1\kdots\x_m)\|_w\leq A_w(eN^2)^{s(w)(r_1+\cdots +r_m)}
(Q_1^{mr_1})^{5\gamma_wn\eps}.
\end{equation}
In a similar fashion we find for $w\in M_E$ with $w|v_0$,
using \eqref{14.8a}, \eqref{14.6}, \eqref{11.17}, \eqref{14.8c},
noting that now we don't have a factor $(eN^2)^{s(w)(r_1+\cdots +r_m)}$
since $w$ is non-archimedean,
\[
\| P_{\ii}(\x_1\kdots\x_m)\|_w\leq A_w(Q_1^{r_1})^{\alpha_w}
\]
with
\begin{eqnarray*}
\alpha_w&\leq& \sum_{h=1}^m
\frac{1}{r_h}\left(\sum_{l=1}^N \widehat{c}_{lw}(Q_h)j_{lw}\right)
+\eps m\max_{h,l} |\widehat{c}_{lw}(Q_h)|
\\
&\leq& 
d(w|v_0)\left(-\frac{m\delta}{nN}\, +5mn\eps\right).
\end{eqnarray*}
This gives for $w\in M_E$ with $w|v_0$,
\begin{equation}\label{14.12}
\|P_{\ii}(\x_1\kdots\x_m)\|_w\leq 
A_w\big(Q_1^{mr_1}\big)^{d(w|v_0)(-(\delta /nN)+5n\eps)}.
\end{equation} 
Now taking the product over $w\in M_E$, combining \eqref{14.11}, 
\eqref{14.12}, \eqref{14.9}, \eqref{14.8}, $\sum_{w|v_0} d(w|v_0)=1$, 
we obtain
\[
\prod_{w\in M_E}\| P_{\ii}(\x_1\kdots\x_m )\|_w\leq 
(eN^2\cdot 2^{7n}\HL^{2R^n})^{r_1+\cdots +r_m}
\big( Q_1^{mr_1}\big)^{10n\eps-\delta/nN}.
\]
By our choice of $\eps$ in \eqref{14.1}, and the inequalities
$n\geq 2$, $N\leq 2^{n-1}$,
the exponent on $Q_1^{mr_1}$ is $\leq -\delta /(11n\cdot 2^{n-1})$.
Together with \eqref{14.2} this implies
\[
\prod_{w\in M_E}\| P_{\ii}(\x_1\kdots\x_m )\|_w\leq 
\left(2^{9n}\HL^{2R^n}\cdot Q_1^{-\delta /11n\cdot 2^{n-1}}\right)^{mr_1}<1,
\]
as required. This completes the proof of Theorem \ref{th:8.1}.

\section{Construction of a filtration}\label{15}

\noindent
We construct a vector space filtration
which is an adaptation of the Harder-Narasimhan filtration
constructed in \cite{Faltings-Wuestholz1994}.

Let $K\subset\OQq$ be an algebraic number field, and $n$ an integer
which we now assume $\geq 1$ instead of $\geq 2$.
Further,
let $\LL =(L_i^{(v)}:\, v\in M_K ,\, i=1\kdots n)$ be a tuple of linear forms
and $\cc =(c_{iv}:\, v\in M_K,\, i=1\kdots n)$ a tuple of reals,
satisfying \eqref{2.10a}--\eqref{2.7a}. 

Let $w_v=w_{\mathcal{L},\cc ,v}$ ($v\in M_K$) be the local weight
functions on the collection of linear subspaces of $\OQq^n$,
defined by \eqref{2.weightv-space}. Then the global weight function
is given by $w=w_{\mathcal{L},\cc}=\sum_{v\in M_K} w_v$.
 
We give some convenient expressions for the local weights $w_v$. 
For $v\in M_K$ we reorder the indices $1\kdots n$ in such a way that
\begin{equation}\label{15.2}
c_{1v}\leq \cdots \leq c_{nv}\ \mbox{for } v\in M_K.
\end{equation}
Let $U$ be a $k$-dimensional linear subspace of $\OQq^n$.
Let $v\in M_K$.
Define 
\begin{equation}\label{15.1}
\left\{\begin{array}{l}
\mbox{$I_v(U):=\emptyset$ if $k=0$,} 
\\
\mbox{$I_v(U):=\{ i_1(v)\kdots i_k(v)\}$ if $k>0$,}
\end{array}\right. 
\end{equation}
where
$i_1(v)$ is the smallest index $i\in\{ 1\kdots n\}$
such that 
$L_i^{(v)}|_U\not =0$, and for $l=2\kdots k$,  
$i_l(v)$ is the smallest index $i>i_{l-1}(v)$ in $\{ 1\kdots n\}$ 
such that $L_{i_1(v)}^{(v)}|_U\kdots L_{i_{l-1}(v)}^{(v)}|_U,\, L_i^{(v)}|_U$ 
are linearly independent. 
Then
\begin{equation}\label{15.6}
w_v(U)=\sum_{i\in I_v(U)} c_{iv}.
\end{equation}
It is not difficult to show that $I_v(U_1)\subseteq I_v(U_2)$ if 
$U_1$ is a linear subspace of $U_2$.

Define the linear subspaces of $\OQq^n$,
\begin{eqnarray*}
&&U_{0v}:=\OQq^n ,
\\
\nonumber 
&&U_{iv}:=\{ \x\in \OQq^n:\, L_1^{(v)}(\x)=\cdots = L_i^{(v)}(\x )=0\} \ \ 
(v\in M_K,\,\,i=1\kdots n).
\end{eqnarray*}
Then
\begin{eqnarray}
\label{15.3}
w_v(U)&=&\sum_{i=1}^n 
c_{iv}\big(\dim (U\cap U_{i-1,v})-\dim (U\cap U_{iv})\big)
\\
\nonumber
&=&
c_{1v}\dim U\, 
+\sum_{i=1}^n (c_{i+1,v}-c_{iv})\dim (U\cap U_{iv}).
\end{eqnarray}

\begin{lemma}\label{le:15.1}
For any two linear subspaces $U_1,U_2$ of $\OQq^n$ we have
\[
w(U_1\cap U_2)+w(U_1+U_2)\geq w(U_1)+w(U_2).
\]
\end{lemma}

\begin{proof}
Let $U_1,U_2$ be two linear subspaces of $\OQq^n$. 
It clearly suffices to show that for any $v\in M_K$, we have
\begin{equation}\label{15.5}
w_v(U_1\cap U_2)+w_v(U_1+U_2)\geq w_v(U_1)+w_v(U_2).
\end{equation}
But this follows easily 
by combining \eqref{15.3}
with $c_{i+1,v}-c_{iv}\geq 0$ for $i=1\kdots n-1$
and
\begin{eqnarray*}
&&\dim (U_1\cap U_2)+\dim (U_1+U_2)=\dim U_1 +\dim U_2,
\\
&&U\cap (U_1+U_2)\supseteq (U\cap U_1)+(U\cap U_2)
\end{eqnarray*}
for any three linear subspaces $U,U_1,U_2$ of $\OQq^n$.
\end{proof}

For any two linear subspaces 
$U_1,U_2$ of $V$ with $\dim U_1<\dim U_2$, we define
\begin{equation}\label{15.slope}
\left\{\begin{array}{l}
d(U_2,U_1):= \dim U_2-\dim U_1,\\[0.1cm]
w(U_2,U_1)=w_{\mathcal{L} ,\cc}(U_2,U_1):= w_{\mathcal{L} ,\cc}(U_2)-w_{\mathcal{L} ,\cc}(U_1),
\\[0.2cm]
\mu (U_2,U_1)=\mu_{\mathcal{L} ,\cc}(U_2,U_1):=\displaystyle{\frac{w(U_2,U_1)}{d(U_2,U_1)}}.
\end{array}\right.
\end{equation}
We prove the following lemma.

\begin{lemma}\label{le:15.between}
Let $V$ be a linear subspace of $\OQq^n$, defined over $K$.
\\[0.1cm]
(i) There exists a unique proper linear subspace $T$ of $V$ such that
\begin{eqnarray*}
&&\mbox{$\mu (V,T)\leq \mu (V,U)$ for every proper linear subspace $U$ of $V$,}
\\
&&\mbox{subject to this constraint, $T$ has minimal dimension.} 
\end{eqnarray*}
This space $T$ is defined over $K$.
\\[0.1cm]
(ii) Let $T$ be as in (i) and let $U$ be any other proper linear subspace of $V$. 
Then $\mu (V,U\cap T)\leq\mu (V,U)$.
\end{lemma}

\begin{proof}
Obviously, there exists a proper linear subspace $T$ of $V$ with (i) 
since $\mu (\cdot ,\cdot )$ assumes only finitely many values.
We prove first that $T$ satisfies (ii), 
and then that $T$ is uniquely determined and defined over $K$.
Put $\mu :=\mu (V,T)$.
Then by Lemma \ref{le:15.1} and since $\mu (V,W)\geq\mu$ 
for any proper linear subspace $W$ of $V$,
\begin{eqnarray*}
w(V,U\cap T) 
&\leq& w(V,U)+w(V,T)-w(V,T+U)
\\
&\leq& w(V,U)+\mu d(V,T)-\mu d(V,T+U)
\\
&=& \mu (V,U)d(V,U) +\mu d(T+U,T)
\\
&\leq& \mu (V,U) (d(V,U)+d(T+U,T))
\\
&=&\mu (V,U)d(V,U\cap T).
\end{eqnarray*} 
This clearly proves (ii).

Now suppose that there exists another subspace $T'$ with (i),
i.e., $\mu (V,T')=\mu$ and $\dim T'=\dim T$. By (ii) we have
$\mu (V,T\cap T')\leq \mu (V,T')=\mu$. By the definition of $\mu$
and the minimality of $\dim T$ we must have $T\cap T'=T=T'$.

It remains to prove that $T$ is defined over $K$.
Let $\sigma\in G_K$. Since $V$ is defined over $K$ and all linear forms
$L_i^{(v)}$ have their coefficients in $K$, we have $\mu (V,\sigma (T))=\mu (V,T)=\mu$,
while $\dim \sigma (T) =\dim T$. So by what we just proved, $\sigma (T)=T$.
This holds for arbitrary $\sigma$, hence $T$ is defined over $K$.
\end{proof}

\noindent
{\bf Remark.} In the situation of Section \ref{2} we have $V=\OQq^n$, $w(\OQq^n)=0$, 
and thus, the subspace $T=T(\mathcal{L} ,\cc )$ defined by \eqref{2.scss}
is precisely the subspace from (i).

In a special case we can give more precise information about the subspace $T$.

\begin{lemma}\label{le:15.special}
Let $V=\OQq^n$ and let $T$ be the subspace from Lemma \ref{le:15.between} (i).
Suppose that 
\begin{equation}\label{15.101}
\bigcup_{v\in M_K}\{ L_1^{(v)}\kdots L_n^{(v)}\}\subseteq\{ X_1\kdots X_n,X_1+\cdots +X_n\}.
\end{equation}
Then there are non-empty, pairwise disjoint subsets $I_1\kdots I_p$ of $\{ 1\kdots n\}$
such that
\begin{equation}\label{15.102}
T=\{ \x\in\OQq^n:\, \sum_{j\in I_i} x_j=0\  \mbox{for } j=1\kdots p\}.
\end{equation}
\end{lemma}

\begin{proof}
Let $k:=\dim T$, $p:=n-k$.
Define the $\OQq$-linear subspace of $\OQq^{n+1}$:
\[
H:=\{ {\bf u}=(u_0\kdots u_n)\in\OQq^{n+1}:\, \sum_{j=1}^n u_jX_j\, -u_0\sum_{j=1}^n X_j\in T^{\bot}\}.
\]
Notice that $\dim H =p+1$ and $(1\kdots 1)\in H$.
We show that $H$ is closed under coordinatewise multiplication, i.e., $H$ is a sub-$\OQq$-algebra
of $\OQq^{n+1}$. This being done, it is not difficult to show that there are pairwise
disjoint subsets $I_0\kdots I_p$ of $\{ 0\kdots n\}$ such that $H$ is the set of ${\bf u}\in\OQq^{n+1}$
with $u_i=u_j$ for each pair $i,j$ for which there is $l\in \{ 0\kdots p\}$
with $i,j\in I_l$. This easily translates into \eqref{15.102}.      

Fix ${\bf a}=(a_0\kdots a_n)\in H$. Choose $c\in\OQq$ such that
$b_i:= a_i+c\not= 0$ for $i=0\kdots n$. Then ${\bf b}:=(b_0\kdots b_n)\in H$. 
Define the linear transformation 
\[
\varphi :\, \OQq^n\to\OQq^n :\, (x_1\kdots x_n)\mapsto (b_1x_1\kdots b_nx_n).
\]
In general, $\sum_{j=1}^n \xi_jX_j\in \varphi(T)^{\bot}$ if and only if 
$\sum_{j=1}^n b_j\xi_jX_i\in T^{\bot}$. Using this and ${\bf b}\in H$,
it follows that
for $(u_0\kdots u_n)\in \OQq^{n+1}$ we have
\begin{eqnarray}\label{15.103}
&&\sum_{j=1}^n u_jX_j -u_0\sum_{j=1}^n X_j\in \varphi (T)^{\bot} 
\\
\nonumber
&&\qquad\Longleftrightarrow
\sum_{j=1}^n b_ju_j X_j\,-u_0\sum_{j=1}^n b_jX_j\in T^{\bot}
\\
\nonumber
&&\qquad\Longleftrightarrow 
\sum_{j=1}^n b_ju_jX_j\, -b_0u_0\sum_{j=1}^n X_j \in T^{\bot}.
\end{eqnarray}
This implies for any $v\in M_K$ and any subset $\{ i_1\kdots i_k\}$ of $\{ 1\kdots n\}$,
that
$L_{i_1}^{(v)}|_{\varphi (T)}\kdots L_{i_k}^{(v)}|_{\varphi (T)}$ are linearly independent
if and only if $L_{i_1}^{(v)}|_T\kdots L_{i_k}^{(v)}|_T$ are linearly independent.
Consequently, $w(\varphi (T))=w(T)$ and thus,\\ $\mu (\OQq^n,\varphi (T))=\mu (\OQq^n, T)$.
Now Lemma \ref{le:15.between} (i) implies that $\varphi (T)=T$. 

Combined with \eqref{15.103},
this implies that if ${\bf u}\in H$, then ${\bf b}\cdot {\bf u}\in H$.
But then, ${\bf a}\cdot {\bf u}={\bf b}\cdot {\bf u}-c{\bf u}\in H$. This shows that $H$
is closed under coordinatewise multiplication and proves our lemma.
\end{proof}
\vskip0.1cm 
 
For every linear subspace $U$ of $\OQq^n$, 
we define the point $P(U)=P_{\mathcal{L} ,\cc}(U):=(\dim U, w(U))\in\Rr^2$.
In particular, $P(\{{\bf 0}\})=(0,0)$.
Notice that $\mu (U_2,U_1)$ defined by \eqref{15.slope} is precisely the slope of the line segment 
from $P(U_1)$ to $P(U_2)$.

Let again $V$ be a linear subspace of $\OQq^n$, defined over $K$.
Denote by $C(V,\mathcal{L} ,\cc )$ the upper convex hull of the points
$P(U)$ for all linear subspaces $U$ of $V$,
and by $B(V,\mathcal{L} ,\cc )$ the upper boundary of $C(V,\mathcal{L} ,\cc )$.
Thus, $B(V,\mathcal{L} ,\cc )$ is the graph of a piecewise linear, convex function 
from $[0,\dim V]$ to $\Rr$, and $C(V,\mathcal{L} ,\cc )$ is the set of points
on and below $B(V,\mathcal{L} ,\cc )$.
 
As long as it is clear which are the underlying tuples $\LL ,\cc$,
we suppress the dependence on these tuples in our notation,
i.e., we write $w,\mu,P$ for $w_{\mathcal{L} ,\cc},\mu_{\mathcal{L} ,\cc} ,P_{\mathcal{L},\cc}$.

\setlength{\unitlength}{0.7truecm}
\begin{picture}(18,9)
\thicklines

\put(0,2){\line(1,0){18}}
\put(18,1){\makebox(0,0)[r]{\vector(1,0){1.5}\,\raisebox{-0.1truecm}{dim}}}

\put(2,0){\line(0,1){9}}
\put(1,9){\makebox(0,0)[t]{\shortstack{$w$\\ \vector(0,1){1.5}}}}

\put(2,2){\circle*{0.2}}
\put(1.9,2.1){\makebox(0,0)[br]{$P(\{{\bf 0}\})$}}

\put(2,2){\line(1,1){4}}

\put(6,6){\circle*{0.2}}
\put(5.9,6.1){\makebox(0,0)[br]{$P(T_1)$}}

\put(6,6){\line(2,1){4}}

\put(10,8){\circle*{0.2}}
\put(9.9,8.1){\makebox(0,0)[br]{$P(T_2)$}}

\put(10,8){\line(1,0){1.5}}

\multiput(11.6,8)(0.2,-0.08){12}{\circle*{0.07}}

\put(14,7){\line(1,-1){1}}

\put(15,6){\circle*{0.2}}
\put(14.9,5.9){\makebox(0,0)[tr]{$P(T_{r-1})$}}

\put(15,6){\line(2,-3){2}}

\put(17,3){\circle*{0.2}}
\put(16.9,2.9){\makebox(0,0)[tr]{$P(V)$}}

\put(9,4){\circle*{0.2}}
\put(8.9,3.9){\makebox(0,0)[tr]{$P(U)$}}

\end{picture}

\begin{lemma}\label{le:15.3}
There exists a unique filtration
\begin{equation}\label{15.8}
\{ {\bf 0}\}\propersubset T_1\propersubset\cdots\propersubset T_{r-1}\propersubset T_r=V
\end{equation}
such that $P(\{{\bf 0}\}),\, P(T_1)\kdots P(T_{r-1}),\, P(V)$ 
are precisely the vertices of $B(V,\mathcal{L} ,\cc )$.
\\
The spaces $T_1\kdots T_{r-1}$ are defined over $K$.
\end{lemma}

\begin{proof}
The proof is by induction on $m:=\dim V$.
The case $m=1$ is trivial.
Let $m\geq 2$. 
There is only one candidate for the subspace in the filtration preceding $V$,
that is the subspace $T$ from Lemma \ref{le:15.between} (i).
This space $T$ is defined over $K$. 
By the induction hypothesis applied to $T$, 
there exists a unique filtration
\[
\{ {\bf 0}\}\propersubset T_1\propersubset\cdots\propersubset T_{r-1}=T
\]
such that $P(\{{\bf 0}\}),\, P(T_1)\kdots P(T_{r-1})$ 
are precisely the vertices of $B(T,\mathcal{L},\cc )$.
Moreover, $T_1\kdots T_{r-2}$ are defined over $K$.

We have to prove that together with $P(V)$ these points are the vertices of $B(V,\mathcal{L} ,\cc )$.
We first note that since $T_{r-2}\propersubset T_{r-1}$, we have 
$\mu (V,T_{r-2})>\mu (V,T_{r-1})$, hence
\[
\mu (T_{r-1},T_{r-2})=\frac{d(V,T_{r-2})\mu (V,T_{r-2})-d(V,T_{r-1})\mu (V,T_{r-1})}{d(T_{r-1},T_{r-2})}>\mu (V,T_{r-1}).
\]
Therefore, 
$P(\{ {\bf 0}\}),\, P(T_1)\kdots P(V)$ are the vertices of the graph of a 
piecewise linear convex function on $[0,m]$.
Let $C$ be the set of points on and below this graph. 
To prove that this graph is $B(V,\mathcal{L} ,\cc )$,
we have to show that $C$ contains all points $P(U)$ 
with $U$ a linear subspace of $V$.

If $U\subseteq T_{r-1}$ we have $P(U)\in C(T_{r-1},\LL ,\cc )\subset C$. 
Suppose that $U\not\subseteq T_{r-1}$.
Then by Lemma \ref{le:15.between} (ii), we have
$\mu (V,U\cap T_{r-1})\leq\mu (V,U)$.
Since $P(U\cap T_{r-1})\in C$, 
$\dim U\geq \dim U\cap T_{r-1}$ and $C$ is upper convex,
this implies that $P(U)\in C$. This completes our proof.
\end{proof}

\noindent
The filtration constructed above is called the \emph{filtration of } $V$
with respect to $(\mathcal{L} ,\cc )$.

\noindent
{\bf Remark.} 
The Harder-Narasimhan filtration introduced by Faltings and
W\"{u}stholz in \cite{Faltings-Wuestholz1994} is given by
$\{{\bf 0}\}\propersubset T_{r-1}'
\propersubset\cdots\propersubset {\rm Hom}(V,\OQq )$,
where for a linear subspace $T$ of $V$, we define $T'$
as the set of linear functions from $V$ to $\OQq$ that vanish
identically on $T$.

\section{The successive infima of a twisted height}\label{16}

As before, $K\subset\OQq$ is an algebraic number field, $n$ an integer 
$\geq 1$,
and $\LL$ a tuple of linear forms
and $\cc$ a tuple of reals satisfying \eqref{2.10a}--\eqref{2.7a}.
We denote as usual by $\lambda_1(Q)\kdots\lambda_n(Q)$ the successive infima
of $\HLcQ$. In this section, we prove a limit result for these successive
infima as $Q\to\infty$.

Define 
\begin{equation}\label{16.spaces}
T_i(Q):=\bigcap_{\lambda >\lambda_i(Q)}\span\{\x\in\OQq^n:\, 
\HLcQ (\x)\leq\lambda\}\ \ (i=1\kdots n).
\end{equation}
Let
\[
\{ {\bf 0}\}=: T_0\propersubset T_1\propersubset\cdots\propersubset T_{r-1}\propersubset T_r :=\OQq^n
\]
be the filtration of $\OQq^n$ with respect to $(\mathcal{L} ,\cc )$, 
as defined in Lemma \ref{le:15.3}, 
and put $d_l:=\dim T_l$ for $l=0\kdots r$.

Given any two linear subspaces $U,V$ of $\OQq^n$ with $\dim U<\dim V$,
we define again 
$\mu (V,U)=\mu_{\mathcal{L} ,\cc}(V,U):=\frac{w(V)-w(U)}{\dim V-\dim U}$. 

Our general result on the successive infima of $\HLcQ$ is as follows.

\begin{thm}\label{th:16.2}
For every $\delta >0$ there exists $Q_0$, such that for every $Q\geq Q_0$
the following holds:
\begin{eqnarray}
\label{16.4}
&&Q^{-\mu (T_l,T_{l-1})-\delta}\leq \lambda_i(Q)\leq
Q^{-\mu (T_l,T_{l-1})+\delta}
\\
\nonumber
&&
\hspace*{3truecm}\mbox{for $l=1\kdots r$, $i=d_{l-1}+1\kdots d_l$,}
\\
\label{16.5}
&&T_{d_l}(Q)=T_l\ \ \mbox{for } l=1\kdots r.
\end{eqnarray}
\end{thm}

We start with some preparations and lemmas. 
Fix a linear subspace $T$ of $\OQq^n$ 
of dimension $k\in\{ 1\kdots n-1\}$ which is defined over $K$. 
Choose an injective linear map 
\[
\varphi{'}:\, \OQq^k\hookrightarrow \OQq^n\ \ \mbox{with } \varphi{'}(\OQq^n)=T
\]
and a surjective linear map
\[
\varphi{''}:\, \OQq^n\twoheadrightarrow \OQq^{n-k}\ \ \mbox{with } 
{\rm Ker}(\varphi{''})=T,
\]
both defined over $K$.
Recall that for every linear form
$L\in K[X_1\kdots X_n]^{\lin}$ vanishing identically on $T$ 
there is a unique linear form $L{''}\in K[X_1\kdots X_{n-k}]^{\lin}$
such that $L=L{''}\circ\varphi{''}$; we denote this $L{''}$
by $L\circ\varphi{''}^{-1}$.

We assume \eqref{15.2}, which is no loss of generality.
For $v\in M_K$, let the set $I_v(T)$ be given by \eqref{15.1},
and define a tuple $\LL{'}$ from $K[X_1\kdots X_k]^{\lin}$
and a tuple of reals $\cc{'}$ by 
\begin{equation}\label{15.6a}
\left\{\begin{array}{l}
\mathcal{L} ' := (L_i^{(v)}\circ\varphi{'}:\, v\in M_K,\, i\in I_v(T)),\\
\cc ' := (c_{iv}:\, v\in M_K,\, i\in I_v(T)).
\end{array}\right.
\end{equation}

Let $v\in M_K$.
Since $L_j^{(v)}|_T$ ($j\in I_v(T)$) form a basis
of ${\rm Hom}(T,\OQq )$,
and since $T$ is defined over $K$, 
there are unique $\alpha_{ijv}\in K$
such that $L_i^{(v)}|_T =\sum_{j\in I_v(T)} \alpha_{ijv}L_i^{(v)}|_T$
for $i\in I_v(T)^c :=\{ 1\kdots n\}\setminus I_v(T)$.
By our definition of $I_v(T)$, we have $\alpha_{ijv}=0$ 
for $i\in I_v(T)^c,\, j\in I_v(T)$, $j>i$.
In other words, there are unique linear forms
\begin{equation}\label{15.7}
\widetilde{L}_i^{(v)}=L_i^{(v)}-\,\sum_{\stackrel{j\in I_v(T)}{j<i}}\alpha_{ijv}L_j^{(v)}\ \ (i\in I_v(T)^c) 
\end{equation}
with $\alpha_{ijv}\in K$ that vanish identically on $T$. 
These linear forms are linearly independent,
so they may be viewed as a basis of ${\rm Hom} (\OQq^n/T ,\OQq )$.
 
We now define a tuple $\LL{''}$ in $K[X_1\kdots X_{n-k}]^{\lin}$
and a tuple of reals $\cc{''}$ by
\begin{equation}\label{15.7a}
\left\{\begin{array}{l}
\mathcal{L}'':=( \widetilde{L}_i^{(v)}\circ\varphi{''}^{-1}:\, v\in M_K,\,i\in I_v(T)^c),\\
{\bf c}'' := (c_{iv}:\, v\in M_K,\, i\in I_v(T)^c).
\end{array}\right.
\end{equation}

Let $U$ be a linear subspace of $\OQq^k$ of dimension $u$, say.
Then
$w_{\mathcal{L}{'},\cc{'}}(U)=\sum_{v\in M_K} w_{\mathcal{L}{'},\cc{'},v}(U)$,
where in analogy to \eqref{15.6},
\begin{equation}\label{16.101}
w_{\mathcal{L}{'},\cc{'},v}(U)=
\left\{\begin{array}{l}
0\ \mbox{if $u=0$,}
\\
c_{i_1(v),v}+\cdots +c_{i_u(v),v}\ \mbox{if $u>0$,}
\end{array}\right.
\end{equation}
where $i_1(v)$ is the smallest index $i\in I_v(T)$ such that $L_i^{(v)}\circ\varphi{'}|_U\not= 0$ and for $l=2\kdots u$,
$i_l(v)$ is the smallest index $i>i_{l-1}(v)$ in $I_v(T)$ such that
$L_{i_1(v)}^{(v)}\circ\varphi{'}|_U\kdots L_{i_{l-1}(v)}^{(v)}\circ\varphi{'}|_U$, $L_i^{(v)}\circ\varphi{'}|_U$ are linearly independent.

Likewise, if $U$ is an $u$-dimensional linear subspace of $\OQq^{n-k}$, then
$w_{\mathcal{L}{''},\cc{''}}(U)=\sum_{v\in M_K} w_{\mathcal{L}{''},\cc{''},v}(U)$,
with
\begin{equation}\label{16.102}
w_{\mathcal{L}{''},\cc{''},v}(U)=
\left\{\begin{array}{l}
0\ \mbox{if $u=0$,}
\\
c_{i_1(v),v}+\cdots +c_{i_u(v),v}\ \mbox{if $u>0$,}
\end{array}\right.
\end{equation}
where $i_1(v)$ is the smallest index $i\in I_v(T)^c$ such that $\widetilde{L}_i^{(v)}\circ\varphi{''}^{-1}|_U\not= 0$ and for $l=2\kdots u$,
$i_l(v)$ is the smallest index $i>i_{l-1}(v)$ in $I_v(T)^c$ such that
$\widetilde{L}_{i_1(v)}^{(v)}\circ\varphi{''}^{-1}|_U\kdots \widetilde{L}_{i_{l-1}(v)}^{(v)}\circ\varphi{''}^{-1}|_U$, $\widetilde{L}_i^{(v)}\circ\varphi{''}^{-1}|_U$ are linearly independent.

\begin{lemma}\label{le:15.2}
(i) Let $U$ be a linear subspace of $\OQq^k$. Then
\[
w_{\mathcal{L}{'},\cc{'}}(U)=w_{\mathcal{L},\cc}(\varphi{'}(U)).
\]
(ii) Let $U$ be a linear subspace of $\OQq^{n-k}$. Then
\[
w_{\mathcal{L}{''},{\bf c}''}(U)=w_{\mathcal{L} ,\cc}(\varphi{''}^{-1}(U))-w_{\mathcal{L} ,\cc}(T).
\]
\end{lemma}

\begin{proof}
(i) For $U=\{ {\bf 0}\}$ the assertion is true.
Suppose $U$ has dimension $u>0$. Let $v\in M_K$. 
The set $\{ i_1(v)\kdots i_u(v)\}$ from \eqref{16.101} is precisely $I_v(\varphi{'}(U))$ since
$I_v(\varphi{'}(U))\subseteq I_v(T)$. 
Therefore, $w_{\mathcal{L}{'},\cc{'},v}(U)=w_{\mathcal{L},\cc ,v}(\varphi{'}(U))$ for $v\in M_K$.
Now (i) follows by summing over $v$.

(ii) Suppose $U$ has dimension $u>0$. Let $v\in M_K$. Put $W:=\varphi{''}^{-1}(U)$.
Recall that $I_v(W)=\{ j_1(v)\kdots j_m(v)\}$, where $m:=\dim W$,
$j_1(v)$ is the smallest index $j\in\{ 1\kdots n\}$ 
such that $L_j^{(v)}|_W\not= 0$, etc.
The indices $j_1(v),j_2(v),\ldots$
do not change if we replace $L_j^{(v)}$ by $\widetilde{L}_j^{(v)}$
for $j\in I_v(T)^c$. This implies that the set    
$\{ i_1(v)\kdots i_{n-k}(v)\}$ from \eqref{16.102} is $I_v(W)\setminus I_v(T)$,
and so $w_{\mathcal{L}{''},\cc{''},v}(U)=
w_{\mathcal{L} ,\cc ,v}(W)-w_{\mathcal{L} ,\cc ,v}(T)$.
By summing over $v$ we get (ii).
\end{proof}

The pair $(\mathcal{L} ' ,\cc ')$ gives rise to a class of twisted heights
$H_{\mathcal{L}{'} ,\cc{'},Q}:\, \OQq^k\to\Rr_{\geq 0}$ in the usual
manner. That is, if $\x\in E^k$ for some finite extension $E$ of $K$, then
\begin{equation}\label{16.twisted-height1}
H_{\mathcal{L}{'},\cc{'},Q}(\x )=
\prod_{w\in M_E}\max_{i\in I_w(T)}\|L_i^{(w)}\circ\varphi{'}(\x )\|_wQ^{-c_{iw}}
\end{equation}
where $I_w(T):=I_v(T)$ if $w$ lies above $v\in M_K$.

Likewise, we have twisted heights
$H_{\mathcal{L}{''},\cc{''},Q}:\, \OQq^{n-k}\to\Rr_{\geq 0}$, defined such that if
$\x\in E^{n-k}$ for some finite extension $E$ of $K$, then 
\begin{equation}\label{16.twisted-height2}
H_{\mathcal{L}{''},\cc{''},Q}(\x )=
\prod_{w\in M_E}\max_{i\in I_w^c(T)}\|\widetilde{L}_i^{(w)}\circ\varphi{''}^{-1}(\x )\|_wQ^{-c_{iw}}
\end{equation}
where $\widetilde{L}_i^{(w)}:=\widetilde{L}_i^{(v)}$ if $w$ lies above $v\in M_K$.

In what follows, constants implied by $\ll$, $\gg$ 
depend only on $\LL ,\cc$ and $T$.

\begin{lemma}\label{le:16.0} 
(i) For $\x\in\OQq^k$, $Q\geq 1$ we have
\[
H_{\mathcal{L}{'},\cc{'},Q}(\x )\gg\ll \HLcQ (\varphi{'}(\x)).
\]
(ii) For $\x\in\OQq^n$, $Q\geq 1$ we have
\[
H_{\mathcal{L}{''},\cc{''},Q}(\varphi{''}(\x ))\ll \HLcQ (\x).
\]
\end{lemma}

\begin{proof}
(i) The inequality $H_{\mathcal{L}{'},\cc{'},Q}(\x )\leq \HLcQ (\varphi{'}(\x))$
for $\x\in\OQq^k$, $Q\geq 1$ is trivial. We prove the reverse inequality.
Since the linear forms $\widetilde{L}_i^{(v)}$ $(i\in I_v(T)^c)$ defined in \eqref{15.7}
vanish identically on $T$,
there exist constants $C_v>0$ ($v\in M_K$), all but finitely many of which are $1$,
such that for $\x\in K^k$, $v\in M_K$, $i\in I_v(T)^c$,
\[
\| L_i^{(v)}(\varphi{'}(\x ))\|_v\leq C_v\max_{\stackrel{j\in I_v(T)}{j<i}} \|L_j^{(v)}(\x )\|_v.
\]
Taking $Q\geq 1$ we obtain, in view of \eqref{15.2},
\[
\| L_i^{(v)}(\varphi{'}(\x ))\|_vQ^{-c_{iv}}\leq 
C_v\max_{\stackrel{j\in I_v(T)}{j<i}} \|L_j^{(v)}(\x )\|_vQ^{-c_{jv}}.
\]
This shows that for $\x\in K^k$, $Q\geq 1$, $v\in M_K$, we have
\[
\max_{1\leq i\leq n}\| L_i^{(v)}(\varphi{'}(\x ))\|_vQ^{-c_{iv}}\leq 
C_v\max_{j\in I_v(T)} \|L_j^{(v)}(\x )\|_vQ^{-c_{jv}}.
\]
If instead we have $\x\in E^k$ for some finite extension $E$ of $K$, we have the same
inequalities for $w\in M_E$, but with constants $C_w:=C_v^{d(w|v)}$ 
where $v\in M_K$ is the place below $w$.
By taking the product over $w\in M_E$, we get (i).

The proof of (ii) is entirely similar.  
\end{proof}

\begin{lemma}\label{le:16.1}
Suppose that
\begin{equation}\label{16.1}
\mu (\OQq^n,U)\geq \mu (\OQq^n,\{ {\bf 0}\})\ \
\mbox{for every proper linear subspace $U$ of $\OQq^n$.}
\end{equation}
Then for every $\delta >0$ there is $Q_0$ such that for every $Q\geq Q_0$,
\begin{equation}\label{16.2}
Q^{-\mu (\OQq^n,\{ {\bf 0}\}) -\delta}
\leq \lambda_1(Q)\leq\cdots\leq \lambda_n(Q)
\leq Q^{-\mu (\OQq^n ,\{ {\bf 0}\})+\delta}.
\end{equation}
\end{lemma}

\begin{proof}
We first assume that $n=1$. In this case, 
$L_1^{(v)}=\alpha_vX$ with $\alpha_v\in K^*$ for $v\in M_K$, 
and $\mu (\OQq ,\{0\})= \sum_{v\in M_K} c_{1v}$.
By the product formula, we have for ${\bf x}=x\in K^*$,
\[
\HLcQ (x)=\prod_{v\in M_K}\|\alpha_vx\|_vQ^{-c_{1v}}=CQ^{-\mu (\OQq, \{ 0\}) }
\]
for some non-zero constant $C$. This is true also for $x\not\in K$.
So for $n=1$, our lemma is trivially true.

Next, we assume $n\geq 2$. 
We first make some reductions and then apply Theorem \ref{th:8.1}.
By Lemma \ref{le:9.3} there is no loss of generality if in the proof of our lemma,
we replace $c_{iv}$ by $c_{iv}' := c_{iv}-\frac{1}{n}\sum_{j=1}^n c_{jv}$ 
for $v\in M_K$, $j=1\kdots n$.
This shows that there is no loss of generality to assume that
$\sum_{i=1}^n c_{iv}=0$ for $v\in M_K$, i.e., condition \eqref{8.6}.
This being the case, suppose that $\sum_{v\in M_K}\max_{1\leq i\leq n} c_{iv}\leq \theta$
with $\theta >0$. Then we can make a reduction to \eqref{8.7} by replacing $Q$ by $Q^{\theta}$
and $c_{iv}$ by $c_{iv}/\theta$ for $v\in M_K$, $i=1\kdots n$. So we may also assume that
\eqref{8.7} is satisfied. Finally, by Lemma \ref{le:9.2} and the subsequent remark,
there is no loss of generality to assume \eqref{8.2}.
Under assumption \eqref{8.6}, condition \eqref{16.1} translates into \eqref{8.8}. 
So we may assume without loss of generality that all conditions of Theorem \ref{th:8.1} are satisfied.
Notice that with these assumptions,
\[
\mu (\OQq^n ,\{{\bf 0}\})=\frac{1}{n}\sum_{v\in M_K}\sum_{i=1}^n c_{iv} =0.
\] 

Let $0<\delta\leq 1$. 
Theorem \ref{th:8.1} implies that the set of $Q$ with
$\lambda_1(Q)\leq Q^{-\delta /2n}$ is bounded.
Together with \eqref{10.minkowski}, this implies that for every sufficiently large $Q$,
we have $\lambda_1(Q)\geq Q^{-\delta /2n}$, 
$\lambda_n(Q)\leq Q^{\delta}$. 
\end{proof}

\begin{proof}[Proof of Theorem \ref{th:16.2}]
We proceed by induction on $r$. For $r=1$ we can apply
Lemma \ref{le:16.1}. Assume $r\geq 2$.
We fix $\delta >0$, and then $\delta '>0$ which is a sufficiently
small function of $\delta$.
We write $w$ for $w_{\mathcal{L} ,\cc}$, $\mu$ for $\mu_{\mathcal{L} ,\cc}$.

By Lemma \ref{15.2} (ii) with $T=T_{r-1}$, $k=d_{r-1}=\dim T$, 
we have for any two 
linear subspaces $U_1\propersubset U_2$ of $\OQq^{n-d_{r-1}}$ that
\[
\mu_{\mathcal{L}{''},\cc{''}}(U_2,U_1)=
\mu(\varphi{''}^{-1}(U_2),\varphi{''}^{-1}(U_1)).
\]
Thus, the property of $T_{r-1}$ that 
$\mu (\OQq^n ,T_{r-1})\leq \mu (\OQq^n,U)$ 
for any proper linear
subspace $U$ of $\OQq^n$ translates into
\[ 
\mu_{\mathcal{L}{''},\cc{''}}(\OQq^{n-d_{r-1}},\{ {\bf 0}\})\leq
\mu_{\mathcal{L}{''},\cc{''}}(\OQq^{n-d_{r-1}},U)
\]
for any proper linear subspace $U$ of $\OQq^{n-d_{r-1}}$. So by
Lemma \ref{le:16.1} we have for every sufficiently large $Q$,  
\[
H_{\mathcal{L}{''},\cc{''},Q}({\bf y})\geq  Q^{-\mu (\OQq^n ,T_{r-1})-\delta ' }\ \ \mbox{for }
{\bf y}\in\OQq^{n-d_{r-1}}\setminus\{ {\bf 0}\}.
\]
Together with Lemma \ref{le:16.0} (ii),
this implies for every sufficiently large $Q$,
\begin{equation}\label{16.8}
\HLcQ (\x )\geq Q^{-\mu (\OQq^n,T_{r-1})-2\delta '}\ \ \mbox{for } \x\in\OQq^n\setminus T_{r-1}.
\end{equation}
Consequently, for every sufficiently large $Q$ we have
\begin{equation}\label{16.9}
Q^{-\mu (\OQq^n,T_{r-1})-2\delta '}\leq \lambda_{d_{r-1}+1}(Q)\leq\cdots\leq\lambda_n(Q).
\end{equation}

For $i=1\kdots d_{r-1}$, 
denote by $\lambda_i'(Q)$ the $i$-th successive infimum
of $\HLcQ$ restricted to $T_{r-1}$, 
i.e., the infimum of all
$\lambda >0$ such that the set of $\x\in T_{r-1}$ 
with $\HLcQ (\x)\leq \lambda$ contains
at least $i$ linearly independent points.
By Lemma \ref{le:16.0} (i) with $T=T_{r-1}$, $k=d_{r-1}$ this is,
apart from bounded multiplicative factors independent of $Q$,
equal to the $i$-th successive
infimum of $H_{\mathcal{L}{'},\cc{'},Q}$.
Further, 
by Lemma \ref{le:15.2} (i) with $T=T_{r-1}$, $k=d_{r-1}$, 
for any two subspaces $U_1\propersubset U_2$ of $\OQq^{d_{r-1}}$ 
we have $w_{\mathcal{L} ',\cc '}(U_2,U_1)=w(\varphi '(U_2),\varphi '(U_2))$.
By applying the induction hypothesis to $(\mathcal{L}',\cc' )$
and then carrying it over to $T_{r-1}$ by means of $\varphi '$, 
we infer that for every sufficiently
large $Q$, we have
\begin{equation}\label{16.10}
Q^{-\mu (T_l,T_{l-1})-\delta '}\leq \lambda_i'(Q)\leq
Q^{-\mu (T_l,T_{l-1})+\delta '}
\end{equation}
for $l=1\kdots r-1$, $i=d_{l-1}+1\kdots d_l$ and moreover, 
\begin{equation}\label{16.11}
\bigcap_{\lambda>\lambda_{d_l}'(Q)}\span\{ \x\in T_{r-1}:\,\HLcQ(\x )\leq\lambda\}= T_l
\end{equation}
for $l=1\kdots r-1$. 
Clearly, we have $\lambda_i(Q)\leq\lambda_i'(Q)$ for $i=1\kdots d_{r-1}$, 
and so
\[ 
\lambda_{d_{r-1}}(Q)\leq Q^{-\mu (T_{r-1},T_{r-2})+\delta '}
\]
for $Q$
sufficiently large. Assuming $\delta '$ is sufficiently small,
this is smaller
than the lower bound $Q^{-\mu (\OQq^n,T_{r-1})-2\delta '}$ in \eqref{16.8}.
Hence for sufficiently large $Q$ and sufficiently small $\eps$,
all vectors $\x\in\OQq^n$ with $\HLcQ (\x )\leq\lambda_{d_{r-1}}(Q)+\eps$ 
lie in $T_{r-1}$. 
That is,
\[
T_{d_{r-1}}(Q)=T_{r-1},\ \lambda_i(Q)=\lambda_i'(Q)\ 
\mbox{for $i=1\kdots d_{r-1}$.}
\]
Together with \eqref{16.11} this implies \eqref{16.5}. Further,
\eqref{16.10} becomes
\begin{equation}\label{16.x}
Q^{-\mu (T_l,T_{l-1})-\delta '}\leq \lambda_i(Q)\leq
Q^{-\mu (T_l,T_{l-1})+\delta '}
\end{equation}
for $l=1\kdots r-1$, $i=d_{l-1}+1\kdots d_l$.
Using subsequently  
Proposition \ref{pr:9.1}, the lower bounds in \eqref{16.x}, \eqref{16.9}, 
and that the quantity $\alpha =\sum_{v\in M_K}\sum_{i=1}^n c_{iv}$ from
Propostion \ref{pr:9.1} equals
\[
w(\OQq^n,\{ {\bf 0}\})=\sum_{l=1}^r w(T_l,T_{l-1})
=\sum_{l=1}^r d_l\mu (T_l,T_{l-1}),
\]
and taking $\delta '$ sufficiently small,
we infer that for every sufficiently large $Q$,
\begin{eqnarray*} 
\lambda_n(Q)&\leq& 
2^{n(n-1)/2}\DL Q^{-\alpha}\big( \lambda_1(Q)\cdots\lambda_{n-1}(Q)\big)^{-1} 
\\
&\leq& Q^{-\alpha +
\sum_{l=1}^r d_l\mu (T_l,T_{l-1})-\mu (\OQq^n,T_{r-1})+2n\delta{'}}
\leq Q^{-\mu (\OQq^n,T_{r-1})+\delta}.
\end{eqnarray*} 
As a consequence, \eqref{16.4} holds as well. This completes our proof.
\end{proof}

\section{A height estimate for the filtration subspaces}\label{17}

As before, $K$ is a number field, $n$ an integer $\geq 2$,
and $(\mathcal{L} ,\cc )$ a pair with \eqref{2.10a}--\eqref{2.7a}.
We derive an upper bound for the heights of the spaces
occurring in the filtration of $(\mathcal{L} ,\cc )$ in terms of the heights
of the linear forms from $\mathcal{L}$. We start with some
auxiliary results. 

Let $p$ be an integer with $1<p<n$. Put $N:=\binom{n}{p}$.
Similarly as in Section \ref{6}, let $C(n,p)=(I_1\kdots I_N)$ 
be the lexicographically ordered sequence of $p$-element subsets 
of $\{ 1\kdots n\}$.
For $j=1\kdots N$, $v\in M_K$ define
\begin{equation}\label{17.1}
\widehat{L}_j^{(v)}:=L_{i_1}^{(v)}\wedge\cdots\wedge L_{i_p}^{(v)},\ \ 
\widehat{c}_{jv}:= c_{i_1,v}+\cdots +c_{i_p,v}
\end{equation}
where $I_j=\{ i_1<\cdots <i_p\}$ is the $j$-th set from $C(n,p)$, and put 
\begin{equation}\label{17.2}
\left\{
\begin{array}{l}
\widehat{\mathcal{L}}:=(\widehat{L}_j^{(v)}:\, v\in M_K,\,j=1\kdots N),
\\
\widehat{{\bf c}}:=(\widehat{c}_{jv}:\, v\in M_K,\, j=1\kdots N).
\end{array}\right.
\end{equation}
Then $H_{\widehat{{\mathcal{L}}},\widehat{{\bf c}},Q}:\, \OQq^N\to\Rr_{\geq 0}$ 
is defined in a similar manner as $\HLcQ$, i.e.,
if $\widehat{\x}\in E^N$ for some finite extension $E$ of $K$,
then 
\[
H_{\widehat{{\mathcal{L}}},\widehat{{\bf c}},Q}(\widehat{\x}):=\prod_{w\in M_E}
\max_{1\leq j\leq N} \| \widehat{L}_j^{(w)}(\widehat{\x})\|_wQ^{-\widehat{c}_{jw}}
\]
where $\widehat{L}_j^{(w)}:=\widehat{L}_j^{(v)}$,
$\widehat{c}_{jw}:=d(w|v)\widehat{c}_{jv}$ if $w$ lies above $v\in M_K$.

\begin{lemma}\label{le:17.1}     
Let $\x_1\kdots\x_p\in \OQq^n$, $Q\geq 1$. Then
\[
H_{\widehat{{\mathcal{L}}},\widehat{{\bf c}},Q}(\x_1\wedge\cdots\wedge \x_p)\leq
p^{p/2}\HLcQ (\x_1)\cdots\HLcQ (\x_p).
\]
\end{lemma}

\begin{proof}
Put $\widehat{\x}:=\x_1\wedge\cdots\wedge\x_p$.
Let $E$ be a finite extension of $K$ such that $\x_1\kdots\x_p\in E^n$.
Let $I_j=\{ i_1<\cdots <i_p\}$ be one of the $p$-element subsets
from $I_1\kdots I_N$ and let $w\in M_E$.
Then
by an argument completely similar to the proofs of \eqref{4.4b},\eqref{4.4c},
one shows
\begin{eqnarray*}
\| \widehat{L}_j^{(w)}(\widehat{\x} )\|_wQ^{-\widehat{c}_{jw}}&=&
\|\det \big(L_{i_k}^{(w)}(\x_l)\big)_{k,l=1\kdots p}\|_wQ^{-\widehat{c}_{jw}}
\\
&\leq& 
p^{ps(w)/2}\prod_{l=1}^p\max_{1\leq k\leq p} 
\| L_{i_k}^{(w)}(\x_l)\|_wQ^{-c_{i_k,w}}
\\
&\leq&
p^{ps(w)/2}\prod_{l=1}^p\max_{1\leq i\leq n} \| L_i^{(w)}(\x_l)\|_wQ^{-c_{iw}}.
\end{eqnarray*}
By taking the maximum over $j=1\kdots N$ and then the product over
$w\in M_E$, our Lemma follows.
\end{proof}

We keep the notation from above.
For $Q\geq 1$, let $\lambda_1(Q)\kdots\lambda_n(Q)$ denote the
successive infima of $\HLcQ$.
Further, let $\nu_1(Q)\kdots \nu_N(Q)$ 
be the products $\lambda_{i_1}(Q)\cdots\lambda_{i_p}(Q)$
($1\leq i_1<\cdots <i_p\leq n$), ordered such that
\[ 
\nu_1(Q)\leq\cdots\leq \nu_N(Q),
\]
and let $\widehat{\lambda}_1(Q)\kdots\widehat{\lambda}_N(Q)$ 
denote the successive infima of $H_{\widehat{\mathcal{L}},\widehat{\cc},Q}$.

\begin{lemma}\label{le:17.2}
For $Q\geq 1$, $j=1\kdots N$ we have
\[
N^{-npN}\nu_j(Q)\leq\widehat{\lambda}_j(Q)\leq p^{p/2}\nu_j(Q).
\]
\end{lemma}

\begin{proof}
Fix $Q\geq 1$ and write $\lambda_i,\widehat{\lambda}_j,\nu_j$ for
$\lambda_i,\widehat{\lambda}_j(Q),\nu_j(Q)$.
Let $\varepsilon >0$.
Choose $\OQq$-linearly independent vectors 
${\bf g}_1\kdots {\bf g}_n\in\OQq^n$
such that $\HLcQ ({\bf g}_i)\leq \lambda_i(1+\varepsilon )$
for $i=1\kdots n$. 
Then the vectors ${\bf g}_{i_1}\wedge\cdots\wedge {\bf g}_{i_p}$
($1\leq i_1<\cdots <i_p\leq n$) are $\OQq$-linearly independent.
Let $j\in\{ 1\kdots N\}$ and let $i_1\kdots i_p$
be the indices from $\{ 1\kdots n\}$ such that
$i_1<\cdots <i_p$ and $\nu_j=\lambda_{i_1}\cdots\lambda_{i_p}$.
Then by Lemma \ref{le:17.1}, 
\begin{equation}\label{9.0}
H_{\widehat{\mathcal{L}},\widehat{\cc},Q}({\bf g}_{i_1}\wedge\cdots\wedge {\bf g}_{i_p})
\leq p^{p/2}(1+\varepsilon )^p\nu_j .
\end{equation}
So $\widehat{\lambda}_j\leq p^{p/2}(1+\varepsilon )^p\nu_j$.
This holds for every $\varepsilon >0$, hence 
\begin{equation}\label{17.3}
\frac{\widehat{\lambda}_j}{\nu_j}\leq p^{p/2}\ \ \mbox{for } j=1\kdots N.
\end{equation}

Put
\[
\widehat{\alpha}:=\sum_{v\in M_K}\sum_{j=1}^N \widehat{c}_{jv}.
\]
Notice that $\widehat{\alpha}=N'\alpha$,
where $\alpha :=\sum_{v\in M_K}\sum_{j=1}^n c_{iv}$, $N' := \binom{n-1}{p-1}$.
Also, by \eqref{6.4b},
$\Delta_{\widehat{\mathcal{L}}}=\DL^{N'}$.
These facts together with Proposition \ref{pr:9.1} imply
\[
\nu_1\cdots\nu_N\leq 2^{n(n-1)N'/2}\Delta_{\widehat{\mathcal{L}}}Q^{-\widehat{\alpha}}.
\]  
On the other hand, Proposition \ref{pr:9.1} 
applied to $\widehat{\mathcal{L}},\widehat{\cc}$ gives
\[
\widehat{\lambda}_1\cdots\widehat{\lambda}_N\geq 
N^{-N/2}\Delta_{\widehat{\mathcal{L}}}Q^{-\widehat{\alpha}},
\]
and so
\[
\prod_{j=1}^N \frac{\widehat{\lambda}_j}{\nu_j}\,
\geq N^{-N/2} 2^{-n(n-1){N'}/2}.
\]
Now our lemma follows by combining this with \eqref{17.3}.
\end{proof}

Let $T_i(Q)$ ($i=1\kdots n$) be the spaces defined by \eqref{16.spaces}. 
Further, define the linear subspaces of $\OQq^N$,
\[ 
\widehat{T}_j(Q):=
\bigcap_{\lambda>\widehat{\lambda}_j(Q)}\span\{ \widehat{\x}\in\OQq^N:\, H_{\widehat{\mathcal{L}},\widehat{\cc},Q}(\widehat{\x})\leq\lambda\}\ \ (j=1\kdots N).
\]

\begin{lemma}\label{le:17.3}
Put $k:=n-p$. Let $Q\geq 1$ and suppose that
\begin{equation}\label{17.4}
\lambda_{k+1}(Q)> 2^{2n^3 2^n}\lambda_k(Q).
\end{equation}
Then
\begin{eqnarray}
\label{17.4a}
&&\frac{\widehat{\lambda}_{N-1}(Q)}{\widehat{\lambda}_N(Q)}
\leq 2^{n^3 2^n}\frac{\lambda_k(Q)}{\lambda_{k+1}(Q)}<2^{-n^32^n},
\\
\label{17.4b}
&&H_2(\widehat{T}_{N-1}(Q))=H_2(T_k(Q)).
\end{eqnarray}
\end{lemma}

\begin{proof}
Write again $\lambda_i,\widehat{\lambda}_j,\nu_j$ for
$\lambda_i(Q),\widehat{\lambda}_j(Q),\nu_j(Q)$. 
Since
\[
\nu_{N-1}=\lambda_k\lambda_{k+2}\cdots\lambda_N,\ \ \nu_N=\lambda_{k+1}\cdots\lambda_N
\]
we have $\nu_{N-1}/\nu_N=\lambda_k/\lambda_{k+1}$.
Together with Lemma \ref{le:17.2}, $N=\binom{n}{p}\leq 2^n$ 
and assumption \eqref{17.4} this implies \eqref{17.4a}.

As for \eqref{17.4b}, let $\varepsilon >0$.
Put $T:= T_k(Q)$, $\widehat{T}:=\widehat{T}_{N-1}(Q)$. 
Choose $\OQq$-linearly independent vectors ${\bf g}_1\kdots {\bf g}_n$
such that $\HLcQ({\bf g}_i)\leq (1+\varepsilon )\lambda_i$ for $i=1\kdots n$.
Write $\widehat{{\bf g}}_j:=
{\bf g}_{i_1}\wedge\cdots\wedge {\bf g}_{i_p}$ where $I_j=\{ i_1<\cdots <i_p\}$
is the $j$-th set in $C(n,p)$.
Then by \eqref{9.0}, 
\[
H_{\widehat{\mathcal{L}},\widehat{\cc},Q}(\widehat{{\bf g}}_j)\leq 
p^{p/2}(1+\varepsilon )^p\nu_{N-1}\ 
\mbox{for $j=1\kdots N-1$.}
\]
Assuming $\varepsilon$ is sufficiently small,
$\{ {\bf g}_1\kdots {\bf g}_k\}$ is a basis of $T$.
Moreover, by Lemma \ref{le:17.2} and \eqref{17.4a} we have
$p^{p/2}(1+\varepsilon )^p\nu_{N-1}<\widehat{\lambda}_N$.
Hence by \eqref{9.0},
$\{\widehat{{\bf g}}_1\kdots\widehat{{\bf g}}_{N-1}\}$ 
is a basis of $\widehat{T}$.
Now $H_2(\widehat{T})=H_2(T)$ follows from
Lemma \ref{le:6.1}.
\end{proof} 
 
We now make a first step towards estimating the heights of the subspaces
in the filtration of $(\mathcal{L} ,\cc )$. As usual, 
$n$ is an integer $\geq 2$,
$K$ an algebraic number field,
and $(\mathcal{L} ,\cc )$ a pair satisfying \eqref{2.10a}--\eqref{2.7a}. 
Put
\[
H_2:=\max\{ H_2(L_i^{(v)}):\, v\in M_K,\, i=1\kdots n\}.
\]

\begin{lemma}\label{le:17.4}
Assume that the subspace $T_{r-1}$ preceding $\OQq^n$ in the filtration
of $(\mathcal{L} ,\cc )$ has dimension $n-1$. Then
\[
H_2(T_{r-1})\leq H_2^{(n-1)^2}.
\]
\end{lemma}

\begin{proof}
We assume without loss of generality that $c_{1v}\leq\cdots\leq c_{nv}$
for $v\in M_K$.  
Put $T:=T_{r-1}$.
By our choice of $T$, if $T'$ is any other $(n-1)$-dimensional 
linear subspace of $\OQq^n$,
then $\mu (\OQq^n,T)<\mu (\OQq^n,T')$, implying $w(T')<w(T)$.

Take $v\in M_K$. Let $i(v)$ be the smallest index $i$ such that
\[
U_{iv}:=\{ \x\in\OQq^n:\, L_1^{(v)}(\x )=\cdots =L_i^{(v)}(\x )=0\}\subseteq T.
\]
$T$ is given by an up to a constant factor unique linear equation,
which we may express as $\sum_{j=1}^n \alpha_{jv}L_j^{(v)}(\x )=0$
where not all $\alpha_{jv}$ are $0$. In fact, $T$ is given by
$\sum_{j=1}^{i(v)}\alpha_{jv}L_j^{(v)}(\x )=0$, where $\alpha_{i(v),v}\not= 0$. 
It follows that
$i(v)$ is the largest index $i$ such that
$\{ L_i^{(v)}|_T:\, j\in\{ 1\kdots n\}\setminus\{ i\}\}$
is linearly independent.  
Hence
\begin{equation}\label{17.6}
w(T)=\sum_{v\in M_K} w_v(T)=
\sum_{v\in M_K}\sum_{\stackrel{j=1}{j\not= i(v)}}^n c_{jv}.
\end{equation}
Moreover,
\begin{equation}\label{17.7}
\sum_{v\in M_K} U_{i(v),v}\subseteq T.
\end{equation}
We prove that in \eqref{17.7} we have equality. Assume the contrary.
Then there is an $(n-1)$-dimensional linear subspace $T'\not= T$ of $\OQq^n$
such that $\sum_{v\in M_K} U_{i(v),v}\subset T'$. Then if $j(v)$ denotes
the smallest index $i$ such that $U_{iv}\subseteq T'$ we have $j(v)\leq i(v)$
for $v\in M_K$.
So 
\[
w(T')= \sum_{v\in M_K}\sum_{\stackrel{j=1}{j\not= j(v)}}^n c_{jv}\geq w(T),
\]
contrary to what we observed above. 
  
Knowing that we have equality in \eqref{17.7}, 
there is a subset $\{ v_1\kdots v_s\}$
of $M_K$ with $s\leq n-1$ such that $T=U_{i(v_1),v_1}+\cdots +U_{i(v_s),v_s}$.
By \eqref{6.10}, \eqref{6.8} we have
\[
H_2(U_{i(v_l),v_l})=H_2(U_{i(v_l),v_l}^{\bot})\leq H_2^{n-1}\ \ 
\mbox{for } l=1\kdots s,
\]
and then by \eqref{6.9},
\[
H_2(T)\leq\prod_{l=1}^s H_2(U_{i(v_l),v_l})\leq H_2^{(n-1)^2}.
\]
This completes our proof.
\end{proof}

Our final result is as follows.

\begin{prop}\label{pr:17.5}
Let $T_1\kdots T_{r-1}$ be the subspaces of $\OQq^n$ in the filtration
of $(\mathcal{L} ,\cc )$. 
Put $H_2:=\max\{ H_2(L_i^{(v)}):\, v\in M_K,\, i=1\kdots n\}$. 
Then
\[
H_2(T_i)\leq H_2^{4^n}\ \ \mbox{for } i=1\kdots r-1.
\]
\end{prop}

\begin{proof}
Let $i\in\{ 1\kdots r-1\}$ and put $T:=T_i$, $k:=\dim T$, $p:=n-k$, $N:=\binom{n}{p}$.
Further, let $\widehat{\mathcal{L}}$, $\widehat{\cc}$ be as in
\eqref{17.1}, \eqref{17.2}. By \eqref{6.5}, for the linear forms $\widehat{L}_j^{(v)}$ in 
$\widehat{\mathcal{L}}$ we have
\begin{equation}\label{17.8}
H_2(\widehat{L}_j^{(v)})\leq H_2^p\ \ \mbox{for } v\in M_K,\, j=1\kdots N.
\end{equation}

Let $0<\theta <\mu (T_{i+1},T_i)-\mu (T_{i+2},T_{i+1})$.
By Theorem \ref{th:16.2} we have for every sufficiently large $Q$, that
\begin{equation}\label{17.9}
T_k(Q)=T
\end{equation}
and $\lambda_k(Q)/\lambda_{k+1}(Q)\leq Q^{-\theta}$. Together with
Lemma \ref{le:17.3} (i), this implies that for $Q$ sufficiently large we have
$\widehat{\lambda}_{N-1}(Q)/\widehat{\lambda}_N(Q)\leq Q^{-\theta /2}$,
with a positive exponent $\theta /2$ independent of $Q$,
and so $\dim \widehat{T}_{N-1}(Q)=N-1$.
Again from Theorem \ref{th:16.2}, 
but now applied with $\widehat{\mathcal{L}}, \widehat{\cc},N$
instead of $\mathcal{L} ,\cc ,n$, it follows that 
there is a subspace $\widehat{T}$ of dimension $N-1$
in the filtration of $(\widehat{\mathcal{L}} ,\widehat{\cc})$, such that
\[
\widehat{T}_{N-1}(Q)=\widehat{T}
\]
for every sufficiently large $Q$.

Now using subsequently \eqref{17.9}, Lemma \ref{le:17.3} (ii), Lemma \ref{le:17.4} (with 
$\widehat{\mathcal{L}},\widehat{\cc},N$ instead of $\mathcal{L} ,\cc ,n$),
and \eqref{17.8}, we obtain for $Q$ sufficiently large,
\[
H_2(T)=H_2(T_k(Q))=H_2(\widehat{T}_{N-1}(Q))=H_2(\widehat{T})\leq (H_2^p)^{(N-1)^2}
\leq H_2^{4^n}
\]
where in the last step we have used $p(N-1)^2\leq p\binom{n}{p}^2\leq 4^n$.
This completes our proof.
\end{proof}

\section{Proof of Theorem \ref{th:2.3}}\label{18}

Let $n,\LL ,\cc ,\delta ,R$ satisfy \eqref{2.10a}--\eqref{2.10}.
Let $T=T(\LL ,\cc )$ be the subspace from \eqref{2.scss}.
Recall that this space is defined over $K$.
The hard core of our proof is to make explicit Lemma \ref{le:16.0} (ii).

Put $k:=\dim T$. 
Choose a basis $\{ {\bf g}_1\kdots {\bf g}_k\}$ of $T$, contained in $K^n$.
Write in the usual manner 
$\bigcup_{v\in M_K}\{ L_1^{(v)}\kdots L_n^{(v)}\}=\{ L_1\kdots L_r\}$, where $r\leq R$, 
and let $\theta_1\kdots\theta_u$ be the distinct, non-zero numbers among
\begin{equation}\label{18.1}
\left(\det (L_{i_l}({\bf g}_j)\right)_{l,j=1\kdots k},\ \ 
1\leq i_1<\cdots <i_k\leq r.
\end{equation}
For $v\in M_K$, put
\[
M_v:=\max (\|\theta_1\|_v\kdots \|\theta_u\|_v),\ \ 
m_v:=\min (\|\theta_1\|_v\kdots \|\theta_u\|_v).
\]

\begin{lemma}\label{le:18.1} 
We have
\[
\prod_{v\in M_K} \frac{M_v}{m_v} \leq (2\HL )^{(4R)^n}.
\]
\end{lemma}

\begin{proof}
Let $\varphi$ be a linear transformation of $\OQq^n$, defined over $K$.
By Lemma \ref{le:9.2}, replacing $\LL$ by $\LL\circ\varphi$ has the effect
that $T=T(\LL ,\cc )$ is replaced by $\varphi^{-1}(T)$.
Taking the basis $\varphi^{-1}({\bf g}_1)\kdots\varphi^{-1}({\bf g}_k)$ of $\varphi^{-1}(T)$,
we see that the quotients $M_v/m_v$ ($v\in M_K$) remain unchanged.
This shows that to 
prove our lemma, we may replace $\LL$ by $\LL\circ\varphi$.
Now choose linearly independent $L_1\kdots L_n$ from $\LL$, and then $\varphi$
such that $L_i\circ\varphi =X_i$ for $i=1\kdots n$. Then $L\circ\varphi$
contains $X_1\kdots X_n$.

So we may assume without loss of generality that 
$\LL$ contains $X_1\kdots X_n$
and then apply Lemma \ref{le:10.3a}.  
Thus, we conclude that
\begin{equation}\label{18.2}
\prod_{v\in M_K} \frac{M_v}{m_v} 
\leq \left(\binom{n}{k}^{1/2}\HL\cdot H_2(T)\right)^{\binom{r}{k}}.
\end{equation}
We estimate $H_2(T)$ from above by means of Proposition \ref{pr:17.5}.
The coefficients of $L_1\kdots L_r$ belong to the set $\{ d_1\kdots d_m\}$
from Lemma \ref{le:9.5}. Hence
\[
H_2(L_i)\leq n^{1/2}\prod_{v\in M_K}\max (\|d_1\|_v\kdots \|d_m\|_v)\leq n^{1/2}\HL
\]
for $i=1\kdots r$, and so $H_2(T)\leq (n^{1/2}\HL)^{4^n}$.
By inserting this inequality together with $\binom{r}{k}\leq R^n/n!$
into \eqref{18.2}, we infer
\[
\prod_{v\in M_K} \frac{M_v}{m_v} \leq
\left(\binom{n}{k}^{1/2}n^{4^n/2}\cdot \HL^{4^n+1}\right)^{R^n/n!}
\leq (2\HL )^{(4R)^n}.
\]
\end{proof}

In addition to \eqref{2.10a}--\eqref{2.10}, we assume that
\begin{equation}
\label{18.3}
c_{1v}\leq\cdots \leq c_{nv}\ \ \mbox{for } v\in M_K
\end{equation}
which is no restriction.

By \eqref{15.6} we have
\[ 
w(T)=w_{\mathcal{L},\cc}(T)=\sum_{v\in M_K}\sum_{i\in I_v} c_{iv},
\]
where $I_v=I_v(T)=\{ i_1(v)\kdots i_k(v)\}$ is the set defined by \eqref{15.1}.
Put $I_v^c := \{ 1\kdots n\}\setminus T$.

Let
\begin{equation}\label{18.3a}
\widetilde{L}_i^{(v)}:=L_i^{(v)}-
\sum_{\stackrel{j\in I_v}{j<i}}\alpha_{ijv}L_j^{(v)}
\ \ (v\in M_K,\, i\in I_v^c)
\end{equation}
be the linear forms from \eqref{15.7}. Recall that these linear forms
vanish identically on $T$.
For $v\in M_K,\, i\in I_v$, 
put $\widetilde{L}_i^{(v)}:=L_i^{(v)}$, and define the system
\[
\widetilde{\mathcal{L}}:=(\widetilde{L}_i^{(v)}:\, v\in M_K,\, i=1\kdots n).
\]
Clearly, for every $v\in M_K$,
the set $\{\widetilde{L}_i^{(v)}:\, i=1\kdots n\}$ is linearly
independent.

\begin{lemma}\label{le:18.2} 
The system $\widetilde{\LL}$ has the following properties:
\begin{eqnarray}
\label{18.4}
&H_{\widetilde{\mathcal{L}},\cc, Q}(\x )\leq (2\HL )^{(8R)^n}\HLcQ(\x )\ \ 
\mbox{for $\x\in\OQq^n$, $Q\geq 1$;}&
\\[0.1cm]
\label{18.7}
&H_{\widetilde{\mathcal{L}}}\leq (n\HL )^{(8R)^n}.&
\end{eqnarray}
\end{lemma} 

\begin{proof}
Let $v\in M_K$. 
We find expressions for the coefficients $\alpha_{ijv}$ from the relations
\[
L_i^{(v)}({\bf g}_h)=\sum_{j\in I_v} \alpha_{ijv}L_j^{(v)}({\bf g}_h)\ \ 
\mbox{for $i\in I_v^c$, $h=1\kdots k$}
\]
and Cramer's rule. Recall
that $\alpha_{ijv}=0$ for $j>i$ by the definition of $I_v$.
In fact, each $\alpha_{ijv}$ is of the shape $\delta_{ijv}/\delta_v$,
where 
$\delta_v=\det\left( (L_{i_l(v)}^{(v)}({\bf g}_h))_{l,h=1\kdots k}\right)$,
and $\delta_{ijv}$ is a similar sort of determinant,
but with $L_j^{(v)}$ replaced by $L_i^{(v)}$.
Clearly, 
$\delta_v$ and the numbers $\delta_{ijv}$ all occur among
the numbers \eqref{18.1}. Hence
\begin{equation}\label{18.8}
\|\alpha_{ijv}\|_{v'}\leq \frac{M_{v{'}}}{m_{v{'}}}\ \ 
\mbox{for $i\in I_v^c,\,\, j\in I_v,\,\, v'\in M_K$.}
\end{equation} 

We now prove \eqref{18.4}.
Let $\x\in\OQq^n$, $Q\geq 1$, and choose a finite extension $E$ of $K$ 
such that $\x\in E^n$. For $w\in M_E$ lying above $v\in M_K$,
define in the usual manner
$c_{iw}, L_i^{(w)}$ by \eqref{2.5} and similarly,
$\widetilde{L}_i^{(w)}:=\widetilde{L}_i^{(v)}$, $\alpha_{ijw}:=\alpha_{ijv}$,
$I_w:=I_v$, $M_w:= M_v^{d(w|v)}$, $m_w:= m_v^{d(w|v)}$.
Thus, \eqref{18.3a}, \eqref{18.8} and Lemma \ref{le:18.1} hold with
$w\in M_E$ instead of $v\in M_K$. It follows that for $w\in M_E$
we have
\[
\max_{1\leq i\leq n} \|\widetilde{L}_i^{(w)}(\x )\|_wQ^{-c_{iw}}
\leq
 n^{s(w)}\frac{M_w}{m_w}\cdot
\max_{1\leq i\leq n} \| L_i^{(w)}(\x )\|_wQ^{-c_{iw}}.
\]
By taking the product over $w\in M_E$ it follows
\[
H_{\widetilde{\mathcal{L}},\cc ,Q}(\x )\leq n(2\HL)^{(4R)^n}\HLcQ (\x ),
\]
which implies \eqref{18.4}.

We next prove \eqref{18.7}. Let 
$d_1\kdots d_t$ be the determinants of the $n$-element subsets
of $\bigcup_{v\in M_K}\{ L_1^{(v)}\kdots L_n^{(v)}\}$,
and 
$\widetilde{d}_1\kdots\widetilde{d}_s$
the determinants of the $n$-element subsets of 
$\bigcup_{v\in M_K}\{ \widetilde{L}_1^{(v)}\kdots \widetilde{L}_n^{(v)}\}$.
Then each $\widetilde{d}_i$ is a linear combination 
of elements from $d_1\kdots d_t$ 
with at most $n^n$ terms, each coefficient of which is a product of 
at most $n$ elements
from $\alpha_{ijv}$ ($v\in M_K,\, i\in I_v,\, j\in I_v^c$).
So by \eqref{18.8},
\[
\max_{1\leq i\leq s} \|\widetilde{d}_i\|_v\leq
n^{ns(v)}\left(\frac{M_v}{m_v}\right)^n\cdot\max_{1\leq i\leq t}\| d_i\|_v
\]
for $v\in M_K$. By taking the product over $v\in M_K$ and using 
Lemma \ref{le:18.1}, we obtain
\[
H_{\widetilde{\mathcal{L}}}\leq n^n (2\HL)^{n(4R)^n}\cdot \HL\leq (2\HL)^{(8R)^n},
\] 
which is \eqref{18.7}.
\end{proof}  

In the proof of Theorem \ref{th:2.3} we assume
\begin{eqnarray}
\label{18.9}
&
\begin{array}{l}
\mbox{there is a non-archimedean place $v_0\in M_K$ such that}
\\
\mbox{$c_{i,v_0}=0$, $\widetilde{L}_i^{(v_0)}=X_i$ for $i=1\kdots n$,}
\end{array}
&
\\[0.1cm]
\label{18.10}
&T=\{ \x\in\OQq^n:\, x_1=\cdots =x_{n-k}=0\}.&
\end{eqnarray}
We show that these are no restrictions.
Let $\varphi$ be a linear transformation of $\OQq^n$, defined over $K$.
Lemma \ref{le:9.2} says that $T(\LL\circ\varphi ,\cc )=\varphi^{-1}(T)$. 
Hence, if we construct a system of linear forms from $\LL\circ\varphi$
and $T(\LL\circ\varphi ,\cc )$ in the same way as $\widetilde{\LL}$ has been
constructed from $\LL$ and $T$, we obtain $\widetilde{\LL}\circ\varphi$. 
Now choose $\varphi$ such that 
$\{ \widetilde{L}_1^{(v)}\circ\varphi\kdots \widetilde{L}_n^{(v)}\circ\varphi\}=
\{ X_1\kdots X_n\}$, and moreover,
$\{ \widetilde{L}_i^{(v)}\circ\varphi :\, i\in I_v^c\} =\{ X_1\kdots X_{n-k}\}$.
Then $\widetilde{\LL}\circ\varphi$ contains $X_1\kdots X_n$,
and $T(\LL\circ\varphi ,\cc )$ is given by $X_1=\cdots =X_{n-k}=0$.
Now Lemma \ref{le:9.2} implies that in the proof of Theorem \ref{th:2.3},
we may replace $\LL$ by $\LL\circ\varphi$.

So henceforth, in addition to \eqref{2.10a}--\eqref{2.10} and \eqref{18.3},
we assume \eqref{18.9}, \eqref{18.10}.

The projection 
\begin{equation}\label{18.101}
\varphi '' :\, (x_1\kdots x_n)\mapsto (x_1\kdots x_{n-k})
\end{equation}
has kernel $T$.
We now define a tuple in $K[X_1\kdots X_{n-k}]^{\lin}$,
\begin{eqnarray}
\label{18.102}
&&\mathcal{L}{''}=(L_i^{(v)}{''}:\, v\in M_K,\, i\in I_v^c)
\\
\nonumber
&&\mbox{with }
L_i^{(v)}{''}:= \widetilde{L}_i^{(v)}\circ\varphi{''}^{-1}\ \, (v\in M_K,\, i\in I_v^c)
\end{eqnarray}
and a tuple of reals
\begin{eqnarray}
\label{18.103}
&&{\bf d}=(d_{iv}:\, v\in M_K,\, i\in I_v^c)
\\
\nonumber
&&\mbox{with } d_{iv}:= \frac{n-k}{n}\left( c_{iv}-\theta_v\right),
\ \, (v\in M_K,\,\, i\in I_v^c),
\\
\nonumber
&&\mbox{where } \theta_v:=\frac{1}{n-k}\Big(\sum_{j\in I_v^c} c_{jv}\Big)\ \,
(v\in M_K).
\end{eqnarray}
Notice that by 
Lemma \ref{le:15.2} (ii) and assumption \eqref{2.8} we have
\begin{equation}\label{18.14}
\sum_{v\in M_K} \theta_v=\frac{w(\OQq^n )-w(T)}{n-k}=-\,\frac{w(T)}{n-k}.
\end{equation} 
The tuple $\mathcal{L}''$ is precisely that defined in \eqref{15.7a},
while ${\bf d}$ is a normalization of the tuple $\cc ''$ from \eqref{15.7a}.
Eventually, we want to apply Theorem \ref{th:8.1} to $(\LL{''} , {\bf d})$,
and to this end we have to verify that this pair satisfies the analogues
of \eqref{8.5}--\eqref{8.8} with $\LL ,\cc$ replaced by $\LL '' ,{\bf d}$;
in fact, the tuple ${\bf d}$ has been chosen to satisfy 
\eqref{8.6}, \eqref{8.7}. Further, we need an estimate for $H_{\mathcal{L}{''}}$
in terms of $\HL$. Finally, we have to relate 
the twisted height $H_{\mathcal{L}{''},{\bf d},Q'}(\varphi '' (\x ))$
to $\HLcQ (\x )$, where $Q' := Q^{n/(n-k)}$.
 
We start with the verification of \eqref{8.5}--\eqref{8.8},
with $n-k,nR^n,\mathcal{\LL}{''},{\bf d}$ replacing $n,R,\LL ,\cc$, 
and with indices $i$ taken from $I_v^c$
instead of $\{ 1\kdots n\}$, for $v\in M_K$.
It is clear that ${\bf d}$ satisfies \eqref{8.5}, \eqref{8.6}, and that
$\LL ''$ satisfies \eqref{8.3}. Further, from \eqref{18.9}, \eqref{18.10}
it follows easily that $\LL ''$ satisfies \eqref{8.2}.
In the lemma below we show that $\LL '' ,{\bf d}$ has properties
\eqref{18.11}, \eqref{18.12}, \eqref{18.13}. 
which are precisely \eqref{8.7}, \eqref{8.4}, \eqref{8.8}
with $n-k,nR^n,\mathcal{\LL}{''},{\bf d}$ replacing $n,R,\LL ,\cc$.
The weight $w_{\mathcal{L}{''},{\bf d}}$ and twisted heights 
$H_{\mathcal{L}{''},{\bf d},Q}$ are defined similarly as in
Section \ref{16}, but with $d_{iv}$ in place of $c_{iv}$
in \eqref{16.102}, \eqref{16.twisted-height2}.

\begin{lemma}\label{le:18.3}
We have
\begin{eqnarray}
\label{18.11}
&\displaystyle{\sum_{v\in M_K}\max_{i\in I_v^c} d_{iv}\, \leq 1,}&
\\[0.1cm]
\label{18.12}
&\displaystyle{\#\left(\bigcup_{v\in M_K}\{ L_i^{(v)}{''}:\, i\in I_v^c\}
\right)\leq nR^n,}&
\\[0.1cm]
\label{18.13}
&w_{\mathcal{L}{''},{\bf d}}(U)\leq 0\ \ \mbox{for every linear subspace $U$
of $\OQq^{n-k}$.}
\end{eqnarray}
\end{lemma}

\begin{proof} We start with \eqref{18.11}.
Put $c_{iv}' :=c_{iv}-\frac{1}{n}\sum_{j=1}^n c_{jv}$
for $v\in M_K$, $i=1\kdots n$. Then $\sum_{i=1}^n c_{iv}'=0$ for $v\in M_K$,
while $\sum_{v\in M_K}\max_{1\leq i\leq n} c_{iv}'\leq 1$ by \eqref{2.9}.

Consequently,
\begin{eqnarray*}
\sum_{v\in M_K}\max_{i\in I_v^c} d_{iv}&=&
\frac{n-k}{n}\sum_{v\in M_K}\left( \max_{i\in I_v^c} c_{iv}'-\frac{1}{n-k}\sum_{j\in I_v^c} c_{jv}'\right)
\\
&=&\frac{n-k}{n}\cdot \sum_{v\in M_K}
\left(\max_{i\in I_v^c} c_{iv}'+\frac{1}{n-k}\sum_{j\in I_v} c_{jv}'\right)
\\
&\leq&\frac{n-k}{n}\cdot 
\left(1 +\frac{k}{n-k}\right)\max_{1\leq i\leq n} c_{iv}'
\leq 1. 
\end{eqnarray*}
This proves \eqref{18.11}.

Next, we prove \eqref{18.12}.
Let $v\in M_K$. The set
$\{L_i^{(v)}{''}:\, i\in I_v^c\}$ is determined by the linear forms
$\widetilde{L}_i^{(v)}$ given by \eqref{18.3a}, and the latter
by the ordered tuple $(L_1^{(v)}\kdots L_n^{(v)})$.
By \eqref{2.7} there are at most $R^n$ distinct tuples among these
as $v$ runs through $M_K$. This proves \eqref{18.12}.

We finish with proving \eqref{18.13}.
Take a linear subspace $U$ of $\OQq^{n-k}$ and let $W:=\varphi{''}^{-1}(U)$.
By \eqref{18.103}, \eqref{18.14}, we have
\begin{eqnarray*}
w_{\mathcal{L}{''},{\bf d}}(U)&=&
\frac{n-k}{n}\left(w_{\mathcal{L}{''},\cc{''}}(U)-\dim U\sum_{v\in M_K} \theta_v\right)
\\
&=&
\frac{n-k}{n}\left(w_{\mathcal{L}{''},\cc{''}}(U)+\,\dim U\cdot\frac{w(T)}{n-k}\right)
\end{eqnarray*}
and then by Lemma \ref{le:15.2} (ii),
\begin{eqnarray*}
w_{\mathcal{L}{''},{\bf d}}(U)&=&
\frac{n-k}{n}\left(w(W)-w(T)+\,\dim U\cdot\frac{w(T)}{n-k}\right)
\\
&=&
\frac{n-k}{n}
\left( w(W)\, -\frac{w(T)}{n-k}\cdot (n -\dim W)\right).
\end{eqnarray*}
Since this is $\leq 0$ by \eqref{2.scss}, this proves \eqref{18.13}.
\end{proof}

\begin{lemma}\label{le:18.4} 
We have
\[
H_{\mathcal{L}{''}}\leq (2\HL )^{(8R)^n}.
\]
\end{lemma}

\begin{proof}
Let $\widetilde{d}_1\kdots\widetilde{d}_s$ be the determinants of the $n$-element subsets
of\\
$\bigcup_{v\in M_K}\{ \widetilde{L}_1^{(v)}\kdots \widetilde{L}_n^{(v)}\}$$=:$
$\{ \widetilde{L}_1\kdots \widetilde{L}_r\}$,
and let $d_1{''}\kdots d_u{''}$ be the determinants of the $(n-k)$-element subsets
of $\bigcup_{v\in M_K}\{ L_i^{(v)}{''}:\, i\in I_v^c\}$.
Pick one of the determinants $d_i{''}$. Then for some 
$i_1\kdots i_{n-k}$, by \eqref{18.101}, \eqref{18.102}, 
\[
d_i{''}=\det (\widetilde{L}_{i_1}\circ\varphi{''}^{-1}\kdots\widetilde{L}_{i_{n-k}}\circ\varphi{''}^{-1})
=\det (\widetilde{L}_{i_1}\kdots\widetilde{L}_{i_{n-k}},X_{n-k+1}\kdots X_n)
\]
and then by \eqref{18.9}, $\pm d_i{''}\in\{ \widetilde{d}_1\kdots \widetilde{d}_s\}$.
Consequently, 
\[
H_{\mathcal{L}{''}}=\prod_{v\in M_K}\max_{1\leq i\leq u}\|d_i{''}\|_v
\leq\prod_{v\in M_K}\max_{1\leq i\leq s}\|\widetilde{d}_i\|_v= H_{\widetilde{\mathcal{L}}}.
\]
Together with \eqref{18.7} this implies our lemma.  
\end{proof}

\begin{prop}\label{pr:18.5}
Let $Q$ be a real with
\begin{equation}\label{18.15}
Q\geq (2\HL )^{200(8R)^n/\delta}
\end{equation}
and $\x\in\OQq^n$ with
\begin{equation}\label{18.16}
\HLcQ (\x )\leq\DL^{1/n}Q^{-\delta}.
\end{equation}
Put $Q' := Q^{n/(n-k)}$. Then
\begin{equation}\label{18.17}
H_{\mathcal{L}{''},{\bf d},Q'}(\varphi '' (\x ))
\leq Q{'}^{-\frac{99}{100}\delta /n}.
\end{equation}
\end{prop}

\begin{proof} We need the crucial observation that by \eqref{18.14}, \eqref{2.scss}, \eqref{2.8},
\begin{equation}\label{18.observation}
\sum_{v\in M_K} \theta_v= -\, \frac{w(T)}{n-k}<-\,\frac{w(\OQq^n)}{n}=0.
\end{equation}

Let $E$ be a finite extension of $K$ with $\x\in E^n$.
In accordance with our usual conventions, we put 
$L_i^{(w)}{''} :=L_i^{(v)}{''}$,
$d_{iw}:=d(w|v)d_{iv}$, $I_w^c:=I_v^c$ for places $w\in M_E$ 
lying above $v\in M_K$. 
Thus, \eqref{18.103}, \eqref{18.14}, \eqref{18.observation} imply 
$d_{iw}:=\frac{n-k}{n}({c_{iw}-\theta_w})$ for $w\in M_E, i\in I_w^c$
with $\sum_{w\in M_E}\theta_w<0$,
and so
\begin{eqnarray*}
H_{\mathcal{L}{''},{\bf d},Q'}(\varphi '' (\x ))&=&
\prod_{w\in M_E} 
\max_{i\in I_w^c}\| \widetilde{L}_i^{(w)}(\x )\|_w{Q'}^{-d_{iw}}
\\
&=&
\prod_{w\in M_E} Q^{\theta_w}
\max_{i\in I_w^c} \|\widetilde{L}_i^{(w)}(\x )\|_wQ^{-c_{iw}}
\\
&\leq&
\prod_{w\in M_E}\max_{1\leq i\leq n} \|\widetilde{L}_i^{(w)}(\x )\|_wQ^{-c_{iw}}
\\
&=&H_{\widetilde{\mathcal{L}},\cc ,Q}(\x ).
\end{eqnarray*}
Together with \eqref{18.4}, \eqref{9.3}, \eqref{18.observation} this implies
\begin{eqnarray*}
H_{\mathcal{L}{''},{\bf d},Q{'}}(\varphi{''} (\x ))
&\leq& (2\HL )^{(8R)^n}\HLcQ (\x )
\\
&\leq& (2\HL )^{(8R)^n+R^n}\cdot \DL^{-1/n}\HLcQ (\x ).
\end{eqnarray*}
Now \eqref{18.17} follows easily from this last inequality and 
\eqref{18.15}, \eqref{18.16}.
\end{proof}  
\vskip0.2cm

\begin{proof}[Proof of Theorem \ref{2.3}]
We assume for the moment, that $n-k\geq 2$.
We intend to apply Theorem \ref{th:8.1}
with 
\begin{equation}\label{18.substitutions}
n-k,\ nR^n,\ \frac{99}{100}\delta/n,\ \LL{''},\ {\bf d}
\end{equation}
replacing
$n,R,\delta, \LL ,\cc$, respectively.
Clearly, with these replacements
\eqref{8.1} holds, and we verified above that
conditions \eqref{8.5}--\eqref{8.8}
are satisfied as well.

Let $m_2'$, $\omega_2'$ be the quantities $m_2,\omega_2$ from
Theorem \ref{th:8.1}, with the objects in \eqref{18.substitutions}
replacing $n,R,\delta, \LL ,\cc$, respectively.
Further, let $C_2'$ be the quantity obtained by applying the substitutions
from \eqref{18.substitutions} to $C_2$, but replacing $H_{\mathcal{L}{''}}$
by the upper bound $(2\HL)^{(8R)^n}$ from Lemma \ref{le:18.4}.  
Then Theorem \ref{th:8.1} implies that 
there exist reals $Q'_1\kdots Q'_{m_2{'}}$ with 
$C_2'\leq Q_1'<\cdots <Q'_{m_2{'}}$
such that if $Q'\geq 1$ is a real with
\begin{equation}\label{18.18}
\{ {\bf y}\in\OQq^{n-k}:\, H_{\mathcal{L}{''},{\bf d},Q{'}}({\bf y})
\leq {Q'}^{-\frac{99}{100}\delta /n}\}\, \not=\{ {\bf 0}\},
\end{equation}
then
\begin{equation}\label{18.19}
Q'\in \left[\left. 1,C_2'\right)\right.\cup\,
\bigcup_{h=1}^{m_2{'}} 
\left[\left. Q_h' ,{Q_h'}^{\omega_2{'}}\right)\right. .
\end{equation}

We proved \eqref{18.19} under the assumption $n-k\geq 2$.
We now assume that $n-k=1$ 
and show that \eqref{18.19} is valid also in this case.
The quantities $m_2' ,\omega_2' ,C_2'$ are defined
as above, but with $n-k=1$ replacing $n$. We have
$L_1^{(v)}{''}=\alpha_vX$, $d_{1v}=0$ for $v\in M_K$, 
and so for ${\bf y}=y\in K^*$, by the product formula,
\[
H_{\mathcal{L}{''},{\bf d},Q'}(y)=\prod_{v\in M_K} \|\alpha_vy\|_v
=\prod_{v\in M_K}\|\alpha_v\|_v.
\]
This is valid also if $y\not\in K$. 
Let $\{ \alpha_v:\, v\in M_K\}=\{ \alpha_1\kdots \alpha_r\}$.
By \eqref{18.12}, we have $r\leq nR^n$. Moreover, by Lemma \ref{le:18.4},
\[
\prod_{v\in M_K}\max_{1\leq i\leq r}\|\alpha_i\|_v=H_{\mathcal{L}{''}}
\leq (2\HL )^{(8R)^n}.
\]
Hence if $y\not= 0$,
\begin{eqnarray*}
H_{\mathcal{L}{''},{\bf d},Q'}(y)&\geq& 
\displaystyle{\prod_{v\in M_K}\min_{1\leq i\leq r}\|\alpha_i\|_v}
\\
&\geq& \prod_{v\in M_K}
\frac{\|\alpha_1\cdots\alpha_r\|_v}
{(\max_{1\leq i\leq r}\|\alpha_i\|_v)^{r-1}}\geq (2\HL )^{-(8R)^{2n}}.
\end{eqnarray*}
Now if $y$ satisfies \eqref{18.18}, then certainly, $Q'\leq C_2'$
and so \eqref{18.19} is satisfied.

Let $Q$ be one of the reals 
being considered in Theorem \ref{th:2.3}, i.e., with
\[
\{ \x\in\OQq^n :\, \HLcQ (\x )\leq\DL^{1/n}Q^{-\delta}\}\not\subset T.
\]
Then by Proposition \ref{pr:18.5}, either $Q$ does not satisfy \eqref{18.15},
or $Q':= Q^{n/(n-k)}$ satisfies \eqref{18.18}. 
The first alternative implies $Q< C_2{'}^{n/(n-k)}$. So in either case,
\[
Q\in \left[\left. 1,{C_2'}^{(n-k)/n}\right)\right.\cup\, 
\bigcup_{h=1}^{m_2{'}} 
\left[\left. Q_h^* ,{Q_h^*}^{\omega_2{'}}\right)\right. ,
\]
where $Q_h^* := {Q_h'}^{(n-k)/n}$ for $h=1\kdots m_2{'}$.

To prove Theorem \ref{th:2.3}, we have to cut the intervals into smaller pieces.
In general, any interval $[A,A^{\theta})$ is contained in a union of at most
$[\log \theta /\log \omega_0 ]+1$ intervals of the shape $[Q^* ,{Q^*}^{\omega_0})$.
It follows that there are reals $Q_1\kdots Q_m$, with $C_0\leq Q_1<\cdots <Q_m$,
such that
\[
Q\in\left[\left. 1,C_0\right)\right.\cup\,
\bigcup_{h=1}^m \left[\left. Q_h ,Q_h^{\omega_0}\right)\right. ,  
\]
where
\[
m:= 1+\left[\frac{\log(\log C_2{'}^{(n-k)/n}/\log C_0 )}{\log\omega_0}\right]
\, +\, m_2'\left( 1+\left[\frac{\log \omega_2'}{\log\omega_0}\right]\right).
\]
To finish our proof, we have to show that $m\leq m_0$.

We first estimate from above $m_2'$. Taking the definition of $m_2$ from 
\eqref{8.definitions} and the substitutions from \eqref{18.substitutions},
and using $R\geq n\geq 2$, we obtain
\begin{eqnarray*}
m_2'&\leq& 61(n-k)^6 2^{2(n-k)}(100n/99\delta )^2
\log( 22(n-k)^2 2^{n-k}\cdot nR^n\cdot 100n/99\delta )
\\
&\leq& 62n^8 2^{2n}\delta^{-2}\log (23n^4 2^n R^n\delta^{-1})
\leq 62n^{10} 2^{2n}\delta^{-2}\log\big( (3\delta^{-1}R)^{3n}\big)
\\
&\leq& 186n^9 2^{2n}\delta^{-2}\log (3\delta^{-1}R)\, =: m_* .
\end{eqnarray*}
Further,
\begin{eqnarray*}
&&1+\left[\frac{\log(\log C_2{'}^{(n-k)/n}/\log C_0)}{\log\omega_0}\right]
\\
&&\qquad\leq
1+\left[
\frac{\log\Big(\log \big(2\times (2\HL)^{(8R)^n}\big)^{m_*^{2m_*}}/
\log\max (\HL^{1/R} ,n^{1/\delta})\Big)}
{\log\omega_0}\right]
\\
&&\qquad\leq
\frac{3m_*\log m_*}{\log (\delta^{-1}\log 3R)},
\end{eqnarray*}
and
\[
1+\left[\frac{\log \omega_2'}{\log\omega_0}\right]\leq 1+\frac{5}{2}
\cdot 
\frac{\log m_*}{\log \omega_0}\leq \frac{3\log m_*}{\log (\delta^{-1}\log 3R)}.
\]
So altogether,
\[
m\leq \frac{6m_*\log m_*}{\log (\delta^{-1}\log 3R )}.
\]
Using $R\geq n\geq 2$, $186n^9 2^{2n}\leq 50^{2n}$, 
$\delta^{-2}\log (3\delta^{-1}R)\leq (\delta^{-1}\log 3R)^3$, 
this leads to
\begin{eqnarray*}
m&\leq& 6m_*\times
\frac{\log \big(186 n^9 2^{2n}\delta^{-2}\log (3\delta^{-1}R)\big)}{\log (\delta^{-1}\log 3R )}
\\
&\leq&
6m_*\left(\frac{2n\log 50}{\log\log 6}\, + 3\right)
\leq 100n m_*
\\
&\leq&
10^5 2^{2n}n^{10}\delta^{-2}\log (3\delta^{-1}R) ,
\end{eqnarray*}
i.e., $m\leq m_0$. 
This completes the proof of Theorem \ref{th:2.3}.
\end{proof}

\end{document}